\documentclass[10pt]{amsart}

\topskip  \usepackage{amsmath,amsthm,cite}
\usepackage{amssymb,longtable,multido,pstricks}
\usepackage{mathdots}
\usepackage{color}
\usepackage{xspace}
\usepackage{float}

\newtheorem{theorem}{Theorem}
\newtheorem{lemma}[theorem]{Lemma}
\newtheorem{proposition}[theorem]{Proposition}
\newtheorem{corollary}[theorem]{Corollary}
\def\al{\alpha}
\def\be{\beta}
\def\ga{\gamma}

\def\gf{generating function\xspace}

\def\v{\vert}
\def\F{F}

\begin{document}
\title{Enumeration of permutations avoiding a triple of 4-letter patterns is all done}
\author[D. Callan]{David Callan}
\address{Department of Statistics, University of Wisconsin, Madison, WI 53706}
\email{callan@stat.wisc.edu}
\author[T.~Mansour]{Toufik Mansour}
\address{Department of Mathematics, University of Haifa, 31905 Haifa, Israel}
\email{tmansour@univ.haifa.ac.il}
\author[M.~Shattuck]{Mark Shattuck}
\address{Department of Mathematics, University of Tennessee, Knoxville, TN 37996}
\email{shattuck@math.utk.edu}

\begin{abstract}
\bigskip
This paper completes a project to enumerate permutations avoiding a triple $T$ of 4-letter patterns, in the sense of classical pattern avoidance, for every $T$.
There are 317 symmetry classes of such triples $T$ and previous papers have enumerated avoiders for all but 14 of them. One of these 14 is conjectured not to have an algebraic \gf. Here, we find the \gf for each of the remaining 13, and it is algebraic in each case.
\medskip

\end{abstract}
\maketitle

\section{Introduction}
This paper is the last in a series whose goal is to enumerate the permutations avoiding the patterns in $T$  for each of the $\binom{24}{3}$ triples $T$ of 4-letter patterns. There are 317 symmetry classes of these triples and 242 Wilf classes; hence, 242 distinct counting sequences. A Wilf class is said to be small or large depending on whether it contains one or more symmetry classes.  The symmetry classes in large Wilf classes were enumerated in \cite{CMS3patI,CMS3patII} (combined in \cite{HYL}).
The small Wilf classes that can be enumerated by the insertion encoding algorithm (INSENC) \cite{V} are listed in Table \ref{longinsenc} in the appendix.
The small Wilf classes not amenable to INSENC are listed in Table \ref{long4} below along with a reference to either a published paper or to a result in the present paper.
The numbering in both tables follows that of Table 2 in \cite{HYL}, where representative triples for all 317 symmetry classes are listed in lex order of counting sequence.
The \gf for Case 237, the only one not enumerated, is conjectured not to be differentially algebraic (see \cite{AA20160600} and \cite[Seq. A257562]{Sl}).

Our work extends earlier results concerning the enumeration of permutations avoiding one or two 4-letter patterns.  Permutations avoiding a single 4-letter pattern have been well studied (see, e.g., \cite{St0,St,W}), and there are 56 symmetry classes of pairs of 4-letter patterns, all but 8 of which have been enumerated.  Le \cite{L} established that these 56 symmetry classes form 38 distinct Wilf classes, and of these 38, 12 can be enumerated with regular insertion encodings (see \cite{V}).  Some of these generating functions were computed by hand by Kremer and Shiu \cite{KS}.  In \cite{SiS}, Simion and Schmidt enumerated permutations avoiding any subset of $S_3$, in particular, any subset of $S_3$ of size three.  Our results here then address the analogous problem on $S_4$ and complete the enumeration in the case of three patterns.

The organization of this paper is as follows.  In the next section, we recall some previous terminology and notation.  In the third section, we provide proofs of generating function formulas corresponding to the final thirteen symmetry classes for three patterns of length four.  In the appendix, the generating functions for all small Wilf classes that can be done using INSENC are listed.

{\footnotesize\begin{longtable}[c]{|l|l|l||l|l|l|}
\caption{Small Wilf classes not amenable to INSENC.\label{long4}}\\ \hline
\multicolumn{6}{| c |}{Start of Table}\\ \hline
No. &$T$ & Reference & No. &$T$ & Reference\\ \hline
\endfirsthead  \hline
\multicolumn{6}{|c|}{Continuation of Table \ref{long4}}\\ \hline
No. &$T$ & Reference & No. &$T$ & Reference \\ \hline
\endhead \hline
\endfoot \hline
\multicolumn{6}{| c |}{End of Table}\\ \hline\hline
\endlastfoot
15&$\{2134,3412,1243\}$&\cite{AA20160601}&29&$\{1324,2143,3421\}$&\cite{AA20160607}\\\hline
30&$\{4231,2143,1324\}$&\cite{AA20160607}&49&$\{2341,1324,4123\}$&\cite{AA20170501}\\\hline
69&$\{3412,1324,1234\}$&\cite{AA20170501}&72&$\{3412,1324,1243\}$&\cite{AA20170501}\\\hline
74&$\{3412,1243,1234\}$&Theorem~\ref{th74a}&75&$\{4231,1324,1243\}$&\cite{AA20170501}\\\hline
76&$\{3412,1324,4123\}$&\cite{AA20170501}&77&$\{3412,3124,1243\}$&\cite{AA1342}\\\hline
80&$\{4312,1324,4123\}$&\cite{AA20170501}&84&$\{4231,1324,4123\}$&\cite{AA20170501}\\\hline
86&$\{3412,4132,1324\}$&\cite{AA20170501}&88&$\{1324,3412,3421\}$&\cite{AA20170501}\\\hline
90&$\{1243,2431,3412\}$&\cite{AA1342}&93&$\{1324,2413,3421\}$&\cite{AA20170501}\\\hline
99&$\{4231,3142,1324\}$&\cite{AA20170501}&103&$\{2314,1342,4123\}$&\cite{AA1342}\\\hline
106&$\{1342,2143,3412\}$&\cite{AA1342}&109&$\{2143,3412,3421\}$&Theorem~\ref{th109a}\\\hline
118&$\{3412,1423,1234\}$&\cite{AA1342}&121&$\{3412,2341,1243\}$&Theorem~\ref{th121a}\\\hline
125&$\{2341,4123,1243\}$&Theorem~\ref{th125a}&130&$\{3412,3124,1342\}$&\cite{AA1342}\\\hline
131&$\{2134,1342,4123\}$&\cite{AA1342}&132&$\{1324,2341,2413\}$&\cite{AA20170501}\\\hline
133&$\{2143,2314,1342\}$&\cite{AA1342}&134&$\{2134,4123,1243\}$&\cite{AA20160601}\\\hline
149&$\{3412,4123,1234\}$&Theorem~\ref{th149a}&150&$\{4312,4132,1324\}$&\cite{AA20170501}\\\hline
151&$\{4312,1324,1423\}$&\cite{AA20170501}&153&$\{4231,1324,1423\}$&\cite{AA20170501}\\\hline
156&$\{1324,2341,2431\}$&\cite{AA20170501}&158&$\{1324,1342,3412\}$&\cite{AA20170501}\\\hline
159&$\{3412,1423,1243\}$&\cite{AA1342}&162&$\{3412,1342,4123\}$&\cite{AA1342}\\\hline
163&$\{3412,2314,1342\}$&\cite{AA1342}&164&$\{2341,4123,1423\}$&\cite{AA1342}\\\hline
165&$\{4312,3124,1423\}$&\cite{AA1342}&172&$\{2143,4132,1324\}$&\cite{AA20160607}\\\hline
175&$\{2413,1342,4123\}$&\cite{AA1342}&176&$\{1342,2431,3412\}$&\cite{AA1342}\\\hline
178&$\{2314,2431,1342\}$&\cite{AA1342}&180&$\{2431,4132,1324\}$&\cite{AA20170501}\\\hline
182&$\{3412,2314,2431\}$&\cite{AA1342}&184&$\{1324,2431,3241\}$&\cite{AA20170501}\\\hline
185&$\{2341,4123,1234\}$&Theorem~\ref{th185a}&187&$\{1324,2314,2431\}$&\cite{AA20170501}\\\hline
188&$\{2143,3214,1432\}$&Theorem~\ref{th188a}&190&$\{3142,2314,1423\}$&\cite{AA1342}\\\hline
192&$\{1243,1342,2431\}$&\cite{AA1342}&193&$\{1324,2431,3142\}$&\cite{AA20170501}\\\hline
194&$\{3124,4123,1243\}$&\cite{AA1342}&195&$\{1324,4123,1243\}$&\cite{AA20170501}\\\hline
197&$\{4312,3142,1423\}$&\cite{AA1342}&198&$\{1342,4123,1234\}$&\cite{AA1342}\\\hline
199&$\{1342,4123,1243\}$&\cite{AA1342}&204&$\{3124,1342,1243\}$&\cite{AA1342}\\\hline
207&$\{2134,1423,1243\}$&\cite{AA20160601}&208&$\{3124,1342,1234\}$&\cite{AA1342}\\\hline
209&$\{3142,1432,1243\}$&Theorem
\ref{th209a}&210&$\{4132,1324,1243\}$&\cite{AA20170501}\\\hline
211&$\{1324,4123,1234\}$&\cite{AA20170501}&212&$\{1324,2413,2431\}$&\cite{AA20170501}\\\hline
213&$\{2431,1324,1342\}$&\cite{AA20170501}&214&$\{1342,2341,3412\}$&\cite{AA1342}\\\hline
216&$\{2143,3412,3142\}$&Theorem~\ref{th216a}&217&$\{4132,1342,1243\}$&\cite{AA1342}\\\hline
219&$\{1342,2413,3412\}$&\cite{AA1342}&220&$\{2431,2314,3142\}$&\cite{AA1342}\\\hline
222&$\{3421,3412,1342\}$&\cite{AA1342}&223&$\{1243,1342,2413\}$&\cite{AA1342}\\\hline
224&$\{4132,1342,1423\}$&\cite{AA1342}&225&$\{2413,3142,1243\}$&Theorem
\ref{th225a}\\\hline
226&$\{2143,2413,1342\}$&\cite{AA1342}&227&$\{2143,1432,1324\}$&\cite{AA20160607}\\\hline
228&$\{2341,2413,3412\}$&Theorem~\ref{th228a}&230&$\{2341,1243,1234\}$&Theorem
\ref{th230a}\\\hline
231&$\{1324,4123,1423\}$&\cite{AA20170501}&232&$\{1234,1342,2341\}$&\cite{AA1342}\\\hline
237&$\{1432,1324,1243\}$&\cite{AA20160600}&240&$\{2341,3412,3421\}$&Theorem~\ref{th240a}\\\hline
241&$\{1324,1243,1234\}$&\cite{AA20170501}&242&$\{2341,2431,3241\}$&\cite{AA1342}\\\hline
\end{longtable}}

\section{Preliminaries}

For a pattern set $T$ under consideration, $F_T(x)$ denotes the generating function $\sum_{n\ge 0}\v S_n(T)\v x^n$ for $T$-avoiders and $G_m(x)$ the generating function for $T$-avoiders $\pi=i_1\pi^{(1)}i_2\pi^{(2)}\cdots i_m\pi^{(m)}\in S_n(T)$ with $m$ left-right maxima
$i_1,i_2,\dots,i_m=n$; thus $F_T(x)=\sum_{m\ge 0}G_m(x)$.
Note that $G_1(x)=xF_T(x)$ if no pattern in $T$ starts with 4, with $G_0(x)=1$.   For several of the triples $T$, our efforts are directed toward finding an expression for $G_m(x)$, usually distinguishing the case $m=2$, and sometimes also $m=3$, from larger values of $m$.  In cases such as these where we make use of the method of generating functions, we examine the structure of an avoider by splitting the class of avoiders into subclasses according to a judicious choice of parameters.  This choice is made so that each member of a subclass can be decomposed into independent parts.  The generating function (g.f) for a subclass (a summand in the full g.f.) is then the product of the g.f.'s for each of the individual parts, and we speak of the ``contribution'' of the various parts to the g.f. for the subclass.   By contrast, in cases 74, 125 and 185, prior to computing the g.f., we  determine an explicit formula that directly enumerates the class of avoiders in question.  To do so, we consider appropriate combinatorial parameters related to the class, which involve here the position of and/or the actual letters contained within the leftmost ascent.  The g.f. then follows from the explicit formula and a calculation, one that often involves multiple sums and is computer-assisted.   In the last three cases, we make use of (and modify in one case) the method of generating forests \cite{W} to determine a system of functional equations satisfied by the related g.f.'s, which we then solve by the \emph{kernel method} (see, e.g., \cite{HM}).

Throughout, $C(x)=\frac{1-\sqrt{1-4x}}{2x}$ denotes the g.f. for
the Catalan numbers $C_n:=\frac{1}{n+1}\binom{2n}{n}=\binom{2n}{n}-\binom{2n}{n-1}$.
As is well known \cite{K,wikipermpatt}, $C(x)$ is the g.f. for
$(|S_n(\pi)|)_{n\ge 0}$ where $\pi$ is any one of the six 3-letter
patterns. The identity $C(x)=\frac{1}{1-xC(x)}$ or, equivalently,
$xC(x)^2=C(x)-1$, is used to simplify some of the results.  Occasionally, we need the g.f. for avoiders of a 3-letter and a 4-letter pattern; see \cite{wikipermpatt2} for a comprehensive list.

\section{Proofs}

\subsection{Case 74: $\{1234,1243,3412\}$}

Recall  within a permutation $\pi=\pi_1\pi_2\cdots \pi_n$ that an index $i$ for which $\pi_i<\pi_{i+1}$ is called an \emph{ascent}.  The letters $\pi_i$ and $\pi_{i+1}$ are referred to as the \emph{ascent bottom} and \emph{ascent top}, respectively.  The \emph{leftmost ascent} is the smallest $j$ such that $\pi_j<\pi_{j+1}$ and the \emph{leftmost ascent bottom} and \emph{top} are the corresponding letters $\pi_j$ and $\pi_{j+1}$, respectively.  For example, if $\pi=87436512$, then $\pi$ has ascents at indices 4 and 7 and the leftmost ascent top and bottom are 6 and 3, respectively.  Similar terminology applies when discussing the \emph{rightmost ascent}.

Suppose that a permutation $\pi=\pi_1\cdots\pi_n$ has its leftmost ascent at index $i$.  Then we will refer to the prefix $\pi_1\cdots\pi_i$ as the \emph{initial descent sequence} (IDS).  For example, if $\pi=64325178$, then the leftmost ascent occurs at index $4$ and the IDS is $6432$.  Let $T=\{1234,1243,3412\}$. To enumerate members of $S_n(T)$, we will classify them according to the value of the leftmost ascent bottom and the nature of the IDS.  We first enumerate the following restricted class of $T$-avoiders.

\begin{lemma}\label{ids1l}
The number of $T$-avoiding permutations of length $n$ whose first letter is $\leq$  $n-2$, whose leftmost ascent top is $n$, and whose IDS  does not comprise a set of consecutive integers is given by
$$b_n=\sum_{a=1}^{n-4}\sum_{b=a+2}^{n-2}\sum_{m=2}^{b-a}\binom{b-a-1}{m-2}\left[(b-m+1)(n-b-1)+\binom{n-b-2}{2}\right], \qquad n \geq 5.$$
Moreover,
$$b_n=\frac{1}{12}(n-2)(3\cdot2^n-(n-1)(n^2-3n+12)).$$
\end{lemma}
\begin{proof}
Let $\mathcal{B}_n$ denote the set of $T$-avoiding permutations in question.  Suppose $\pi \in \mathcal{B}_n$ has first letter $b \leq n-2$, leftmost ascent at index $m$, and leftmost ascent bottom $a$.  Note that the IDS  of $\pi$ not comprising a set of consecutive integers implies $b \geq a+2$ and $2 \leq m \leq b-a$.  Then $\pi$ may be written as $\pi=ba_1\cdots a_{m-2} a n\pi'$ (*), where $a_1>\cdots >a_{m-2}$ belong to $[a+1,b-1]$.  Note that the subsequence of $\pi$ comprising $[b+1,n]$ avoids $\{123,132,3412\}$ and thus must either (i) decrease, (ii) have the form $n,n-1,\ldots,s+1,s-1,\ldots,b+1,s$ for some $b+2 \leq s \leq n-1$, or (iii) have the form $n,n-1,\ldots,s+1,s-1,\ldots,t,s,t-1,\ldots,b+1$ for some $b+3\leq s \leq n-1$ and $b+2 \leq t \leq s-1$.
Observe further that the $b-m$ letters in $[b-1]$ lying within $\pi'$ must  decrease in order to avoid $3412$.  Let $x$ denote the largest (and hence leftmost) letter in $[b-1]$ occurring in $\pi'$.  Note that $x>a$ by the assumption on the IDS  of $\pi$. If (i) holds, then the element $x$ must occur between $b+2$ and $b+1$ or directly following $b+1$ in order to avoid an occurrence of 1243 of the form $ax(b+2)(b+1)$.  Since the letters in $[b-1]$ within $\pi'$ must decrease, it follows that there are $b-m+1$ ways in which to position the elements of $[b-m]$ in $\pi'$, which uniquely determines $\pi'$.  If (ii) holds, then there are once again $b-m+1$ ways in which to position the elements of $[b-1]$ in $\pi'$ for each $s$.  Thus, combining cases (i) and (ii), there are $(b-m+1)(n-b-1)$ possibilities for $\pi'$ for each $b$ and $m$.

On the other hand, if (iii) holds, then $x$ (and hence all elements of $[b-1]$ in $\pi'$) must follow $b+1$, for otherwise there would be a 3412 of the form $(s-1)sx(b+1)$ if $x$ came between $s$ and $b+1$ or a 1243 of the form $axs(b+1)$ if $x$ occurred prior to $s$.  Thus, there are
$$\sum_{s=b+3}^{n-1}\sum_{t=b+2}^{s-1}1=\sum_{s=b+3}^{n-1}(s-b-2)=\binom{n-b-2}{2}$$
possibilities for $\pi'$ satisfying (iii).  In all cases, there are $\binom{b-a-1}{m-2}$ choices for the letters $a_1>\cdots>a_{m-2}$.  Note that for each choice of the $a_i$, the number of possibilities for $\pi'$ is the same.  Furthermore, any permutation $\pi$ of the form (*) above, where $\pi'$ satisfies (i), (ii) or (iii), is seen to avoid $T$.  Considering all possible $a$, $b$ and $m$ then yields $b_n$ members of $\mathcal{B}_n$ altogether.

Thus,
$$b_n=\sum_{a=1}^{n-4}\sum_{b=a+2}^{n-2}\sum_{m=2}^{b-a}\binom{b-a-1}{m-2}\left[(b-m+1)(n-b-1)+\binom{n-b-2}{2}\right], \qquad n \geq 5.$$
Note that
\begin{align*}
\sum_{m=2}^{b-a}\binom{b-a-1}{m-2}&\left[(b-m+1)(n-b-1)+\binom{n-b-2}{2}\right]\\
&=2^{b-a-2}(n^2+(a-b-6)n+5b+7-a-ab)\\
&-\frac{1}{2}(n^2-(2b+5-2a)n+(b+2)(b+3)-2a(b+1)),
\end{align*}
which leads to
\begin{align*}
\sum_{b=a+2}^{n-2}&\sum_{m=2}^{b-a}\binom{b-a-1}{m-2}\left[(b-m+1)(n-b-1)+\binom{n-b-2}{2}\right]\\
&=2^{n-a-2}(n-4+a)-\frac{1}{6}(n^3-6n^2-(3a^2-9a-11)n+2a^3-3a^2-5a-18).
\end{align*}
Hence,
$$b_n=\frac{1}{12}(n-2)(3\cdot2^n-(n-1)(n^2-3n+12)),$$
which completes the proof.

\end{proof}

We now consider a class of $T$-avoiding permutations where the IDS  consists of consecutive integers.

\begin{lemma}\label{ids2l}
The number of $T$-avoiding permutations of length $n$ whose first letter is $\leq$  $n-2$, whose leftmost ascent top is $n$, and whose IDS  comprises a set of consecutive integers is given by
$$(n-5)2^{n-1}+\binom{n+1}{2}+3, \qquad n \geq 3.$$
\end{lemma}
\begin{proof}
Suppose $\pi=\pi_1\cdots\pi_n$ is of the form under consideration where $n \geq 3$ and that $\pi$ has leftmost ascent at index $m$ with $\pi_m=a$.  Then the first $m$ letters of $\pi$ are $a+m-1,a+m-2,\ldots,a$, where $1 \leq a \leq n-2$ and $1 \leq m \leq n-a-1$.  Elements of $[a-1]$ must decrease within $\pi$ in order to avoid 3412.  On the other hand, elements of $[a+m,n]$ satisfy conditions (i), (ii) or (iii) as described in the proof of the previous lemma (where $b$ is taken there to be $a+m-1$).  If (i) holds, then it is seen that members of $[a-1]$ may be placed (in decreasing order) following any of the elements of $[a+m,n]$ within $\pi$ without introducing an occurrence of $T$.  Thus, there are in this case $\binom{n-m-1}{a-1}$ possible $\pi$ given $a$ and $m$.  This implies that there are a total of
$$\sum_{a=1}^{n-2}\sum_{m=1}^{n-a-1}\binom{n-m-1}{a-1}=\sum_{a=1}^{n-2}\left(\binom{n-1}{a}-1\right)=2^{n-1}-n$$
permutations in this case.  If (ii) holds, then letters in $[a-1]$ again may be inserted following any members of $[a+m,n]$, whereas if (iii) holds, then one cannot place letters in $[a-1]$ directly after $s$ or members of $[a+m+1,t-1]$ in order to avoid 3412.  Thus in (ii), there as $n-a-m+1$ possible places to insert letters in $[a-1]$, while in (iii), there are $n-t+1$ such places.

For each $a$ and $m$, considering all $s$ and $t$ yields
\begin{align*}
&\sum_{t=a+m}^{n-2}\sum_{s=t+1}^{n-1}\binom{n-t+a-1}{a-1}=\sum_{t=a+m}^{n-1}\binom{n-t+a-1}{a-1}(n-t-1)\\
&=\sum_{t=a+m}^{n}\binom{n-t+a-1}{a-1}(n-t-1)+1\\
&=\sum_{t=a+m}^n\binom{n-t+a-1}{a-1}(n-t+a)-(a+1)\binom{n-m}{a}+1\\
&= a\binom{n-m+1}{a+1}-(a+1)\binom{n-m}{a}+1=f(a,m)
\end{align*}
possible permutations for (ii) and (iii) combined.  Summing over all possible $a$ and $m$ then gives
\begin{align*}
&\sum_{a=1}^{n-2}\sum_{m=1}^{n-a-1}f(a,m)=\sum_{a=1}^{n-2}\sum_{m=1}^{n-a}f(a,m)=\sum_{a=1}^{n-2}\left(a\binom{n+1}{a+2}-(a+1)\binom{n}{a+1}+n-a\right)\\
&=\sum_{a=0}^{n-1}\left((a+2)\binom{n+1}{a+2}-(a+1)\binom{n}{a+1}-2\binom{n+1}{a+2}+n-a\right)\\
&=(n+1)(2^n-1)-n2^{n-1}-2(2^{n+1}-n-2)+\binom{n+1}{2}\\
&=(n-6)2^{n-1}+n+3+\binom{n+1}{2}.
\end{align*}
Combining this with the prior case (i) completes the proof.
\end{proof}

\begin{lemma}\label{ids3l}
The number of $T$-avoiding permutations of length $n$ whose leftmost ascent top is $n-1$ is given by
$$d_n=(n+8)2^{n-3}+1-\binom{n+2}{2}-\frac{(n-1)(n-2)(2n-3)}{6}, \qquad n \geq 3.$$
\end{lemma}
\begin{proof}
Let $\pi$ be a member of $S_n(T)$ under consideration have first letter $b \leq n-2$ and leftmost ascent bottom $a$.  Then the elements of $[b+1,n-1]$ must decrease within $\pi$.  For if not and $b+1\leq x<y \leq n-2$ with $x$ to the left of $y$, then there is a 1234 if $n$ is to the right of $y$, a 1243 if $n$ occurs between $x$ and $y$, and a 3412 if $n$ occurs between $n-1$ and $x$.  Let $S$ denote the elements of $[n]$ comprising the IDS of $\pi$.  First suppose $S$ does not consist of consecutive integers and let $x$ denote the largest element of $[a+1,b-1]$ not belonging to $S$.  We consider further cases based off of the position of the letter $n$ within $\pi$.  First assume  $n$ lies to the right of $b+1$.  Then $x$ must occur after $n$ or between $b+1$ and $n$ in order to avoid 1234.  Let the leftmost ascent of $\pi$ occur at index $m$. Since the elements of $[b-1]$ to the right of $n-1$ within $\pi$ must decrease, we see that there are $b-m+1$ ways in which to arrange the elements of $[b]-S$. As there are $\binom{b-a-1}{m-2}$ ways in which to select the middle $m-2$ elements of $S$, considering all $a$, $b$ and $m$ gives
$$\sum_{a=1}^{n-4}\sum_{b=2}^{n-2}\sum_{m=2}^{b-a}\binom{b-a-1}{m-2}(b-m+1)$$
possible permutations in this case.  On the other hand, if $n$ lies to the left of $b+1$, then all letters in $[b-1]-S$ must occur after $b+1$, for otherwise there would be a 3412 as seen with $(n-1)nx(b+1)$ or a 1243 in the form $axn(b+1)$.  There are thus $n-b-2$ possible positions for $n$ relative to the letters in $[b+1,n-2]$, which yields
$$\sum_{a=1}^{n-4}\sum_{b=2}^{n-2}\sum_{m=2}^{b-a}\binom{b-a-1}{m-2}(n-b-2)$$
additional permutations.

Combining the two preceding cases implies that the number of permutations under consideration for which $S$ does not consist of consecutive integers is given by
\begin{align*}
&\sum_{a=1}^{n-4}\sum_{b=2}^{n-2}\sum_{m=2}^{b-a}\binom{b-a-1}{m-2}(n-m-1)=\sum_{m=2}^{n-3}(n-m-1)\sum_{b=m+1}^{n-2}\sum_{a=1}^{b-m}\binom{b-a-1}{m-2}\\
&=\sum_{m=2}^{n-3}(n-m-1)\sum_{b=m+1}^{n-2}\left(\binom{b-1}{m-1}-1\right)=\sum_{m=2}^{n-3}(n-m-1)\left(\binom{n-2}{m}-(n-m-1)\right)\\
&=2^{n-2}-n+\sum_{m=2}^{n-3}(n-2)\binom{n-3}{n-m-3}-\sum_{m=2}^{n-3}(n-m-1)^2=n2^{n-3}-n+1-\sum_{m=1}^{n-2}m^2\\
&=n2^{n-3}-n+1-\frac{(n-1)(n-2)(2n-3)}{6}, \qquad n \geq 4.
\end{align*}

Now assume $S=\{a,a+1,\ldots,a+m-1\}$ for some $1\leq a \leq n-2$ and $1 \leq m \leq n-a-1$.  Then the subsequence of $\pi$ comprising the letters in $[a+m,n]$ must have the form $n-1,n-2,\ldots,t,n,t-1,\ldots,a+m$ for some $a+m \leq t \leq n-1$.  Letters in $[a-1]$ may occur only after $a+m$ or members of $[t,n-1]$ if $t>a+m$, for otherwise there would be an occurrence of 3412.  If $t=a+m$, letters in $[a-1]$ may occur after any member of $[a+m,n]$.  In either case, there are
$\binom{n-t+a-1}{a-1}$ ways in which to arrange the members of $[a-1]$, which must decrease.  One can verify that permutations $\pi$ obtained in this manner are $T$-avoiding.  Considering all $a$, $m$ and $t$ gives
\begin{align*}
&\sum_{a=1}^{n-2}\sum_{m=1}^{n-a-1}\sum_{t=a+m}^{n-1}\binom{n-t+a-1}{a-1}=\sum_{a=1}^{n-2}\sum_{m=1}^{n-a-1}\left(\binom{n-m}{a}-1\right)=\sum_{a=1}^{n-2}\left(\binom{n}{a+1}-(n-a)\right)\\
&=2^n-\binom{n+1}{2}-1
\end{align*}
possibilities.  Combining this case with the previous yields $d_n$ permutations in all.
\end{proof}

Let $a_n=|S_n(T)|$.  We may express $a_n$ in terms of the $b_n$ and $d_n$ sequences as follows, where we take $b_n=0$ if $n \leq 4$.

\begin{lemma}\label{ids4l}
If $n\geq 2$, then
\begin{equation}\label{ids3le1}
a_n=a_{n-1}+(2n-9)2^{n-2}+3+\binom{n+1}{2}+b_n+\sum_{m=0}^{n-3}d_{n-m},
\end{equation}
with $a_1=1$, where $b_n$ and $d_n$ are as defined above.
\end{lemma}
\begin{proof}
We enumerate the remaining cases of $S_n(T)$.  Let $\pi \in S_n(T)$ and not of a form described in Lemmas \ref{ids1l}--\ref{ids3l}, where $n \geq 3$.
If $\pi$ starts with $n$, then there are clearly $a_{n-1}$ possibilities.  Since members of $S_n(T)$ starting with a letter $\leq$ $n-2$ must have leftmost ascent top either equal to $n-1$ or $n$, the only remaining case is if $\pi$ starts with $n-1$, which we now assume.  If the leftmost ascent top of $\pi$ is $n$, then $\pi=(n-1)\rho'n\rho''$, where $\rho'$ and $\rho''$ are possibly empty and decreasing.  This gives $2^{n-2}$ possible $\pi$.  So assume $\pi$ has leftmost ascent top $n-m-1$ for some $1 \leq m \leq n-3$.  Since any ascent top of $\pi$ is either the largest or second largest letter yet to appear (otherwise 1234 or 1243 would be present), it follows that $\pi$ must start $n-1,n-2,\ldots,n-m$.  These letters are seen to be extraneous concerning the avoidance of $T$ and thus may be deleted (note that they impose no restrictions when considering 3412 since $n-m-1$ is the first ascent top).  This leaves a permutation of the form enumerated by $d_{n-m}$.  Considering all possible $m$ then yields $\sum_{m=1}^{n-3}d_{n-m}$ possibilities.  Combining the additional cases discussed here with those from Lemmas \ref{ids1l}--\ref{ids3l} completes the proof.
\end{proof}

By Lemmas \ref{ids1l}, \ref{ids3l} and \ref{ids4l}, we have
\begin{align*}
a_n&=a_{n-1}+(2n-9)2^{n-2}+3+\binom{n+1}{2}\\
&+\frac{1}{12}(n-2)(3\cdot2^n-(n-1)(n^2-3n+12))\\
&+\sum_{m=0}^{n-3}\left((n-m+8)2^{n-m-3}+1-\binom{n-m+2}{2}-\frac{(n-m-1)(n-m-2)(2n-m-3)}{6}\right),
\end{align*}
which is equivalent to
\begin{align*}
a_n&=a_{n-1}+(n-1)2^n-\frac{1}{6}(n-1)(n^3-3n^2+14n-6), \qquad n \geq 2,\\
\end{align*}
with $a_1=a_0=1$. Hence, we can state the following result.

\begin{theorem}\label{th74a}
Let $T=\{1234,1243,3412\}$. Then
$$F_T(x)=\frac{1-9x+35x^2-75x^3+98x^4-78x^5+36x^6-12x^7}{(1-x)^6(1-2x)^2}.$$
\end{theorem}

\subsection{Case 121: $\{1243,2341,3412\}$}
We need the following lemmas.
\begin{lemma}\label{lem121a1}
For $d\geq1$, define $N_d(x)$ to be the generating function for permutations $\pi=(n-d-1)\pi'n\pi''\in S_n(T)$ such that $\pi''$ contains the subsequence $(n-1)(n-2)\cdots(n-d)$,
and set $N(x)=\sum_{d\geq1}\frac{N_d(x)}{(1-x)^d}$. Then
$$N(x)=\frac{x^3(1-6x+15x^2-21x^3+15x^4-3x^5)}{(1-x)^4(1-2x)^2(1-3x+1+x^2)}.$$
\end{lemma}
\begin{proof}
Refine $N_d(x)$ to $N_{d,e}$ counting avoiders where $\pi'$ has $e$ letters.
Since $\pi$ avoids $1243$ and $d\geq1$, we see that $\pi'$ is decreasing.

Let $\pi=(n-d-1)\pi'n\pi''\in S_n(T)$ such that $\pi''$ contains the subsequence $(n-1)(n-2)\cdots(n-d)$ and $\pi'$ contains $e$ letters, say $\pi'=j_ej_{e-1}\cdots j_1$. To find $N_1(x)$, write $\pi''$ as $\alpha (n-1)\beta$. If $\alpha=\emptyset$, then $\beta$ is decreasing and the contribution is
$\frac{x^{e+3}}{(1-x)^{e+1}}$. Otherwise, the maximal letter of $\alpha$ is between $j_{s-1}$ and $j_{s}$ for some $s \in[e+1]$ where $j_0=0$ and $j_{e+1}=n-2$. Since $\pi$ avoids $1243$ and $2341$, $j_{s-1}<\beta< j_s$. The letters in $\pi''$ smaller than $n-2$ are decreasing. Thus, by considering whether $\beta$ is empty or not, we obtain a contribution of $\frac{x^{e+4}}{(1-x)^s}+\frac{x^{e+5}}{(1-x)^2}$ for $2\le s \le e+1$, and of $\frac{x^{e+4}}{(1-x)^2}$ for $s=1$.
Hence,
$$N_{1,e}(x)=\frac{x^{e+3}}{(1-x)^{e+1}}+\sum_{s=2}^{e+1}\left(\frac{x^{e+4}}{(1-x)^s}+\frac{x^{e+5}}{(1-x)^2}\right)+\frac{x^{e+4}}{(1-x)^2}\,,$$
which leads to
\begin{align}\label{eq121a1b}
N_1(x)=\sum_{e\geq0}N_{1,e}(x)=\frac{x^3(1-3x+4x^2-3x^3)}{(1-x)^4(1-2x)}\,.
\end{align}

For $d\geq2$, by a decomposition similar to the case $d=1$, we see that $$N_{d,e}(x)=xN_{d-1,e}(x)+\frac{x^{d+e+3}}{(1-x)^{d+1}}\,.$$
Summing over $e\geq0$, we obtain for $d\geq2$,
$$N_d(x)=xN_{d-1}(x)+\frac{x^{d+3}}{(1-x)^{d+2}}.$$
Multiplying by $\frac{1}{(1-x)^d}$ and summing over $d\geq2$, we have
$$N(x)-\frac{N_1(x)}{1-x}=\frac{x}{1-x}N(x)+\sum_{d\geq2}\frac{x^{d+3}}{(1-x)^{2d+2}}\,.$$
Now use \eqref{eq121a1b} and solve for $N(x)$ to complete the proof.
\end{proof}

\begin{lemma}\label{lem121a2}
Define $M(x)$ to be the generating function for permutations $\pi=i\pi'n\pi''\in
S_n(T)$ with $2$ left-right maxima such that $i \le n-2$. Then
$$M(x)=\frac{x^3(1-6x+16x^2-25x^3+20x^4-5x^5)}{(1-x)^4(1-2x)^2(1-3x+x^2)}\,.$$
\end{lemma}
\begin{proof}
Let $M_d(x)$ denote the generating function for permutations $\pi=i\pi'n\pi''\in
S_n(T)$ with $2$ left-right maxima such that $\pi''$ contains the subsequence
$(n-1)\cdots(n-d)$ and $i \le n-d-1$. Clearly, $M_1(x)=M(x)$.  Since $\pi$ avoids $1243$,  there is no letter greater than $i$ between $n$ and
$n-d+1$. Thus, all letters greater than $i$ in $\pi''$ must occur to the right of $n-d+1$. We denote the subsequence comprising these letters by $\ga$.

To find $M_d(x)$, write $\ga$ as $\alpha (n-d)\beta$. In cases (i)
$\alpha=\beta=\emptyset$, (ii) $\alpha\neq\emptyset$ and $\beta=\emptyset$, (iii)
$\alpha=\emptyset$ and $\beta\neq\emptyset$, we have the contributions (i) $N_d(x)$,
(ii) $xM_d(x)$, (iii) $M_{d+1}(x)$, respectively.  Thus, we may assume that
$\alpha,\beta\neq\emptyset$. Note first that $i+1$ must occur in $\beta$, for otherwise $i(i+1)$ is a 12 within a 1243, and by similar reasoning $\beta<\alpha$.  Thus, $\alpha$ and $\beta$ form decreasing subsequences since $\pi$ avoids $2341$ and $3412$, which implies $\alpha\beta$ is decreasing. Furthermore, no letter in $[i-1]$ can occur to the right of $n-d$, for otherwise there would be an occurrence of $2341$ of the form $i(n-d-1)(n-d)i'$ for some $i'\in [i-1]$.  From the preceding observations, we see that $\pi$ can be expressed as follows:
\begin{align*}
&\pi=i(i-1)\cdots i'
n\gamma^{(1)}(n-1)\gamma^{(2)}\cdots\gamma^{(d-1)}(n-d+1)\gamma^{(d)}(n-d-1)\cdots\gamma^{(e-1)}(n-e)\\
&\qquad\qquad\qquad\qquad\qquad\qquad\qquad\qquad\gamma^{(e)}(n-d)(n-e-1)\cdots(i+1),
\end{align*}
where $e \geq d+1$, $i\geq i'>\gamma^{(1)}\cdots\gamma^{(e)}$ and
$\gamma^{(1)}\cdots\gamma^{(e)}$ is decreasing (for otherwise, 1243 would be present). Thus, we have a contribution of
$\sum_{e\geq d+1}\frac{x^{e+3}}{(1-x)^{e+2}}=\frac{x^{d+4}}{(1-x)^{d+2}(1-2x)}$.
Hence
$$M_d(x)=N_d(x)+xM_d(x)+M_{d+1}(x)+\frac{x^{d+4}}{(1-x)^{d+2}(1-2x)},$$
which implies
$$M_d(x)=\frac{1}{1-x}N_d(x)+\frac{1}{1-x}M_{d+1}(x)+\frac{x^{d+4}}{(1-x)^{d+3}(1-2x)}.$$
Multiplying by $\frac{1}{(1-x)^{d-1}}$ and summing over $d\geq1$, we obtain
$$\sum_{d\geq1}\frac{M_d(x)}{(1-x)^{d-1}}=N(x)+\sum_{d\geq1}\frac{M_{d+1}(x)}{(1-x)^d}+\sum_{d\geq1}\frac{x^{d+4}}{(1-x)^{2d+2}(1-2x)}\,,$$
which is equivalent to
$$M_1(x)=N(x)+\sum_{d\geq1}\frac{x^{d+4}}{(1-x)^{2d+2}(1-2x)}\,,$$
and the result follows from Lemma \ref{lem121a1}.
\end{proof}

\begin{lemma}\label{lem121a3}
$$G_2(x)=x^2F_T(x)+\frac{x^3(2-14x+42x^2-70x^3+64x^4-27x^5+4x^6)}{(1-x)^4(1-2x)^2(1-3x+x^2)}.$$
\end{lemma}
\begin{proof}
Clearly, $G_2(x)=M(x)+H(x)$, where $H(x)$ is the generating function for permutations $\pi=(n-1)\pi'n\pi''\in S_n(T)$. Since $(n-1)\pi'n$ avoids $T$ if and only if $\pi'$ avoids $T$, $G_2(x)=M(x)+x^2F_T(x)+H'(x)$, where $H'(x)$ is the generating function for permutations
$\pi=(n-1)\pi'n\pi''\in S_n(T)$ where $\pi''$ is nonempty. By similar arguments as in the proofs of Lemmas \ref{lem121a1} and \ref{lem121a2}, one can show that $H'(x)=\frac{x^3(1-2x+2x^2)}{(1-x)^3(1-2x)}$ and the result follows using Lemma \ref{lem121a2}.
\end{proof}

\begin{theorem}\label{th121a}
Let $T=\{1243,2341,3412\}$. Then
$$F_T(x)=\frac{1}{1-3x+x^2}+\frac{1}{1-x}+\frac{1-x+x^2}{(1-x)^4}-\frac{2(1-x)(1-2x-x^2)}{(1-2x)^3}.$$
\end{theorem}
\begin{proof}
To find $G_m(x)$ with $m\geq3$, let $\pi=i_1\pi^{(1)}\cdots i_m\pi^{(m)}\in S_n(T)$ with $m\ge 3$ left-right maxima. Since $\pi$ avoids $T$, we see that $\pi^{(s)}=\emptyset$ for all $s\geq4$, with $i_1<\pi^{(3)}<i_2$. If $\pi^{(3)}=\emptyset$, then we have a contribution of $x^{m-2}G_2(x)$. Otherwise, the letters in $\pi^{(1)}\pi^{(2)}$ smaller than $i_1$ are decreasing and
the letters of $\pi^{(2)}\pi^{(3)}$ between $i_1$ and $i_2$ are decreasing. Thus, we have a contribution of $\frac{x^{m+1}}{(1-x)^2(1-2x)}$. Hence,
$$G_m(x)=x^{m-2}G_2(x)+\frac{x^{m+1}}{(1-x)^2(1-2x)}\,.$$
Summing over $m\geq3$, we obtain
$$F_T(x)-1-xF_T(x)-G_2(x)=\frac{x}{1-x}G_2(x)+\frac{x^4}{(1-x)^3(1-2x)},$$
and the result follows by substituting for $G_2(x)$ and solving for $F_T(x)$.
\end{proof}

\subsection{Case 125: $\{1243,2341,4123\}$} We treat this case in 4 subsections.

\subsubsection{Case I}

Let $T=\{1243,2341,4123\}$ and $u_n$ denote the number of $T$-avoiding permutations of length $n$ starting with $n-1$ and having leftmost ascent of the form $a,n$, where $2 \leq a \leq n-1$.
If $1 \leq i \leq n-1$, then let $C_{n,i}'$ denote the number of $123$-avoiding permutations of length $n$ whose leftmost ascent occurs at index $i$, with $C_{n,n}'=1$.  The numbers $u_n$ may be expressed explicitly in terms of $C_{n,i}'$ as follows.

\begin{lemma}\label{unlem}
We have
\begin{equation}\label{unleme1}
u_n=C_{n-2}+\sum_{a=2}^{n-2}\sum_{t=0}^{n-2-a}\binom{n-2-a}{t}u'(n-t,a), \qquad n \geq 3,
\end{equation}
where
$$u'(m,a)=\sum_{i=1}^{a-1}\binom{i+m-2-a}{i}C_{a-1,i}', \qquad 2 \leq a \leq m-2.$$
\end{lemma}
\begin{proof}
Clearly, there are $C_{n-2}$ permutations avoiding $T$ that start $n-1,n$.  So let $\pi \in S_n(T)$ be a permutation enumerated by $u_n$ having leftmost ascent bottom $a \leq n-2$.  Then $\pi$ is either of the form $\pi=(n-1)a_1a_2\cdots a_t a n\pi'$, where $a_1>a_2>\cdots>a_t>a$ for some $t \geq 1$ or $\pi=(n-1)an\pi'$.  One may verify that the letters $a_1,\ldots,a_t$ are extraneous concerning the avoidance of the patterns in $T$ (note that these letters have no bearing on $2341$, since letters in $[a+1,n-2]$ occurring to the right of $a$ in $\pi$ must occur as a decreasing subsequence to avoid $4123$).  Let $u'(m,a)$ denote the number of $T$-avoiding permutations of length $m$ starting $(m-1)am$.  Considering all possible $a$ and $t$ thus implies that there are $\sum_{a=2}^{n-2}\sum_{t=0}^{n-2-a}\binom{n-2-a}{t}u'(n-t,a)$ permutations $\pi$ of the stated form above.

To complete the proof of \eqref{unleme1}, we need to establish the formula above for $u'(m,a)$.  To do so, let $\rho \in S_m(T)$ be of the form enumerated by $u'(m,a)$.  Then $\rho=(m-1)am\rho'$ for some permutation $\rho'$ of $[a-1]\cup[a+1,m-2]$. Note that $\rho'$ is $123$-avoiding with all letters in $[a+1,m-2]$ descending since $\rho$ avoids $4123$.  Thus, $\rho'$ may be obtained by inserting letters from $[a+1,m-2]$ within a $123$-avoiding permutation $\lambda$ of $[a-1]$.  Suppose $\lambda=\lambda_1\lambda_2\cdots \lambda_{a-1}$, with the first ascent of $\lambda$ occurring at index $i$ (where $i=a-1$ if $\lambda=(a-1)\cdots 21$).  Since $\rho'$ avoids $123$, no letter from $[a+1,m-2]$ may be inserted beyond the $(i+1)$-st letter of $\lambda$.  Also, letters are to be inserted in decreasing order, with letters preceding $\lambda$ allowed.  Since there are $i+1$ positions in which to insert the letters from $[a+1,m-2]$, it follows that there are $\binom{i+m-2-a}{i}$ possible $\rho$ for each $\lambda$ whose first ascent occurs at index $i$.  One may verify that permutations $\rho$ so obtained avoid $T$.  Considering all $1 \leq i \leq a-1$ yields the desired formula for $u'(m,a)$ and completes the proof.
\end{proof}

Let $C_{n,i}$ denote the number of $123$-avoiding permutations of length $n$ having first letter $i$.  Let $q(x,y)=\sum_{n\geq 1}\sum_{i=1}^nC_{n,i}x^ny^i$ and $\ell(x,y)=\sum_{n\geq 1}\sum_{i=1}^nC_{n,i}'x^ny^i$.  Then the latter may be expressed in terms of the former as follows.

\begin{lemma}\label{refingf}
We have
\begin{equation}\label{refingfe1}
\ell(x,y)=\frac{xy}{(1-xy)^2}q\left(\frac{x}{1-xy},1-xy\right)+\frac{xy}{1-xy}.
\end{equation}
\end{lemma}
\begin{proof}
We refine $C_{a,i}'$ according to the leftmost ascent bottom.  That is, let $C_{a,i,j}'$ denote the number of $123$-avoiding permutations $\pi$ of length $a$ with leftmost ascent at index $i$ and leftmost ascent bottom $j$.  Note that $C_{a,i,j}'$ can only be non-zero when $1 \leq j \leq a-i$.  Furthermore, any letters of $\pi$ coming prior to $j$ are extraneous concerning avoidance of the patterns in $T$ and thus may be deleted, as can the leftmost ascent top.  This leaves a permutation of the form enumerated by $C_{a-i,j}$.  Upon selecting the first $i-1$ letters of $\pi$, which must belong to $[j+1,a]$, yields the formula
$$C_{a,i}'=\sum_{j=1}^{a-i}C_{a,i,j}'=\sum_{j=1}^{a-i}\binom{a-j}{i-1}C_{a-i,j}, \qquad 1 \leq i \leq a-1,$$
with $C_{a,a}'=1$.  Thus we have
\begin{align*}
\ell(x,y)&=\sum_{a\geq 1}C_{a,a}'(xy)^a+\sum_{a\geq 2}\sum_{i=1}^{a-1}\left(\sum_{j=1}^{a-i}\binom{a-j}{i-1}C_{a-i,j}\right)x^ay^i\\
&=\frac{xy}{1-xy}+\sum_{a\geq 2}\sum_{i=1}^{a-1}\left(\sum_{j=1}^i\binom{a-j}{i-j+1}C_{i,j}\right)x^ay^{a-i}\\
&=\frac{xy}{1-xy}+\sum_{i\geq 1}\sum_{j=1}^iC_{i,j}\sum_{a \geq i+1}\binom{a-j}{i-j+1}x^ay^{a-i}\\
&=\frac{xy}{1-xy}+\frac{xy}{(1-xy)^2}\sum_{i\geq 1}\sum_{j=1}^iC_{i,j}\left(\frac{x}{1-xy}\right)^i(1-xy)^j,
\end{align*}
which implies \eqref{refingfe1}.
\end{proof}

Note that from \cite{FM}, we have
\begin{align}
q(x,y)=\frac{xy(1-yC(xy))}{1-x-y}.\label{eqtqq}
\end{align}
We now calculate the generating function $u(x)=\sum_{n\geq 3}u_nx^n$, which will be needed later.

\begin{lemma}\label{ungf}
We have
\begin{equation}\label{ungfe1}
u(x)=x^2(C(x)-1)+(1-x)r\left(\frac{x}{1-x},1-x\right),
\end{equation}
where $$r(x,y)=\frac{x^4y^2}{(1-x-xy)^2}q\left(\frac{xy(1-x)}{1-x-xy},\frac{1-x-xy}{1-x}\right)+\frac{x^4y^2}{(1-x)(1-x-xy)}.$$
\end{lemma}
\begin{proof}
First note that
\begin{align*}
&\sum_{n\geq 4}\sum_{a=2}^{n-2}u'(n,a)x^ny^a=\sum_{n\geq 4}\sum_{a=2}^{n-2}\sum_{i=1}^{a-1}C_{a-1,i}'\binom{i+n-2-a}{i}x^ny^a\\
&=\sum_{a\geq 2}\sum_{i=1}^{a-1}C_{a-1,i}'y^a\sum_{n\geq a+2}\binom{i+n-2-a}{i}x^n=\sum_{a\geq2}\sum_{i=1}^{a-1}C_{a-1,i}'\frac{x^{a+2}y^a}{(1-x)^{i+1}}\\
&=\frac{x^3y}{1-x}\sum_{a\geq1}\sum_{i=1}^{a}C_{a,i}'\frac{(xy)^a}{(1-x)^i}=\frac{x^3y}{1-x}\ell\left(xy,\frac{1}{1-x}\right)\\
&=\frac{x^4y^2}{(1-x-xy)^2}q\left(\frac{xy(1-x)}{1-x-xy},\frac{1-x-xy}{1-x}\right)+\frac{x^4y^2}{(1-x)(1-x-xy)},
\end{align*}
where we have used \eqref{refingfe1} in the final equality.  Next observe that
\begin{align*}
&\sum_{n\geq 4}\sum_{a=2}^{n-2}\sum_{t=0}^{n-2-a}\binom{n-2-a}{t}u'(n-t,a)x^ny^a=\sum_{a\geq 2}\sum_{n\geq a+2}\sum_{t=a+2}^n\binom{n-2-a}{n-t}u'(t,a)x^ny^a\\
&=\sum_{a\geq 2}\sum_{t \geq a+2}u'(t,a)y^a\sum_{n\geq t}\binom{n-2-a}{t-2-a}x^n=\sum_{a\geq 2}\sum_{t \geq a+2}u'(t,a)\frac{x^ty^a}{(1-x)^{t-a-1}}\\
&=(1-x)\sum_{t\geq 4}\sum_{a=2}^{t-2}u'(t,a)\left(\frac{x}{1-x}\right)^t((1-x)y)^a=(1-x)r\left(\frac{x}{1-x},(1-x)y\right).
\end{align*}
Multiplying both sides of \eqref{unleme1} by $x^n$, and summing over $n \geq 3$, then gives
\begin{align*}
u(x)&=x^2(C(x)-1)+\sum_{n\geq 4}\left(\sum_{a=2}^{n-2}\sum_{t=0}^{n-2-a}\binom{n-2-a}{t}u'(n-t,a)\right)x^n\\
&=x^2(C(x)-1)+(1-x)r\left(\frac{x}{1-x},(1-x)y\right)\mid_{y=1},
\end{align*}
which implies \eqref{ungfe1}.
\end{proof}

Now let $b_n$ denote the  number of $T$-avoiding permutations of length $n$ whose leftmost ascent is of the form $a,n$ for some $2 \leq a \leq n-1$.  For example, $b_3=1$ (for $231$) and $b_4=5$ (for $2413$, $2431$, $3241$, $3412$ and $3421$).  Then we have the following recurrence formula.
\begin{lemma}\label{b_nlem}
We have
\begin{equation}\label{b_nrec}
b_n=b_{n-1}+u_{n}+2^{n-3}-n+2+\binom{n-2}{4}+\sum_{j=0}^{n-4}\sum_{i=1}^{n-3-j}\binom{n-i-2}{j+1}C_{n-2-j,i},\qquad n \geq 3,
\end{equation}
with $b_2=0$, where $u_n$ is given by \eqref{unleme1} above.
\end{lemma}
\begin{proof}
Let $\pi$ be a member of $S_n(T)$ enumerated by $b_n$.  If the first letter of $\pi$ is $n-1$, then there are $u_n$ possibilities by definition, so assume $n\geq 4$ and that the first letter of $\pi$ is less than $n-1$.  Then $\pi=a_1a_2\cdots a_jan\pi'$, where $j \geq 0$ and $2 \leq a<a_j<\cdots< a_1 \leq n-2$.  Let $S=\{a_1,\ldots,a_j\}$, where $S$ is possibly empty.  Note that $\pi'$ is nonempty and let $x$ denote the first letter of $\pi'$.  If $x=n-1$, then it is seen that $x$ may be deleted and there are $b_{n-1}$ possibilities.

Now suppose $x\in [a-1]$.  Note that $\pi'$ is $123$-avoiding due to $4123$.  Furthermore, this is the only restriction on $\pi'$ since $n,x$ implies all letters in $[a+1,n-1]$ in $\pi'$ must decrease, and thus no letter to the left of $n$ can serve as a ``2'' within an occurrence of 2341 or as a ``4'' within a 4123.  Also, $x<a$ implies all elements of $S\cup\{a\}$ are extraneous concerning the avoidance of 1243 as well. Thus, these letters may all be deleted in addition to $n$, leaving a permutation of the form enumerated by $C_{n-2-j,i}$ for some $i \in [a-1]$. Considering all $2 \leq a \leq n-2$, $0 \leq j \leq n-2-a$ and $1 \leq i \leq a-1$ gives
\begin{align*}
&\sum_{a=2}^{n-2}\sum_{j=0}^{n-2-a}\sum_{i=1}^{a-1}\binom{n-2-a}{j}C_{n-2-j,i}=\sum_{j=0}^{n-4}\sum_{i=1}^{n-3-j}C_{n-2-j,i}\sum_{a=i+1}^{n-2-j}\binom{n-2-a}{j}\\
&\quad=\sum_{j=0}^{n-4}\sum_{i=1}^{n-3-j}\binom{n-2-i}{j+1}C_{n-2-j,i}
\end{align*}
possibilities in this case.

Next let $a+1 \leq x \leq n-2$.  Here, we consider two cases on $S$.  Note that $S$ cannot comprise all of $[a+1,n-2]$ due to our assumption on $x$.  First suppose $S=[a+1,s]\cup[t,n-2]$ for some $a \leq s<t-1\leq n-2$, where $S=\emptyset$ is possible.  Then $\pi'$ is of the form
$$\pi'=b_1b_2\cdots b_p(n-1)d_1d_2\cdots d_q,$$
where $p \geq 1$, $q \geq 0$ and $b_1\in[a+1,n-2]$.  Note that no elements of $[a-1]$ can occur to the right of $n-1$, for otherwise there would be an occurrence of $2341$ of the form $ab_1(n-1)z$ for some $z \in [a-1]$.  Since letters occurring between $n$ and $n-1$ within $\pi$ are decreasing due to 4123, the final $a-1$ letters $b_i$ must comprise $[a-1]$; that is, we have $b_{p-a+2}=a-1, b_{p-a+2}=a-2, \ldots, b_p=1$.  Since $a \geq 2$, the $d_i$ letters must also decrease in order to avoid 4123.  Furthermore, all of the $d_i$ letters are less than all of the letters in $[a+1,n-2]$ occurring between $n$ and $n-1$, for otherwise, there would be an occurrence of 1243 (with $a$ and $n-1$ corresponding to the ``1'' and ``4'', respectively).  Thus, the subsequence $b_1b_2\cdots b_{p-a+1}d_1d_2\cdots d_q$ comprises all elements of $[a+1,n-2]$ occurring to the right of $n$ within $\pi$ in decreasing order, where $p \geq a$ and $q$ is possibly zero. One may verify that the corresponding $\pi$ of  the stated form is $T$-avoiding.   Considering all $a$ and $S$ implies that there are
$$\sum_{a=2}^{n-3}\sum_{i=0}^{n-3-a}\sum_{j=0}^{n-3-a-i}(n-2-a-i-j)=\sum_{a=2}^{n-3}\sum_{i=0}^{n-3-a}\binom{n-1-a-i}{2}=\sum_{a=2}^{n-3}\binom{n-a}{3}=\binom{n-1}{4}$$
possibilities in this case.

If $S$ is not of the form $[a+1,s]\cup[t,n+2]$, then there must exist some $v_1<u<v_2$ all belonging to $[a+1,n-2]$ such that $u \in S$ and $v_1,v_2 \notin S$.  If $\pi'=b_1\cdots b_p(n-1)d_1\cdots d_q$ is as before, with $q>0$, then $b_1\geq v_2>u$ and $d_q\leq v_1<u$ implies $ub_1(n-1)d_q$ is a 2341, which is impossible.  Thus $q=0$ and there is only one way in which to arrange the letters within $\pi$ to the right of $n$, i.e., letters in $[n-2]-S$ in decreasing order followed by $a-1,\ldots,1,n-1$.  By subtraction, there are for $n \geq 4$,
\begin{align*}
\sum_{a=2}^{n-3}\left(2^{n-2-a}-1-\sum_{i=0}^{n-3-a}\sum_{j=0}^{n-3-a-i}1\right)&= 2^{n-3}-(n-2)-\sum_{a=2}^{n-3}\binom{n-1-a}{2}\\
 &=2^{n-3}-(n-2)-\binom{n-2}{3}
 \end{align*}
 possible members of $S_n(T)$ in this case.  Combining all of the previous cases gives \eqref{b_nrec}.
\end{proof}

\subsubsection{Case II}

Let $d_n$ denote the number of $T$-avoiding permutations of length $n$ not starting with $n$ whose leftmost ascent is of the form $a,n-1$ for some $2 \leq a \leq n-2$.  Note, for example, $d_5=6$, the enumerated permutations being $\{24135, 24153, 24315, 32415, 34125, 34215\}$.  Then $d_n$ may be expressed in terms of $b_n$ as follows.

\begin{lemma}\label{dnlem}
We have
\begin{equation}\label{dnrec}
d_n=d_{n-1}+b_{n-1}-b_{n-2}+\binom{n-3}{2}+\sum_{a=3}^{n-3}\sum_{\ell=0}^{n-3-a}\sum_{m=1}^{n-2-a-\ell}\binom{n-5-\ell-m}{a-3}m,  \qquad n \geq 5,
\end{equation}
with $d_4=1$, where $b_n$ is given by \eqref{b_nrec} above.
\end{lemma}
\begin{proof}
Let $\mathcal{B}_m$ and $\mathcal{D}_m$ denote the subsets of $S_m(T)$ enumerated by $b_m$ and $d_m$ for $m \geq 3$.  To show \eqref{dnrec}, consider forming members of $\mathcal{D}_n$ by inserting $n$ somewhere to the right of $n-1$ within members of $\mathcal{B}_{n-1}$, where $n \geq 5$.  First note that there are $d_{n-1}$ possibilities if the letter $n-2$ is to directly follow $n-1$ within a member of $\mathcal{B}_{n-1}$, upon deleting $n-1$. Let $\mathcal{B}_m'$ denote the subset of $\mathcal{B}_m$ consisting of those members where $m-1$ does not directly follow $m$.  Note that one may always append the letter $n$ to a member of $\mathcal{B}_{n-1}'$ without introducing an occurrence of one of the patterns in $T$, which yields $b_{n-1}-b_{n-2}$ members of $\mathcal{D}_n$.

We now consider the various cases considered in the proof of recurrence \eqref{b_nrec} above in an effort to determine any further members of $\mathcal{D}_n$.  Let $\pi \in \mathcal{B}_{n-1}'$.  If the first letter of $\pi$ is $n-2$, then $n$ must be added as the final letter (in order to avoid $2341$), and hence there are no additional members of $\mathcal{D}_n$ coming from this case (beyond those already accounted for by $b_{n-1}-b_{n-2}$ above).  So assume henceforth that the first letter of $\pi$ is less than $n-2$, and write $\pi=\pi'a(n-1)\pi''$, where $a \geq 2$ is the leftmost ascent bottom.  Let $y$ be the first letter of $\pi''$.  First assume $a+1\leq y \leq n-3$.  Then the element $1$ must (directly) precede $n-2$ within $\pi''$, or $\pi$ would contain $2341$.  Thus, the letter $n$ cannot be inserted between $n-1$ and $y$, for otherwise there is an occurrence of $2341$, and it cannot be inserted anywhere between $y$ and $n-2$ either due to $1243$ occurring in that case.  If $n$ were to be inserted to the right of $n-2$, but not at the very end of $\pi$, then there would be a $2341$ in which the roles of ``2'' and ``3'' are played by the letters $y$ and $n-2$, respectively.  Thus, no additional members of $\mathcal{D}_n$ arise if $y \in [a+1,n-3]$.

Now assume $y \in [a-2]$ and let $z$ denote the last letter of $\pi''$.  Note that $z\neq y$ since $\pi''$ has length at least two.  Then $z$ must belong to $[a-3]\cup\{a-1\}$ since $\pi''$ is $123$-avoiding, which again implies one must add $n$ to the very end of $\pi$ to create a member of $\mathcal{D}_n$ (in order to avoid an occurrence of $2341$ as witnessed by the subsequence $a(n-1)nz$).  So assume henceforth $y=a-1$.  Still, no additional members of $\mathcal{D}_n$ are possible if $z \in [a-2]$, so assume henceforth $z>a$.  We now consider cases on $\pi'$.  First suppose that $\pi'$ does not comprise $[a+1,a+\ell]$ for some $0 \leq \ell \leq n-3-a$.  The there exist $a+1\leq u<v \leq n-3$ with $u \notin \pi'$ and $v \in \pi'$.  Then $u \geq z$ since $z$ is seen to be the smallest element of $[a+1,n-3]$ in $\pi''$.  Thus, adding $n$ anywhere except at the very end of $\pi$ produces an occurrence of $2341$, as witnessed by $v(n-1)nz$.

So assume $\pi'$ comprises $[a+1,a+\ell]$ for some $\ell$.  If $a \geq 3$, then $\pi$ may be written as
$$\pi=(a+\ell)\cdots(a+1)a(n-1)(a-1)\rho1(a+\ell+m)\cdots(a+\ell+2)(a+\ell+1), \qquad n \geq 6,$$
for some $1 \leq m \leq n-2-a-\ell$, where $\rho$ consists of the remaining letters of $[n-1]$.  Since the set of letters to the right of $n-1$ within $\pi$ form a $123$-avoiding permutation, with the last letter greater than the first, the subsequence $\rho$ consists of alternating runs of letters from the disjoint sets $[2,a-2]$ and $[a+\ell+m+1,n-2]$, where the subsequences of $\pi$ comprising the letters from these sets are decreasing (in order to avoid 4123). Thus, selecting the positions of the elements of $[2,a-2]$ uniquely determines $\rho$, which can be done in $\binom{n-5-\ell-m}{a-3}$ ways.  In this case, $n$ may be inserted directly prior to any of the elements of $[a+\ell+1,a+\ell+m]$ and the resulting permutation is seen to always belong to $S_n(T)$.  Considering all possible $a$, $\ell$ and $m$ gives
$$\sum_{a=3}^{n-3}\sum_{\ell=0}^{n-3-a}\sum_{m=1}^{n-2-a-\ell}\binom{n-5-\ell-m}{a-3}m$$
members of $\mathcal{D}_n$ that have not been enumerated already.  If $a=2$, then we have
$$\pi=(\ell+2)\cdots 32(n-1)1(n-2)(n-3)\cdots (\ell+3), \qquad n \geq 5,$$
in which case $n$ may be inserted directly prior to any element of $[\ell+3,n-2]$.  This gives
$\sum_{\ell=0}^{n-5}(n-4-\ell)=\binom{n-3}{2}$ additional members of $\mathcal{D}_n$.  Combining this case with the previous yields \eqref{dnrec}.

\end{proof}

\subsubsection{Case III} Let $e_n$ denote the number of $T$-avoiding permutations of length $n$ not starting with $n$ whose leftmost ascent is of the form $a,b$, where $2\leq a < b \leq n-2$.  Note that $e_4=0$ and $e_5=2$, the enumerated permutations being $23145$ and $42315$. The sequence $e_n$ may be expressed in terms of $b_n$ and $d_n$ as follows.

\begin{lemma}\label{enlem}
We have
\begin{equation}\label{enrec}
e_n=e_{n-1}+d_{n-1}+C_{n-2}-2^{n-3}+\sum_{i=3}^{n-2}(d_{i+1}-b_i), \qquad n \geq 4,
\end{equation}
with $e_3=0$, where $b_n$ and $d_n$ are determined by \eqref{b_nrec} and \eqref{dnrec}.
\end{lemma}
\begin{proof}
Let $\pi$ be a permutation enumerated by $e_n$, where $n \geq 5$.  To show \eqref{enrec}, we consider cases on the first and the last letters of $\pi$.  Let $a$ and $b$ be the letters involved in the leftmost ascent of $\pi$, where $2 \leq a< b \leq n-2$, and let $x$ and $y$ denote the first and last letters of $\pi$.  If $x \neq n-1$ and $y=n$, then deleting $n$ gives $e_{n-1}+d_{n-1}$ possibilities in this case, upon considering whether or not $b$ is less than $n-2$. If $x=n-1$ and $y=n$, then $\pi$ must start $n-1,n-2,\ldots,b+1$.  For if not, and $p \in [b+1,n-1]$ lies to the right of $b$ within $\pi$, then $\pi$ contains a $4123$, as witnessed by the subsequence $(n-1)abp$. Then the elements of $[b+1,n]$ may be deleted from $\pi$, which leaves a $123$-avoiding permutation of length $b$ whose leftmost ascent is $a,b$, where $a \geq 2$.  Upon subtracting the $123$-avoiding permutations of length $b$ that start with $b$ or whose leftmost ascent is $1,b$, one gets $C_b-C_{b-1}-2^{b-2}$ possibilities for each $b$.  Summing over $3 \leq b \leq n-2$ gives $\sum_{b=3}^{n-2}\left(C_b-C_{b-1}-2^{b-2}\right)=C_{n-2}-2^{n-3}$ possible $\pi$ in this case.

Finally, assume $y \neq n$.  Then $\pi$ must start $n-1,n-2,\ldots,b+1$.  To show this, suppose it is not the case and let $q \in [b+1,n-1]$ lie to the right of $b$.  Note that any letter to the right of $n$ within $\pi$ must be less than $b$ so as to avoid $1243$.  In particular, $q$ must lie between $b$ and $n$.  Since $y \neq n$, there exists a letter $r$ to the right of $n$.  But then $bqnr$ is a $2341$, which is impossible.  Furthermore, one may verify, without using any letters to the left of $a$, that any elements of $[a+1,b-1]$ to the right of $b$ within $\pi$ must decrease.  Thus, it is seen that the letters $n-1,n-2,\ldots,b+1$ impose no restriction on the subsequent letters of $\pi$ when considering the pattern $4123$.  The same can be said for the patterns 1243 and 2341, which implies that these letters are extraneous concerning the avoidance of $T$ and thus may be deleted.  Doing so leaves a permutation of length $b+1$, not starting or ending with $b+1$, whose leftmost ascent is of the form $a \leq b$, where $a \geq 2$.  Considering all possible $b$ gives $\sum_{i=3}^{n-2}(d_{i+1}-b_i)$ possibilities in this case, by subtraction, which completes the proof.
\end{proof}

\subsubsection{Case IV} Let $g_n$ denote the number of $T$-avoiding permutations of length $n$ not starting with $n$ whose leftmost ascent is of the form $1,a$ for some $a>1$.
\begin{lemma}\label{gnlem}
We have
\begin{equation}\label{gnrec}
g_n=g_{n-1}+5\cdot2^{n-3}+\binom{n-1}{4}-n+\sum_{m=3}^{n-1}\binom{n-m+1}{2}2^{m-3}, \qquad n \geq 3,
\end{equation}
with $g_2=1$.
\end{lemma}
\begin{proof}
Recurrence \eqref{gnrec} is clear if $n=3$, so assume $n \geq 4$.  Let $\pi=a_1a_2\cdots a_s1\pi'$ be of the form enumerated by $g_n$, where $1<a_s<a_{s-1}<\cdots<a_1\leq n-1$, and let $S=\{a_1,a_2,\ldots ,a_s\}$.  We consider several cases on $S$. If $S=\emptyset$, then $\pi=1\pi'$ where $\pi'$ avoids $\{132,2341,4123\}$, and hence there are $2^{n-1}-n+1$ possibilities.  So assume $S \neq \emptyset$.  If $a_s=2$, then deleting 2 gives $g_{n-1}$ possibilities.  Henceforth assume $a_s \geq 3$.  We will refer to letters in $[a_1+1,n]$ as \emph{large} and those belonging to $[2,a_1]-S$ as \emph{small}.  Note that $\pi'$ contains both large and small letters.

Large letters to the left of any small letters within $\pi'$ must decrease (due to 2341), while large letters to the right of any small letters must increase (1243).  The small letters themselves must decrease, for otherwise there would be a 4123 starting with $a_1$, $1$.  It is not possible for there to exist two or more large letters occurring between the greatest and the least of the small letters, for otherwise $\pi$ would contain 1243 or 2341. Assume now that there is exactly one such large letter.  In order for this to occur, the set $S$ must be an interval of the form $[\ell,\ell+s-1]$, where $4 \leq \ell \leq n-1$ and $1 \leq s \leq n-\ell$.  For if not, then the largest small letter $y$ would exceed the smallest member $x$ of $S$, which would introduce an occurrence of $2341$ of the form $xyz2$ for some large letter $z$.  Furthermore, any large letters occurring to the left of the first small letter must be less than any large letters occurring to the right of the last large letter in this case, for otherwise there would be an 4123 where the ``1'' and ``2'' correspond to the largest small and the smallest large letter, respectively.  One may verify that a permutation satisfying all of the requirements above with respect to the large and small letters and to the set $S$ is $T$-avoiding.  Note that given $\ell$ there are $\ell-3$ possible positions in which to place the smallest large letter since it must go between two small letters.  If the leftmost large letter within $\pi'$ is denoted by $n-t$, where $0 \leq t \leq n-\ell-s$, then considering all possible $\ell$, $s$ and $t$ yields
\begin{align*}
\sum_{\ell=4}^{n-1}\sum_{s=1}^{n-\ell}\sum_{t=0}^{n-\ell-s}(\ell-3)&=\sum_{s=1}^{n-4}\sum_{t=0}^{n-4-s}\sum_{\ell=4}^{n-s-t}(\ell-3)=\sum_{s=1}^{n-4}\sum_{t=0}^{n-4-s}\binom{n-2-s-t}{2}\\
&=\sum_{s=1}^{n-4}\binom{n-1-s}{3}=\binom{n-1}{4}
\end{align*}
permutations in this case.

So assume that there is not a large letter between the greatest and the least of the small letters.  We consider further cases on $\pi'$, which we write as $\pi'=\alpha_1\beta\alpha_2$, where $\beta$ comprises the set of small letters.  First assume that one of $\alpha_1,\alpha_2$ is empty or that both are nonempty with $\min(\alpha_2)>\max(\alpha_1)$.  Let $a_1=m$, where $s+2 \leq m \leq n-1$.  For each $S$ with $a_1=m$, there are $n-m+1$ possible $\pi'$ upon choosing the length of (the interval) $\alpha_2$, which can be anywhere from $0$ to $n-m$ (note that once $S$ is specified, $\pi'$ is uniquely determined in this case by $|\alpha_2|$).  Also $a_s \geq 3$ implies that there are $\binom{m-3}{s-1}$ choices for the remaining elements of $S$.  Since all permutations formed in this manner are seen to be $T$-avoiding, we get
$$\sum_{s=1}^{n-3}\sum_{m=s+2}^{n-1}(n-m+1)\binom{m-3}{s-1}=\sum_{m=3}^{n-1}(n-m+1)\sum_{s=1}^{m-2}\binom{m-3}{s-1}=\sum_{m=3}^{n-1}(n-m+1)2^{m-3}$$
possibilities in this case.  Now assume that $\alpha_1$ and $\alpha_2$ are both nonempty, with $\min(\alpha_2)<\max(\alpha_1)$.  Then it must be the case that  $\min(\alpha_2)$ is the only letter in $\alpha_2$ that is less than $\max(\alpha_1)$, for if not, then there would be an occurrence of $4123$.  If $\max(\alpha_1)=n-t$, then $0 \leq t \leq n-m-2$ where $a_1=m$, with $s+2 \leq m \leq n-2$.  The first letter of $\alpha_2$ must then belong to $[m+1,n-t-1]$, and thus there are $n-m-t-1$ choices for it.  Considering all possible $s$, $m$ and $t$ yields
$$\sum_{s=1}^{n-4}\sum_{m=s+2}^{n-2}\sum_{t=0}^{n-m-2}(n-m-t-1)\binom{m-3}{s-1}=\sum_{s=1}^{n-4}\sum_{m=s+2}^{n-2}\binom{m-3}{s-1}\binom{n-m}{2}=\sum_{m=3}^{n-2}\binom{n-m}{2}2^{m-3}$$
additional permutations.  Thus, the last two cases combined give
$$\sum_{m=3}^{n-1}\left(\binom{n-m+1}{2}+1\right)2^{m-3}=2^{n-3}-1+\sum_{m=3}^{n-1}\binom{n-m+1}{2}2^{m-3}$$
possibilities in all.  Putting this together with the prior cases above implies \eqref{gnrec}.
\end{proof}

For any of the sequence above, put zero if $n$ is such that the corresponding set of permutations is empty.  For example, we have $d_n=0$ for $n \leq 3$. Combining cases I through IV above, and including permutations that start with the letter $n$, implies that the number of $T$-avoiding permutations of length $n$ is given by
\begin{equation}\label{anrel}
a_n=b_n+d_n+e_n+g_n+C_{n-1}, \qquad n \geq 2,
\end{equation}
with $a_0=a_1=1$.

We now  find a formula for the generating function $A(x)=\sum_{n\geq0}a_nx^n$. Define
$B(x)=\sum_{n\geq2}b_nx^n$, $D(x)=\sum_{n\geq4}d_nx^n$, $E(x)=\sum_{n\geq3}e_nx^n$ and $G(x)=\sum_{n\geq2}g_nx^n$. First we rewrite the recurrences in the preceding lemmas in terms of generating functions, by multiplying both sides by $x^n$ and summing over all possible $n$. By Lemmas \ref{b_nlem}--\ref{gnlem}, we obtain
\begin{align}
B(x)&=\frac{1}{1-x}\biggl(u(x)+\frac{x^3}{1-2x}-\frac{x^3}{(1-x)^2}+\frac{x^5}{(1-x)^5}-\frac{x^5}{(1-x)^4}\notag\\
&\quad+xq\left(\frac{x}{1-x},1-x\right)-xq(x,1)\biggr),\label{eqt125tb}\\
D(x)&=xB(x)+\frac{x^5}{(1-x)^3(1-2x)},\label{eqt125td}\\
(1-x)E(x)&=xD(x)+x^2C(x)-\left(x^2+x^3+\frac{2x^4}{1-2x}\right)+\frac{x}{1-x}D(x)-\frac{x^2}{1-x}B(x),\label{eqt125te}\\
G(x)&=\frac{x^2(1-5x+13x^2-18x^3+13x^4-6x^5+x^6)}{(1-x)^6(1-2x)},\label{eqt125tg}
\end{align}
respectively. By solving \eqref{eqt125tb}--\eqref{eqt125te}, along with use of \eqref{eqtqq} and Lemma \ref{ungf}, we obtain
\begin{align*}
B(x)&=(1-x)(C(x)-1)+\frac{x(3x^5-8x^4+13x^3-11x^2+5x-1)}{(1-x)^6},\\
D(x)&=x(1-x)(C(x)-1)-\frac{x^2(7x^6-22x^5+37x^4-36x^3+21x^2-7x+1)}{(1-x)^6(1-2x)},\\
E(x)&=\frac{x^2(2-x)}{1-x}(C(x)-1)-\frac{x^3(8x^6-29x^5+54x^4-57x^3+36x^2-13x+2)}{(1-x)^7(1-2x)}.
\end{align*}
By \eqref{anrel}, we have
$$A(x)=1+x+B(x)+D(x)+E(x)+G(x)+x(C(x)-1),$$
which, upon substituting the expressions just found for the generating functions $B(x)$, $D(x)$, $E(x)$ and $G(x)$, implies
the following result.
\begin{theorem}\label{th125a}
Let $T=\{1243,2341,4123\}$. Then
$$F_T(x)=\frac{1-9x+34x^2-70x^3+87x^4-65x^5+26x^6-5x^7}{(1-x)^7(1-2x)}+\frac{1}{1-x}\big(C(x)-1\big).$$
\end{theorem}

\subsection{Case 149: $\{1234,3412,4123\}$}

We enumerate directly $T$-avoiding permutations that have exactly two or three left-right maxima.

\subsubsection{The case of two l-r maxima}

We first count $T$-avoiding permutations that have two left-right maxima in which the first component is decreasing.  Let $\mathcal{E}_n$ denote the set of permutations of length $n$ avoiding $\{123,3412\}$ and let $e_n=|\mathcal{E}_n|$.

\begin{lemma}\label{2lrp1}
The number of $T$-avoiding permutations of length $n$ having two left-right maxima and of the form $\pi=a\pi'n\pi''$ where $\pi'$ is decreasing is given by
\begin{align}
g_n&=(3-n)2^{n-3}-1+\sum_{a=1}^{n-2}2^{a-1}e_{n-1-a}+\sum_{a=2}^{n-1}\left(2^{n-1-a}-1-\binom{n-a}{2}\right)(a-1)2^{a-2}\notag\\
&\quad+\sum_{i=0}^{n-2}\binom{n-1+i}{2i+1}+\sum_{a=2}^{n-3}\sum_{\ell=1}^{n-2-a}\sum_{i=0}^{a-1}2^{n-2-a-\ell}\binom{a-1}{i}\binom{\ell+i}{i}, \qquad n \geq 3.\label{2lrp1e1}
\end{align}
\end{lemma}
\begin{proof}
We first determine the number $e_n^*$ of members of $\mathcal{E}_n$ whose final letter is \emph{not} part of an occurrence of 231.  Let $\mathcal{E}_n^*$ denote the subset of $\mathcal{E}_n$ enumerated by $e_n^*$ and let $\mathcal{E}_{n,\ell}^*\subseteq\mathcal{E}_n^*$ consist of those members whose final decreasing run is of length $\ell$, with $e_{n,\ell}^*=|\mathcal{E}_{n,\ell}^*|$.  We define a bijection between $\mathcal{E}_{n,\ell}^*$ and $\mathcal{E}_{n-\ell}^*$ as follows where $1 \leq \ell \leq n-1$.  Suppose $\lambda \in \mathcal{E}_{n,\ell}^*$ and let $b$ denote the leftmost ascent bottom of $\lambda$.  First assume $b \geq 2$.  Then letters in $[b-1]$ decrease so as to avoid 3412, while letters greater than $b$ to the right of it must also decrease to avoid 123.  Since $b \geq 2$ and the final letter of $\lambda$ is not part of a 231, it follows that the final $\ell+1$ letters of $\lambda$ are $1,b+r+\ell,b+r+\ell-1,\ldots,b+r+1$ for some $r \geq 0$.  In this case, we delete from $\lambda$ the letters in $[b+r+1,b+r+\ell-1]$, together with $1$, and standardize the resulting permutation which is seen to result in a member $\lambda' \in \mathcal{E}_{n-\ell}^*$.  Note that all members of $\mathcal{E}_{n-\ell}^*$ in which the leftmost ascent top is not the final letter arise uniquely in this manner.  If $b=1$, then $\lambda$ consists of two decreasing runs, the second of which is of length $\ell$.  To obtain $\lambda'$ in this case, we delete the final $\ell-1$ letters of $\lambda$, together with the first letter of $\lambda$ if it exceeds the first letter of the second decreasing run; otherwise, we simply delete the final $\ell$ letters of $\lambda$ prior to standardizing.  Note that one obtains in this manner all members of $\mathcal{E}_{n-\ell}^*$ that consist of two decreasing runs in which the second run is of length one, together with the decreasing permutation $(n-\ell)\cdots 21$.  One may verify in each case that the mapping $\lambda \mapsto \lambda'$ is a bijection, as desired.  Since $e_{n,n}^*=1$ and $e_{n}^*=\sum_{\ell=1}^n e_{n,\ell}^*$, we have $e_n^*=1+\sum_{\ell=1}^{n-1}e_{n-\ell}^*$ for $n\geq 2$, with $e_1^*=1$, which implies $e_n^*=2^{n-1}$ and $e_{n,\ell}^*=2^{n-\ell-1}$ for $1 \leq \ell \leq n-1$.

We now count members of $S_n(T)$ that have two left-right maxima and are of the form $\pi=a\pi'n\pi''$ for some $a \geq 1$ where $\pi'$ is decreasing.  There are clearly $e_{n-2}$ possibilities if $a=1$ and $2^{n-2}$ possibilities if $a=n-1$, so assume $n \geq 4$ and $2 \leq a \leq n-2$. Let $L$ denote the subsequence of $\pi$ comprising the elements of $[a+1,n-1]$.  Then $L$ is $\{123,3412\}$-avoiding, with no letters in $[a-1]$ occurring between $n$ and the rightmost ascent bottom of $L$, for otherwise there is a 4123 of the form $na'\ell_1\ell_2$, where $a' \in [a-1]$ and $\ell_1,\ell_2$ comprise the rightmost ascent of $L$ (viewed as a permutation in its own respect when discussing ascents).  We consider cases on $L$.  First suppose $L$ (when standardized) belongs to $\mathcal{E}_{n-1-a,\ell}^*$ for some $\ell$.  Let $i$ denote the number of elements of $[a-1]$ in $\pi''$.  Since $\pi'$ is decreasing, there is no restriction on the choice of these elements.  There are $e_{n-1-a,\ell}^*$ possibilities for $L$ and $\binom{\ell+i}{i}$ ways in which to arrange the elements of $[a-1]$ occurring in $\pi''$ once $\pi'$ and $L$ are specified, since they must decrease (due to 3412) and occur to the right of the rightmost ascent bottom of $L$.  Considering all possible $a$, $\ell$ and $i$, and treating separately the case when $\ell=n-1-a$, gives
\begin{align*}
&\sum_{a=2}^{n-2}\sum_{\ell=1}^{n-1-a}\sum_{i=0}^{a-1}\binom{a-1}{i}\binom{\ell+i}{i}e_{n-1-a,\ell}^*\\
&=n-3+\sum_{i=1}^{n-3}\sum_{a=i+1}^{n-2}\binom{a-1}{i}\binom{n-1-a+i}{i}+\sum_{a=2}^{n-3}\sum_{\ell=1}^{n-2-a}\sum_{i=0}^{a-1}2^{n-2-a-\ell}\binom{a-1}{i}\binom{\ell+i}{i}\\
&=\sum_{i=0}^{n-2}\binom{n-1+i}{2i+1}-2^{n-2}-1+\sum_{a=2}^{n-3}\sum_{\ell=1}^{n-2-a}\sum_{i=0}^{a-1}2^{n-2-a-\ell}\binom{a-1}{i}\binom{\ell+i}{i}
\end{align*}
possible permutations and each is seen to be $T$-avoiding.  On the other hand, if $L$ does not belong to $\mathcal{E}_{n-1-a}^*$, then the letters from $[a-1]$ in $\pi''$ can always occur at the end (in descending order), which yields
$$\sum_{a=2}^{n-2}2^{a-1}(e_{n-1-a}-e_{n-1-a}^*) = \sum_{a=2}^{n-2}2^{a-1}e_{n-1-a}-(n-3)2^{n-3}$$
permutations in this case.

We now entertain the possibility of $x \in [a-1]$ occurring in a position of $\pi''$ other than beyond the rightmost letter of $L$ when $L$ does not belong to $\mathcal{E}_{n-1-a}^*$.    Note first that if the final letter $q$ of a $\{123,3412\}$-avoiding permutation $\rho$ is part of an occurrence of 231, then one can always take the letter corresponding to the ``3'' within the occurrence to be the leftmost ascent top $r$ of $\rho$.  This is clear if the leftmost ascent bottom $s$ of $\rho$ exceeds $q$.  On the other hand, if $s<q$, then since all letters greater than $s$ and to the right of $r$ decrease, the only possibility for the ``2'' in this case is some letter to the left of $r$, whence one may take $r$ to be the ``3'' (since all letters to the left of $r$ decrease).  Thus if $L$ contains more than one ascent, inserting $x$ anywhere to the right of the rightmost ascent bottom of $L$ other than beyond the rightmost letter of $L$  is seen to introduce an occurrence of 3412, where the ``4'' and the ``2'' correspond to the leftmost ascent top and rightmost letter of $L$, respectively.  Therefore, $L$ must contain exactly one ascent in order for $x$ to occur in a position prior to some letter in $L$.  In this case, $x$ can go either between the two letters comprising the ascent of $L$ or beyond the last letter of $L$.  Since at least one letter in $[a-1]$ must be  positioned prior to some letter in $L$ (in order to distinguish permutations arising in this case from those of previous cases), there are $\binom{a-1}{i}i$ ways in which to select and arrange the letters from $[a-1]$ in $\pi''$ where $i \geq 1$.  Furthermore, there are $\sum_{j=2}^{n-2-a}\left(\binom{n-1-a}{j}-(n-a-j)\right)$ possibilities for $L$, since the second (and final) decreasing run of $L$ cannot comprise an interval (for otherwise, $L$ would belong to $\mathcal{E}_{n-1-a}^*$).  Considering all possible $a$, $i$ and $j$ yields
$$\sum_{a=2}^{n-2}\sum_{i=1}^{a-1}\binom{a-1}{i}i\sum_{j=2}^{n-2-a}\left(\binom{n-1-a}{j}-(n-a-j)\right)=\sum_{a=2}^{n-1}\left(2^{n-1-a}-1-\binom{n-a}{2}\right)(a-1)2^{a-2}$$
additional permutations.  Combining this case with the other cases gives \eqref{2lrp1e1}.
\end{proof}

Let $ \mathcal{E}_{n,\ell}\subseteq\mathcal{E}_n$ consist of those members whose final decreasing run (i.e., the letters to the right of and including the rightmost ascent top) is of length $\ell$ for $1 \leq \ell \leq n-1$, with $\mathcal{E}_{n,n}=\{n(n-1)\cdots 1\}$.  For example, we have $\mathcal{E}_{4,2}=\{2143,3142,3241,4132,4231\}$ and $\mathcal{E}_{4,3}=\{1432,2431,3421\}$.   We will need the following explicit formula for $e_{n,\ell}=|\mathcal{E}_{n,\ell}|$.

\begin{lemma}\label{2lrp2}
If $n \geq 2$, then
\begin{equation}\label{2lrp2e1}
e_{n,\ell}=\sum_{m=0}^{n-1-\ell}e_{n-m,\ell}', \qquad 1 \leq \ell \leq n-1,
\end{equation}
where
$$e_{n,\ell}'=\binom{n-1}{\ell-1}+\sum_{b=2}^\ell\sum_{i=1}^{b-1}\sum_{r=\ell+1-i}^{n-1-b}\binom{r-2-\ell+b}{b-1-i}+\sum_{b=\ell+1}^{n-1}\sum_{i=1}^\ell\sum_{r=\ell+1-i}^{n-1-b}\binom{r-2-\ell+b}{b-1-i}$$
for $1 \leq \ell \leq n-2$, with $e_{n,n-1}'=n-1$ if $n \geq 2$.
\end{lemma}
\begin{proof}
Let $\mathcal{E}_{n,\ell}'\subseteq \mathcal{E}_{n,\ell}$ consist of those members whose leftmost ascent top is $n$, with $e_{n,\ell}'=|\mathcal{E}_{n,\ell}'|$.  Since the leftmost ascent top $q$ of any $\pi\in\mathcal{E}_n$ is the largest letter yet to be used, all elements of $[q+1,n]$ must appear as part of the initial decreasing run of $\pi$.  Since these letters are seen to be extraneous concerning the avoidance of 123 and 3412, they may be deleted which implies $e_{n,\ell}=\sum_{m=0}^{n-\ell-1}e_{n-m,\ell}'$.  To complete the proof, we establish the formula stated above for $e_{n,\ell}'$ where $1 \leq \ell \leq n-2$, the $\ell=n-1$ case being clear from the definitions.

Let $\pi \in \mathcal{E}_{n,\ell}'$ and $b$ be the leftmost ascent bottom of $\pi$.  If the letters to the right of $n$ within $\pi$ are decreasing, then $\pi=\pi'bn\pi''(b-1)\cdots 21$, where $\pi'$ and $\pi''$ are decreasing and $\pi''$ has length $\ell-b$, assuming $b \leq \ell$.  Considering all possible $b$ gives $\sum_{b=1}^\ell\binom{n-1-b}{\ell-b}=\binom{n-1}{\ell-1}$ permutations $\pi$.  So assume that there is at least one ascent to the right of $n$ within $\pi$.  Then $\pi=\pi'bn\pi''$, where $\pi'$ is decreasing and $\pi''$ can be expressed as
$$\pi''=\rho_b(b-1)\rho_{b-1}(b-2)\cdots\rho_21\rho_1,$$
where $\rho_i$ denotes a possibly empty sequence of letters.  Let $K$ denote the subsequence of $\pi$ comprising the letters in $\rho_b\cup\rho_{b-1}\cup\cdots\cup\rho_1$.  Then $K$ is decreasing so as to avoid 123 and nonempty, by the assumption on $\pi$.  First suppose $2 \leq b \leq \ell$ and let $i$ be the smallest index $j$ such that $\rho_j\neq \emptyset$.  Note that $1 \leq i \leq b-1$ since $i=b$ is disallowed, for otherwise $\pi''$ would be decreasing contrary to our assumption. Also, $\rho_i$ contains exactly $\ell+1-i$ letters since the final decreasing run of $\pi$ is to be of length $\ell$.  If $r$ denotes the number of letters in $K$, then $\ell+1-i \leq r \leq n-1-b$, with $K$ comprising the interval $[n-r,n-1]$ and the remaining elements of $[b+1,n-1]$ going in $\pi'$.  For if not, then there would be would an occurrence of 3412 of the form $xn(b-1)y$ for some $x,y \in [b+1,n-1]$ with $x>y$.  Note that the largest $r-(\ell+1-i)$ letters of $K$ may be distributed amongst $\rho_{b},\rho_{b-1},\ldots,\rho_{i+1}$ in $\binom{r-2-\ell+b}{b-1-i}$ ways.  Considering all possible $b$, $i$ and $r$ yields
 $$\sum_{b=2}^\ell\sum_{i=1}^{b-1}\sum_{r=\ell+1-i}^{n-1-b}\binom{r-2-\ell+b}{b-1-i}$$
 permutations $\pi$ and each is seen to belong to $\mathcal{E}_{n,\ell}'$.  If $\ell+1 \leq b \leq n-1$, then the index $i$ satisfies $1 \leq i \leq \ell$, for otherwise $\pi$ would have a final decreasing run of length strictly greater than $\ell$.  Then $r$ satisfies the same conditions as before and $K$ again comprises $[n-r,n-1]$.  Considering all $b$, $i$ and $r$ gives the second triple sum in the expression for $e_{n,\ell}'$ above and combining the previous cases yields the desired expression for the cardinality of $\mathcal{E}_{n,\ell}'$.
\end{proof}

Let $b_n$ denote the number of $T$-avoiding permutations of length $n$ having exactly two left-right maxima.  It is given explicitly as follows.

\begin{lemma}\label{2lrp3}
We have
\begin{align}
b_n&=(7-n)2^{n-3}-2+\sum_{a=1}^{n-2}2^{a-1}e_{n-1-a}+\sum_{a=2}^{n-3}\sum_{\ell=1}^{n-2-a}\sum_{i=0}^{a-1}2^{n-2-a-\ell}\binom{a-1}{i}\binom{\ell+i}{i}\notag\\
&\quad+\sum_{a=2}^{n-1}\left(2^{n-1-a}-1-\binom{n-a}{2}\right)(a-1)2^{a-2}+\sum_{a=3}^{n-1}\sum_{\ell=1}^{a-2}\sum_{j=0}^{\ell}\binom{n-1-a+j}{j}e_{a-1,\ell}, \quad n \geq 3,\label{2lrp3e1}
\end{align}
where $e_{n-1-a}$ and $e_{a-1,\ell}$ are as defined above.
\end{lemma}
\begin{proof}
We count permutations $\pi=a\pi'n\pi''$ having two left-right maxima where $\pi'$ is not decreasing, whence $a \geq 3$.  Then letters in $[a+1,n-1]$ must decrease, so as to avoid an occurrence of 1234.  Let $\alpha$ denote the subsequence of $\pi$ comprising the letters in $[a-1]$.  Then $\alpha$ avoids $\{123,3412\}$, and if $\alpha \in \mathcal{E}_{a-1,\ell}$ for some $\ell$ where $1 \leq \ell \leq a-2$, then at most $\ell$ letters of $\alpha$ can lie within $\pi''$ (or else there would be an occurrence of 3412 starting with $a,n$).  Suppose exactly $j$ letters of $\alpha$ lie within $\pi''$.  If $0 \leq j \leq \ell-1$, then there are $e_{a-1,\ell}$ possibilities for $\pi'$  and $\binom{n-1-a+j}{j}$ ways to arrange the letters in $\pi''$ once $\pi'$ is known since letters from both $[a-1]$ and $[a+1,n-1]$ are decreasing in $\pi''$, which uniquely determines $\pi$.  This yields
$$\sum_{a=3}^{n-1}\sum_{\ell=1}^{a-2}\sum_{j=0}^{\ell-1}\binom{n-1-a+j}{j}e_{a-1,\ell}$$
possible $\pi$ and one may verify that each such $\pi$ is $T$-avoiding.  If $j=\ell$, then the same reasoning applies except that now $\alpha$ must contain at least two ascents, for otherwise $\pi'$ would be decreasing contrary to our assumption.  Thus, there are $e_{a-1,\ell}-\binom{a-1}{\ell}+1$ possibilities for $\alpha$ in this case, which yields
$$\sum_{a=3}^{n-1}\sum_{\ell=1}^{a-2}\binom{n-1-a+\ell}{\ell}\left[e_{a-1,\ell}-\binom{a-1}{\ell}+1\right]$$
additional permutations.  Note that
\begin{align*}
&\sum_{a=3}^{n-1}\sum_{\ell=1}^{a-2}\binom{n-1-a+\ell}{\ell}\left[\binom{a-1}{\ell}-1\right]=\sum_{\ell=1}^{n-3}\sum_{a=\ell+1}^{n-1}\binom{n-1-a+\ell}{\ell}\left[\binom{a-1}{\ell}-1\right]\\
&=\sum_{\ell=1}^{n-3}\left[\binom{n-1+\ell}{2\ell+1}-\binom{n-1}{\ell+1}\right]=\sum_{\ell=0}^{n-2}\binom{n-1+\ell}{2\ell+1}-(2^{n-1}-1).
\end{align*}
The two previous cases taken together then imply that there are
$$2^{n-1}-1-\sum_{\ell=0}^{n-2}\binom{n-1+\ell}{2\ell+1}+\sum_{a=3}^{n-1}\sum_{\ell=1}^{a-2}\sum_{j=0}^{\ell}\binom{n-1-a+j}{j}e_{a-1,\ell}$$
permutations $\pi$ of the stated form above where $\pi'$ is not decreasing.  Combining this expression with the one for $g_n$ in Lemma \ref{2lrp1} above gives \eqref{2lrp3e1}.
\end{proof}

\subsubsection{Three l-r maxima}

Let $d_n$ denote the number of $T$-avoiding permutations of length $n$ having three left-right maxima.  The sequence $d_n$ may be expressed explicitly as follows.

\begin{lemma}\label{3lrp1}
We have
\begin{equation}\label{3lrp1e1}
d_n=\sum_{a=1}^{n-2}\sum_{b=a+1}^{n-1}d_n(a,b), \qquad n \geq 3,
\end{equation}
where
\begin{align*}
d_n(a,b)&=\left(2^{a-1}+\binom{a}{2}\right)\left(2^{b-1-a}-b+a\right)+\left((a-1)2^{a-2}-\binom{a}{2}\right)\left(2^{b-1-a}-1\right)\\
&\quad+\sum_{j=0}^{a-1}\left(\binom{a-1}{j}-1\right)\binom{n-b+j}{j}+\sum_{j=0}^{a-1}\sum_{r=0}^{b-2-a}\left(\binom{a-1}{j}-1\right)\binom{n-3-r-j}{n-1-b}\\
&\quad+\sum_{j=0}^{a-1}\sum_{i=0}^j\sum_{r=0}^{b-1-a}\binom{i+r}{i}\binom{n-2+j-i-a-r}{n-1-b}, \qquad 1 \leq a <b \leq n-1.
\end{align*}
\end{lemma}
\begin{proof}
Let $\mathcal{D}_n(a,b)$ denote the set of $T$-avoiding permutations whose left-right maxima are $a,b,n$, where $1 \leq a <b \leq n-1$.  We will show that $|\mathcal{D}_n(a,b)|$ is given by $d_n(a,b)$ above, which implies \eqref{3lrp1e1}.  Let $\pi=a\pi'b\pi''n\pi''' \in \mathcal{D}_n(a,b)$.  Then all letters in $\pi'$ and $[b+1,n-1]$ decrease so as to avoid 1234 as well as any letters from $[a+1,b-1]$ in $\pi''$.  Also, any letters of $[b-1]$ within $\pi'''$ and any letters of $[a-1]$ occurring to the right of $b$ must decrease, for otherwise there would be an occurrence of 3412 in either case.  We now consider cases based off of whether or not the subsequence $\alpha$ of $\pi$ comprising $[a-1]$ is decreasing.  First assume that it is.  Suppose for now that the subsequence $\beta$ of $\pi$ comprising $[a+1,b-1]$ is not decreasing and thus contains exactly one ascent.  Then all elements of $[b-1]$ within $\pi'''$ in this case must occur at the end of $\pi'''$ (necessarily in decreasing order), for otherwise there is a 1234 of the form $axy(b+1)$ for some $x,y \in [a+1,b-1]$ with $x<y$.  Also, any elements of $[a-1]$ in $\pi''$ must occur at the end of $\pi''$, for otherwise there is a 4123 of the form $bzxy$ where $z \in [a-1]$.

Let $j$ denote the number of elements of $[a-1]$ occurring to the right of $b$ within $\pi$ and $i$, the number of elements of $[a-1]$ in $\pi''$, where $0 \leq j \leq a-1$ and $0 \leq i \leq j$.  Let $r$ denote the number of elements of $[a+1,b-1]$ in $\pi''$.  Then there are $\binom{n-1-a}{r}-1$ ways in which to select these elements as $\beta$ is to contain an ascent.  Since the positions of the elements of $[a-1]$ within $\pi$ are determined by the prior observations once $i$ and $j$ are specified in this case, we get
\begin{align*}
\sum_{j=0}^{a-1}\sum_{i=0}^j\sum_{r=0}^{b-1-a}\left[\binom{b-1-a}{r}-1\right]&=(2^{b-1-a}-b+a)\sum_{j=0}^{a-1}\sum_{i=0}^j1=\binom{a+1}{2}(2^{b-1-a}-b+a)
\end{align*}
possible permutations.  One may verify (and also in the subsequent cases) that the permutations so obtained are $T$-avoiding.  On the other hand, if $\beta$ is decreasing, then we get $\binom{i+r}{i}$ possibilities for $\pi''$ since letters from $[a-1]$ and $[a+1,b-1]$ within $\pi''$ are decreasing, and
$$\binom{(n-1-b)+(j-i)+(b-1-a-r)}{n-1-b}=\binom{n-2+j-i-a-r}{n-1-b}$$
possibilities for $\pi'''$ since letters from $[b-1]$ and $[b+1,n-1]$ within $\pi'''$ are decreasing, where $i$, $j$ and $r$ are as before.  Considering all $i$, $j$ and $r$ gives the triple sum found in the expression for $d_n(a,b)$ above.

Now assume that $\alpha$ contains an ascent. Note that since the letters of $[a-1]$ within $\pi$ occurring to the left and to the right of $b$ form decreasing subsequences, there is exactly one ascent in $\alpha$ and it involves the last letter of $\pi'$ and either the leftmost letter of $[a-1]$ within $\pi''$ if $\pi''$ is nonempty or the leftmost letter of $[a-1]$ within $\pi'''$ if $\pi''$ is empty.  Any elements of $[a-1]$ within $\pi''$ must occur at the end of $\pi''$ in this case so as to avoid 1234.  Also, if $\pi''$ contains at least one member of $[a-1]$ and $\pi'''$ at least one member of $[a+1,b-1]$, then all elements of $[b-1]$ in $\pi'''$ must come at the end, for otherwise there is an occurrence of 1234 of the form $uvw(b+1)$, where $u,v \in [a-1]$ with $u<v$ and $w \in [a+1,b-1]$.  We now consider cases on $[a+1,b-1]$.  If $[a+1,b-1]$ is confined to $\pi''$, then there are $\binom{a-1}{j}-1$ possibilities for the letters of $\alpha$ in $\pi'$ and $\sum_{i=0}^j\binom{n-1-b+j-i}{n-1-b}=\binom{n-b+j}{n-b}$ possibilities for $\pi'''$ once $\pi'$ is determined.  Considering all $j$
gives $\sum_{j=0}^{a-1}\left(\binom{a-1}{j}-1\right)\binom{n-b+j}{j}$ permutations.

So assume $\pi'''\cap[a+1,b-1]\neq\emptyset$.  If $\pi''\cap[a-1]\neq \emptyset$, then members of $[a-1]$ in $\pi''$ must occur at the end and the same holds for members of $[b-1]$ in $\pi'''$.  Thus, we get
\begin{align*}
\sum_{j=0}^{a-1}\left(\binom{a-1}{j}-1\right)\sum_{i=1}^j\sum_{r=0}^{b-2-a}\binom{b-1-a}{r}&=\left(2^{b-1-a}-1\right)\sum_{j=0}^{a-1}\left(\binom{a-1}{j}-1\right)j\\
&=\left(2^{b-1-a}-1\right)\left((a-1)2^{a-2}-\binom{a}{2}\right)
\end{align*}
members of $\mathcal{D}_n(a,b)$ in this case. If $\pi''\cap [a-1]=\emptyset$, then consider whether or not $\beta$ contains an ascent. If it does, then all members of $[b-1]$ in $\pi'''$ must occur at the end (due to 1234), which gives
$$\sum_{j=0}^{a-1}\left(\binom{a-1}{j}-1\right)\sum_{r=0}^{b-1-a}\left(\binom{b-1-a}{r}-1\right)=\left(2^{a-1}-a\right)\left(2^{b-1-a}-b+a\right)$$
possibilities.  If it does not, then there are
$$\binom{(n-1-b)+(b-1-a-r)+(a-1-j)}{n-1-b}=\binom{n-3-r-j}{n-1-b}$$
possibilities for $\pi'''$ once $r$ and $\pi'$ are specified.  Note that $r \leq b-2-a$ since $\pi'''\cap[a+1,b-1]\neq\emptyset$.  Thus, we get $\sum_{j=0}^{a-1}\sum_{r=0}^{b-2-a}\left(\binom{a-1}{j}-1\right)\binom{n-3-r-j}{n-1-b}$ additional members of $\mathcal{D}_n(a,b)$.  Combining all of the previous cases implies that the cardinality of $\mathcal{D}_n(a,b)$ is given by $d_n(a,b)$, as desired.
\end{proof}

Upon including permutations that start with $n$, we have $a_n=b_n+d_n+e_{n-1}$ for $n \geq 2$, with $a_0=a_1=1$.

To obtain an explicit formula for the generating function for the number of $T$-avoiders of length $n$, we define $A(x)=\sum_{n\geq0}a_nx^n$, $B(x)=\sum_{n\geq2}b_nx^n$, $D(x)=\sum_{n\geq2}d_nx^n$ and $E(x)=\sum_{n\geq0}e_nx^n$. It is well known that the generating function for the number of $\{123,3412\}$-avoiders of length $n$ is given by
$$E(x)=1+\frac{x(1-4x+7x^2-5x^3+2x^4)}{(1-x)^4(1-2x)}.$$
Let $E(x,y)=1+\sum_{n\geq 1}\sum_{\ell=1}^ne_{n,\ell}x^ny^\ell$.  Using Lemma \ref{2lrp2}, one can show
\small\begin{align*}
E&(x,y)=1+\sum_{n\geq1}\sum_{\ell=1}^ne_{n,\ell}x^ny^\ell\\
&=1+\frac{xy((2x^5-5x^4+4x^3-x^2)y^2+(2x^5-6x^4+11x^3-8x^2+2x)y-x^4+4x^3-6x^2+4x-1)}{(1-x)^2(1-2x)(1-xy)^2(xy+x-1)}.
\end{align*}\normalsize
Note that $E(x,1)=E(x)$. Multiplying both sides of \eqref{2lrp3e1} by $x^n$, and summing over all $n\geq3$, yields
\begin{align*}
B(x)&=x^2+\sum_{n\geq3}((7-n)2^{n-3}-2)x^n+\sum_{n\geq3}\sum_{a=1}^{n-2}2^{a-1}e_{n-1-a}x^n\\
&\quad+\sum_{n\geq3}\sum_{a=2}^{n-3}\sum_{\ell=1}^{n-2-a}\sum_{i=0}^{a-1}2^{n-2-a-\ell}\binom{a-1}{i}\binom{\ell+i}{i}x^n\\
&\quad+\sum_{n\geq3}\sum_{a=2}^{n-1}\left(2^{n-1-a}-1-\binom{n-a}{2}\right)(a-1)2^{a-2}x^n\\
&\quad+\sum_{n\geq3}\sum_{a=3}^{n-1}\sum_{\ell=1}^{a-2}\sum_{j=0}^{\ell}\binom{n-1-a+j}{j}e_{a-1,\ell}x^n\\
&=x^2+\frac{2x^3(1-3x+x^2)}{(1-x)(1-2x)^2}+\frac{x^3(1-4x+7x^2-5x^3+2x^4)}{(1-x)^4(1-2x)^2}\\
&\quad+\frac{x^5(3-6x+2x^2)}{(1-x)(1-2x)^2(1-3x+x^2)}+\frac{x^6}{(1-x)^3(1-2x)^3}\\
&\quad+\frac{x^4(2-13x+34x^2-46x^3+31x^4-7x^5)}{(1-x)^4(1-2x)^3(1-3x+x^2)}\\
&=\frac{x^2(1-10x+44x^2-108x^3+159x^4-144x^5+74x^6-14x^7)}{(1-x)^4(1-2x)^3(1-3x+x^2)}.
\end{align*}
Multiplying both sides of \eqref{3lrp1e1} by $x^n$, and summing over all $n\geq3$, yields
\begin{align*}
D(x)&=\sum_{n\geq3}\sum_{a=1}^{n-2}\sum_{b=a+1}^{n-1}\left(2^{a-1}+\binom{a}{2}\right)\left(2^{b-1-a}-b+a\right)x^n\\
&\quad+\sum_{n\geq3}\sum_{a=1}^{n-2}\sum_{b=a+1}^{n-1}\left((a-1)2^{a-2}-\binom{a}{2}\right)\left(2^{b-1-a}-1\right)x^n\\
&\quad+\sum_{n\geq3}\sum_{a=1}^{n-2}\sum_{b=a+1}^{n-1}\sum_{j=0}^{a-1}\left(\binom{a-1}{j}-1\right)\binom{n-b+j}{j}x^n\\
&\quad+\sum_{n\geq3}\sum_{a=1}^{n-2}\sum_{b=a+1}^{n-1}\sum_{j=0}^{a-1}\sum_{r=0}^{b-2-a}\left(\binom{a-1}{j}-1\right)\binom{n-3-r-j}{n-1-b}x^n\\
&\quad+\sum_{n\geq3}\sum_{a=1}^{n-2}\sum_{b=a+1}^{n-1}\sum_{j=0}^{a-1}\sum_{i=0}^j\sum_{r=0}^{b-1-a}\binom{i+r}{i}\binom{n-2+j-i-a-r}{n-1-b}x^n\quad\qquad\qquad\qquad
\end{align*}
\begin{align*}
&=\frac{x^5(1-2x+x^2-x^3)}{(1-x)^6(1-2x)^2}+\frac{x^6(1-x-x^2)}{(1-x)^5(1-2x)^3}+\frac{x^5(2-x)}{(1-x)^3(1-2x)(1-3x+x^2)}\\
&\quad+\frac{x^6}{(1-x)^2(1-2x)^2(1-3x+x^2)}+\frac{x^3}{(1-2x)^3}\\
&=\frac{x^3(1-9x+37x^2-91x^3+142x^4-141x^5+90x^6-36x^7+6x^8)}{(1-x)^6(1-2x)^3(1-3x+x^2)}.
\end{align*}
By the fact $a_n=b_n+d_n+e_{n-1}$ for $n \geq 2$, with $a_0=a_1=1$, we have
$$A(x)=1+B(x)+D(x)+xE(x),$$
which leads to the following result.

\begin{theorem}\label{th149a}
Let $T=\{1234,3412,4123\}$. Then
$$F_T(x)=\frac{1-14x+87x^2-315x^3+736x^4-1161x^5+1253x^6-918x^7+446x^8-134x^9+18x^{10}}{(1-x)^6(1-2x)^3(1-3x+x^2)}.$$
\end{theorem}
\vspace*{3mm}

\subsection{Case 185: $\{1234,2341,4123\}$}

We follow a similar pattern of proof as in Case 149 above and use the same notation, letting $u_n$, $b_n$, $d_n$, $e_n$ and $g_n$ denote the same subsets as in that case, but with $1234$ in place of $1243$.  From the proof of Lemma \ref{unlem} above, we see that the sequence $u_n$ is the same as before since avoiding 1234 is logically equivalent to avoiding 1243 in this case.  We now proceed with the various cases.

\subsubsection{Case I}

Let $b_n$ denote the number of $T$-avoiding permutations of length $n$ whose leftmost ascent is of the form $a,n$ for some $2 \leq a \leq n-1$.  We have the following recurrence relation for $b_n$.

\begin{lemma}\label{bnl}
If $n \geq 3$, then
\begin{equation}\label{bnle1}
b_n=b_{n-1}+u_n+(n-3)2^{n-4}+\binom{n-2}{5}-\binom{n-2}{2}+\sum_{j=0}^{n-4}\sum_{i=1}^{n-3-j}\binom{n-i-2}{j+1}C_{n-2-j,i},
\end{equation}
with $b_2=0$, where $u_n$ given by \eqref{unleme1} above.
\end{lemma}
\begin{proof}
The recurrence is clear for $n=3$, so assume $n \geq 4$.  Let $\pi \in S_n(T)$ be of the form enumerated by $b_n$ have leftmost ascent bottom $a \geq 2$.  If the first letter of $\pi$ is $n-1$, then there are $u_n$ possibilities, by definition, so assume henceforth that the first letter is $\leq$ $n-2$.
Let $x$ denote the letter that directly follows $n$ within $\pi$.  If $x=n-1$, then there are $b_{n-1}$ possibilities.  If $x \in [a-1]$, then reasoning in this case as in the proof of Lemma \ref{b_nlem} above again gives $\sum_{j=0}^{n-4}\sum_{i=1}^{n-3-j}\binom{n-i-2}{j+1}C_{n-2-j,i}$ possibilities.  Now assume $a+1\leq x \leq n-2$.  Note first that all letters between $n$ and $n-1$ must decrease in order to avoid 4123.  Also, no letter in $[a-1]$ can occur to the right of $n-1$, for otherwise there is a 2341 of the form $ax(n-1)a'$ for some $a'\in[a-1]$, whence letters in $[a-1]$ form a decreasing subsequence.  Thus $a \geq 2$ implies all letters to the right of $n-1$ must decrease in order to avoid 4123.  Hence, we have
$$\pi=\alpha an\beta (a-1)\cdots 1(n-1)\gamma,$$
where $\alpha$, $\beta$ and $\gamma$ are decreasing with $\alpha$ and $\gamma$ possibly empty.  We now consider cases on $\alpha$.  If $\alpha=\emptyset$, then it is seen that there are no further restrictions on $\beta$ and $\gamma$ and we get $\sum_{a=2}^{n-2}(2^{n-2-a}-1)=2^{n-3}-n+2$ possible permutations in this case.

So assume $\alpha\neq\emptyset$ and we consider first the case when some letter of $\gamma$ is less than $d=\max(\alpha)$.  Let $c$ denote the largest such letter of $\gamma$.  First observe that all letters in $[a+1,c-1]$ must belong to $\alpha$ or $\gamma$, for if not, then there exists a 4123 of the form $dauc$ for some $u \in [a+1,c-1]$ in $\beta$.  Also, all letters in $[d+1,n-2]$ must belong to $\gamma$, for otherwise there is a 2341 of the form $dd'(n-1)c$ for some $d' \in [d+1,n-2]$ in $\beta$.  Since $\beta$ is nonempty but does not contain members of either $[a+1,c]$ or $[d,n-2]$, then it must be the case that $d \geq c+2$, whence $a+1\leq c \leq n-4$ and $c+2\leq d \leq n-2$.  Note further that $\beta$ comprises $[c+1,s]$ for some $s \in [c+1,d-1]$.  For if not, and $s_1,s_2 \in [c+1,d-1]$ with $s_1<s_2$, $s_1 \in \alpha$ and $s_2 \in \beta$, then $s_1s_2(n-1)c$ is a 2341.  Finally, the elements of $[a+1,c-1]$ belonging to $\gamma$ must comprise $[t,c-1]$ for some $t \in [a+1,c]$.  For if not, and $t_1,t_2 \in [a+1,c-1]$ with $t_1>t_2$, $t_1 \in \alpha$ and $t_2 \in \gamma$, then $t_1(c+1)(n-1)t_2$ is a 2341.  Thus, we have shown in this case that $\alpha$, $\beta$ and $\gamma$ comprise the sets $[a+1,t-1]\cup[s+1,d]$, $[c+1,s]$ and $[t,c]\cup[d+1,n-2]$, respectively.  One may verify that the corresponding $\pi$ indeed avoids $T$.  Given $a$, $c$ and $d$, note that such $\pi$ are uniquely determined by $s$ and $t$ and that there are $d-c-1$ choices for $s$ and $c-a$ choices for $t$.  Summing over all $a$, $c$ and $d$ then gives
\begin{align*}
&\sum_{a=2}^{n-5}\sum_{c=a+1}^{n-4}\sum_{d=c+2}^{n-2}(c-a)(d-c-1)=\sum_{a=2}^{n-5}\sum_{c=a+1}^{n-4}(c-a)\binom{n-2-c}{2}=\sum_{c=3}^{n-4}\binom{n-2-c}{2}\sum_{a=2}^{c-1}(c-a)\\
&=\sum_{c=3}^{n-4}\binom{n-2-c}{2}\binom{c-1}{2}=\binom{n-2}{5}
\end{align*}
possible permutations.

Finally, assume $\alpha \neq \emptyset$ and that all letters in $\gamma$ are greater than $d$.  Then all elements of $[a+1,d-1]$ must belong to $\alpha$ or $\beta$ and all elements of $[d+1,n-2]$ belong to $\beta$ or $\gamma$.  Since $\beta \neq \emptyset$, it follows that there are $2^{n-3-a}-1$ ways in which to arrange the members of $[a+1,n-2]$ once $a$ and $d$ are specified.  Furthermore, each is seen to give rise to a permutation that avoids $T$.  Considering all possible $a$ and $d$ then yields
\begin{align*}
&\sum_{a=2}^{n-3}\sum_{d=a+1}^{n-2}(2^{n-3-a}-1)=\sum_{a=2}^{n-3}(2^{n-3-a}-1)(n-2-a)=\sum_{a=1}^{n-4}a2^{a-1}-\binom{n-3}{2}\\
&=(n-5)2^{n-4}+1-\binom{n-3}{2}
\end{align*}
additional permutations.  Combining this case with the prior gives \eqref{bnle1}.
\end{proof}

\subsubsection{Case II}

We have the following formula for $d_n$ in terms of $b_n$.

\begin{lemma}\label{dnl}
If $n \geq 5$, then
\begin{equation}\label{dnle1}
d_n=d_{n-1}+b_{n-1}-b_{n-2}+\binom{n-3}{2}-\binom{n-3}{5}+\sum_{a=3}^{n-3}\sum_{\ell=0}^{n-3-a}\sum_{m=1}^{n-2-a-\ell}\binom{n-5-\ell-m}{a-3}m,\\
\end{equation}
with $d_4=1$, where $b_n$ is given by \eqref{bnle1} above.
\end{lemma}
\begin{proof}
We proceed as in the proof of Lemma \ref{dnlem} above and let $\mathcal{B}_m$, $\mathcal{B}_m'$ and $\mathcal{D}_m$ have the same meaning as before but with the pattern 1234 in place of 1243.  We form members of $\mathcal{D}_n$ by inserting $n$ somewhere to the right of $n-1$ within a member of $\mathcal{B}_{n-1}$, where $n \geq 5$.  First note that there are $d_{n-1}$ possibilities if the letter $n-2$ is to directly follow $n-1$ within a member of $\mathcal{B}_{n-1}$.  So consider inserting $n$ into members $\pi \in \mathcal{B}_{n-1}'$ without introducing an occurrence of a pattern in $T$.  We show in every case except one that this is possible.  Note that if $n-2$ is the first letter of $\pi$, then $n$ can only be added as the final letter due to 2341, so assume henceforth that $\pi$ does not start with $n-2$.  Write $\pi=\pi'a(n-1)\pi''$, where $a \geq 2$ is the leftmost ascent bottom and let $y$ be the first letter of $\pi''$.  If $1\leq y \leq a-1$, then it is seen that $n$ can always be added to the end of $\pi$.

Now assume $a+1 \leq y \leq n-3$.  In this case, then it is only possible to insert $n$ in the position between $1$ and $n-2$.  For if $n$ is inserted somewhere between $n-1$ and $1$, then there is a 2341 as seen with $a(n-1)n1$, while if $n$ is inserted to the right of $n-1$ (including at the very end), then $ay(n-2)n$ is a 1234.  If $\pi'$ is empty or if $\pi'$ is nonempty and there are no letters in $[a+1,d-1]$ to the right of $n-2$ where $d=\max(\pi')$, then it is seen from the proof of Lemma \ref{bnl} that inserting $n$ between $1$ and $n-2$ within $\pi$ does not introduce an occurrence of a pattern in $T$.  On the other hand, if $\pi'$ is nonempty and some member of $[a+1,d-1]$ occurs to the right of $n-2$, then one cannot insert $n$ between $1$ and $n-2$ in this case as doing so introduces a 2341 of the form $d(n-1)nd'$ for some $d'\in[a+1,d-1]$.  Upon subtracting this last case, which pertains to $\binom{n-3}{5}$ members of $\mathcal{B}_{n-1}'$, one sees that there are $b_{n-1}-b_{n-2}-\binom{n-3}{5}$ members of $\mathcal{D}_n$ that can be obtained  either by adding $n$ to the end of $\pi$ (if $\pi$ starts with $n-2$ or does not but has $y \in [a-1]$) or by inserting $n$ between $1$ and $n-2$ (if $\pi$ does not start with $n-2$ and has $a+1 \leq y \leq n-3$, where either $\pi'=\emptyset$ or $\pi'\neq \emptyset$ and no member of $[a+1,d-1]$ occurs to the right of $n-2$).  Finally, from the proof of Lemma \ref{dnlem}, one gets
$$\binom{n-3}{2}+\sum_{a=3}^{n-3}\sum_{\ell=0}^{n-3-a}\sum_{m=1}^{n-2-a-\ell}\binom{n-5-\ell-m}{a-3}m$$
additional members of $\mathcal{D}_n$ in the case when $y=a-1$ and the last letter of $\pi$ is greater than $a$, upon inserting $n$ somewhere to the right of $1$, but not at the very end.  Combining all of the prior cases gives \eqref{dnle1}.
\end{proof}

\subsubsection{Case III} We now seek a formula for $e_n$.  Note that $e_4=0$ and $e_5=2$, the enumerated permutations being $23154$ and $42315$.
To determine $e_n$, we refine it as follows.  Let $\mathcal{E}_n$ denote the subset of $S_n(T)$ enumerated by $e_n$.  Given $\pi \in \mathcal{E}_n$ with leftmost ascent $a,b$ where $2 \leq a < b \leq n-2$, let $V$ denote the subset of $[a+1,b-1]$ occurring to the right of $b$ within $\pi$.  At times $V$ will also refer to the corresponding subsequence of $\pi$ comprising the elements of this set.  Let $v_n$ denote the number of $\pi\in \mathcal{E}_n$ in which members of $V$ form a decreasing subsequence with at least one member of $V$ to the right of $n$.  For example, if $n=6$, then $v_n=3$, the enumerated permutations being $241635$, $241653$ and $524163$. Note that $v_n=0$ if $n \leq 5$ since $a<b-1$ is required.  We have the following explicit formula for $v_n$.

\begin{lemma}\label{vnform}
If $n \geq 6$, then
\begin{equation}\label{vne1}
v_n=2^{n-2}-\binom{n-1}{3}-n+1+\sum_{a=2}^{n-4}\sum_{\ell=1}^{n-3-a}\left[\binom{n}{a+\ell+2}-\binom{n-1-\ell}{a+1}\right].
\end{equation}
\end{lemma}
\begin{proof}
Let $\mathcal{V}_n$ denote the subset of $\mathcal{E}_n$ enumerated by $v_n$.  Let $\pi \in \mathcal{V}_n$ have leftmost ascent $a,b$, where $2 \leq a \leq n-4$ and $a+2 \leq b \leq n-2$,  and let $R$ denote the subset of $[a+1,b-1]$ occurring to the left of $a$ within $\pi$.  Then $R$ must be of the form $[a+1,a+r]$ for some $0 \leq r \leq b-2-a$.  To see this, suppose it is not the case and let $x$ be the largest element of $R$ and $y$ be the smallest element of $[a+1,b-1]-R$.  By the assumption on $V$ that it be decreasing, $y$ is the rightmost element of $[a+1,b-1]$ within $\pi$ and hence occurs to the right of $n$.  But then $xbny$ is a 2341, which is impossible.  Furthermore, note that no element of $[a-1]$ can occur to the right of $n$ within $\pi$, for otherwise there is a 2341 of the form $abnz$ for some $z \in [a-1]$.  Thus, the subsequence of $\pi$ comprising $[a-1]$ must decrease, for otherwise there is a 4123 of the form $ba_1a_2v$ where $1 \leq a_1<a_2\leq a-1$ and $v \in V$.  Therefore, we have shown that $\pi$ contains a subsequence  of the form $a+r,\ldots,a+1,a,b,a-1,\ldots,1,n$.  Note that members of $[b+1,n-1]$ cannot occur between $b$ and $n$ (due to 1234), and thus must occur prior to $a+r$ or to the right of $n$, in either case as a decreasing subsequence.  Finally, members of $[b+1,n-1]$ occurring to the right of $n$ must comprise a set of the form $[n-\ell,n-1]$ for some $0 \leq \ell \leq n-1-b$, for otherwise there would be a 4123 as witnessed by $b_1abb_2$ where $b_1,b_2 \in [b+1,n-1]$.

Thus, we have the subsequence
$$n-\ell-1,\ldots,b+1,a+r,\ldots,a+1,a,b,a-1,\ldots,1,n,n-1,\ldots,n-\ell,$$
into which we insert the elements of $V=[a+r+1,b-1]$ to form members of $\mathcal{V}_n$.  One may verify that the letters of $V$ may be inserted in positions directly following any member of $[a-1]\cup\{b\}$ or $[n-\ell,n]$ without introducing a pattern in $T$. If $m$ denotes $|V|$, then $m=b-1-a-r$ and thus $1 \leq m \leq b-1-a$.  Since the subsequence consisting of elements of $V$ must decrease and contain at least one letter to the right of $n$, there are
$\binom{a+\ell+m}{m}-\binom{a-1+m}{m}$ ways in which to insert the elements of $V$, by subtraction.  Considering all possible $a$, $b$, $\ell$ and $m$ implies
$$v_n=\sum_{a=2}^{n-4}\sum_{b=a+2}^{n-2}\sum_{\ell=0}^{n-1-b}\sum_{m=1}^{b-1-a}\left[\binom{a+\ell+m}{m}-\binom{a-1+m}{m}\right], \qquad n \geq 6.$$
Interchanging summation, and treating separately the $\ell=0$ case, gives
\begin{align*}
v_n&=\sum_{a=2}^{n-4}\sum_{b=a+2}^{n-2}\sum_{\ell=0}^{n-1-b}\left[\binom{b+\ell}{a+\ell+1}-\binom{b-1}{a}\right]\\
&=\sum_{a=2}^{n-4}\sum_{b=a+2}^{n-2}\left[\binom{b}{a+1}-\binom{b-1}{a}\right]+\sum_{a=2}^{n-4}\sum_{\ell=1}^{n-3-a}\sum_{b=a+2}^{n-1-\ell}\left[\binom{b+\ell}{a+\ell+1}-\binom{b-1}{a}\right]\\
&=\sum_{a=2}^{n-4}\sum_{b=a+2}^{n-2}\binom{b-1}{a+1}+\sum_{a=2}^{n-4}\sum_{\ell=1}^{n-3-a}\left[\binom{n}{a+\ell+2}-\binom{n-1-\ell}{a+1}\right]\\
&=\sum_{a=2}^{n-4}\binom{n-2}{a+2}+\sum_{a=2}^{n-4}\sum_{\ell=1}^{n-3-a}\left[\binom{n}{a+\ell+2}-\binom{n-1-\ell}{a+1}\right],
\end{align*}
which implies \eqref{vne1}.
\end{proof}

We now count the remaining members of $\mathcal{E}_n$ for which some letter in $V$ lies to the right of $n$.  Note that if $V$ is not decreasing, then there must be at least one letter in $V$ to the right of $n$, for otherwise there is a 1234 of the form $ab_1b_2n$ for some $b_1, b_2 \in [a+1,b-1]$.
Let $w_n$ denote the number of members of $\mathcal{E}_n$ in which the subsequence $V$ is not decreasing.  For example, if $n=7$, then $w_7=1$, the enumerated permutation being 2531764.  Note that $w_n=0$ if $n<7$ since $|V|\geq 2$ implies $a<b-2$.  We have the following formula for $w_n$.

\begin{lemma}\label{wnle}
If $n \geq 7$, then
\begin{align}
w_n&=2^{n-2}-1-\binom{n-1}{2}-\binom{n-1}{4}+4\binom{n-6}{2}+\binom{n-5}{3}
+\sum_{a=2}^{n-6}\sum_{b=a+4}^{n-2}(b-5-a)2^{b-2-a}.
\end{align}
\end{lemma}
\begin{proof}
Let $\mathcal{W}_n$ denote the subset of $\mathcal{E}_n$ enumerated by $w_n$.  Suppose $\pi\in \mathcal{W}_n$ has leftmost ascent $a,b$, where $2 \leq a \leq n-5$ and $a+3 \leq b \leq n-2$.  First observe that since $V$ is not decreasing, all of the letters in $[b+1,n-1]$ must occur to the right of $n$ (and decrease), for otherwise there would be a 4123 of the form $cau_1u_2$, where $c \in [b+1,n-1]$ and $u_1,u_2\in V$ with $u_1<u_2$.  Denote the $V$ subsequence of $\pi=\pi_1\pi_2\cdots \pi_n$ by $\pi_{i_1}\pi_{i_2}\cdots \pi_{i_m}$, where $m=|V|$.  Let $j$ be the smallest index in $[m-1]$ such that $\pi_{i_j}<\pi_{i_{j+1}}$.  Then $y=\pi_{i_{j+1}}$ must occur to the right of the letter $b+1$ in $\pi$, for otherwise there is a 1234 as seen with $axy(b+1)$, where $x=\pi_{i_{j}}$. Also, letters in $[a-1]$ must occur between $b$ and $n$ (due to 2341).  Let $z$ be the leftmost letter in $\pi$ belonging to $[a-1]$.  Then all letters in $V$ to the right of $z$ must decrease, for otherwise there is a 4123 of the form $bzv_1v_2$ for some $v_1,v_2\in V$.  In particular, all letters in $V$ to the right of $b+1$ must decrease, which implies $y$ directly follows $b+1$.  Also, $x$ cannot occur anywhere to the right of $z$ since $x<y$, which means that $x$ must occur between $b$ and $z$.  Finally, the subsequence of $\pi$ comprising $[a-1]$ must itself decrease (and hence $z=a-1$), for otherwise there is a 4123 of the form $ba_1a_2y$ for some $a_1,a_2\in [a-1]$.  Thus, we have shown that $\pi$ must have the form
$$\pi=\alpha ab\beta(a-1)\cdots 1n(n-1)\cdots(b+1)\gamma,$$
where $\alpha$, $\beta$ and $\gamma$ denote subsequences whose union is $[a+1,b-1]$ and $\alpha$ is possibly empty.

Note that $\alpha$, $\beta$ and $\gamma$ must all decrease, $\alpha$ since it precedes $a$, $\beta$ since $\pi$ avoids 1234, and $\gamma$ since $\pi$ avoids 4123.  Furthermore, we have $\min(\beta)=x<y=\max(\gamma)$ and thus $V$ consists of two decreasing runs.  To enumerate members of $\mathcal{W}_n$, we now consider cases on $V$.  First assume $V=[b-m,b-1]$, where $2 \leq m \leq b-1-a$, and thus $\alpha=[a+1,b-1-m]$.  If $j$ denotes the number of elements in $\beta$, then $1 \leq j \leq m-1$ and there are $\binom{m}{j}-1$ possibilities for $\beta$ and $\gamma$ since they are both decreasing with $\min(\beta)<\max(\gamma)$ (the latter condition preventing $\beta$ from being the set $[b-j,b-1]$). Since the letters to the left of $b$ within $\pi$ comprise an interval in this case, one may verify that there are no further restrictions on $\beta$ and $\gamma$ and that each corresponding permutation $\pi$ obtained in this manner is $T$-avoiding.  Considering all possible $a$, $b$, $m$ and $j$ implies that the number of permutations in $\mathcal{W}_n$ for which $V$ is of the stated form is given by
\begin{align*}
&\sum_{a=2}^{n-5}\sum_{b=a+3}^{n-2}\sum_{m=2}^{b-1-a}\sum_{j=1}^{m-1}\left[\binom{m}{j}-1\right]=\sum_{a=2}^{n-5}\sum_{b=a+3}^{n-2}\sum_{m=0}^{b-1-a}(2^m-m-1)\\
&=\sum_{a=2}^{n-5}\sum_{b=a}^{n-2}\left[2^{b-a}-1-\binom{b-a+1}{2}\right]=\sum_{n=2}^{n-1}\left[2^{n-1-a}-(n-a)-\binom{n-a}{3}\right]\\
&=2^{n-2}-1-\binom{n-1}{2}-\binom{n-1}{4}.
\end{align*}

Now assume $V$ does not comprise an interval of the form $[b-m,b-1]$.  Let $q$ be the largest letter in $\alpha$.  Then all $v\in V$ with $v<q$ must occur in $\beta$, for otherwise there is a 2341 of the form $qb(b+1)v$.  In order for there to be an increase in the subsequence $V$, we must have $q \leq b-2$.  Then the elements of $[q+1,b-1]$ can go in either $\beta$ or $\gamma$, with at least one element going in $\gamma$, for otherwise $V$ would be decreasing.  Also, members of $[a+1,q-1]$ can go in either $\alpha$ or $\beta$ with at least one going in $\beta$ since $\alpha$ is not an interval of the form $[a+1,a+r]$ for some $r \geq 1$.  Note that at least one element of $[a+1,q-1]$ belonging to $\beta$ ensures that $V$ does not decrease.  Given $a$, $b$ and $q$, there are thus $2^{q-1-a}-1$ possibilities for the placement of elements of $[a+1,q-1]$ and $2^{b-1-q}-1$ possibilities for the placement of elements of $[q+1,b-1]$.  Since these positions may be chosen independently, there are $(2^{q-1-a}-1)(2^{b-1-q}-1)$ possibilities in all and one can check that all permutations $\pi$ so obtained are $T$-avoiding.  Note that $q \geq a+2$ implies $a+4 \leq b \leq n-2$ and $2 \leq a \leq n-6$.  Considering all possible $a$, $b$ and $q$ gives
\begin{align*}
&\sum_{a=2}^{n-6}\sum_{b=a+4}^{n-2}\sum_{q=a+2}^{b-2}(2^{q-1-a}-1)(2^{b-1-q}-1)\\
&=\sum_{a=2}^{n-6}\sum_{b=a+4}^{n-2}\left[(b-3-a)2^{b-2-a}-2(2^{b-2-a}-2)+(b-3-a)\right]\\
&=\sum_{a=2}^{n-6}\sum_{b=a+4}^{n-2}(b-5-a)2^{b-2-a}+\sum_{a=2}^{n-6}\left[4(n-5-a)+\binom{n-4-a}{2}\right]\\
&=\sum_{a=2}^{n-6}\sum_{b=a+4}^{n-2}(b-5-a)2^{b-2-a}+4\binom{n-6}{2}+\binom{n-5}{3}
\end{align*}
additional members of $\mathcal{W}_n$.  Combining this case with the previous completes the proof.
\end{proof}

We have the following formula for $e_n$ in terms of the prior sequences.

\begin{lemma}\label{len}
If $n \geq 6$, then
\begin{equation}\label{lene1}
e_n=v_n+w_n+C_{n-2}-3\cdot2^{n-3}+1+\sum_{\ell=1}^{n-2}C_\ell,
\end{equation}
with $e_5=2$, where $v_n$ and $w_n$ are as defined above.
\end{lemma}
\begin{proof}
To show this, we need only enumerate the members of $\mathcal{E}_n$ in which no element of $[a+1,b-1]$ occurs to the right of $n$.  Let $\pi \in \mathcal{E}_n$ be of this form.  Then $\pi$ may be written as
$$\pi=(n-\ell-1)\cdots(b+1)\pi'n(n-1)\cdots(n-\ell), \qquad 0 \leq \ell \leq n-1-b,$$
where $\pi'$ is a permutation of length $b$ whose leftmost ascent is $a,b$ with $a \geq 2$.  Note that $\pi'$ is $123$-avoiding, but that there are no further restrictions on $\pi'$ imposed by the letters in $[b+1,n]$, which implies that they may be deleted.  Upon subtracting $123$-avoiding permutations starting with $b$ or having leftmost ascent $1,b$, we get $C_b-C_{b-1}-2^{b-2}$ possibilities for $\pi'$.  Considering all $b$ and $\ell$ gives
\begin{align*}
&\sum_{b=3}^{n-2}\sum_{\ell=0}^{n-1-b}(C_b-C_{b-1}-2^{b-2})=C_{n-2}-2^{n-3}+\sum_{\ell=1}^{n-4}\sum_{b=3}^{n-1-\ell}(C_b-C_{b-1}-2^{b-2})\\
&=C_{n-2}-2^{n-3}+\sum_{\ell=1}^{n-2}(C_{n-1-\ell}-2^{n-2-\ell})=C_{n-2}-3\cdot2^{n-3}+1+\sum_{\ell=1}^{n-2}C_\ell
\end{align*}
additional members of $\mathcal{E}_n$, which completes the proof.
\end{proof}

\subsubsection{Case IV}
In this section, we complete our enumeration of $S_n(T)$ by finding $g_n$, which is given as follows.

\begin{lemma}\label{gnl}
If $n \geq 3$, then
\begin{equation}\label{gnle1}
g_n=g_{n-1}+C_{n-1}+(n+1)2^{n-3}-\binom{n}{3}-n,
\end{equation}
where $g_2=1$.
\end{lemma}
\begin{proof}
We proceed as in the proof of Lemma \ref{gnlem} above and use the same terminology and notation, but with respect to the new pattern set $T$.  Let $\pi=a_1\cdots a_s1\pi'$ be of the form enumerated by $g_n$, where $n\geq 4$ and $1<a_s<\cdots<a_1 \leq n-1$ and let $S=\{a_1,\ldots,a_s\}$.  If $S=\emptyset$, then $\pi=1\pi'$, where $\pi'$ avoids 123, and thus there are $C_{n-1}$ possibilities.  So assume $S \neq \emptyset$.  If $a_s=2$, then deleting 2 gives $g_{n-1}$ possibilities, so assume $a_s \geq 3$ henceforth.  We consider the relative positions of large and small letter within $\pi$, where large and small are defined as before.

First observe that large letters occurring to the left of any small letters must decrease in order to avoid 2341, as do large letters occurring to the right due to 1234.  Small letters must decrease, for otherwise there would be a 4123 starting with $a_1,1$.  First suppose that at least one large letter occurs between the first and the last of the small letters.  If $x$ denotes the rightmost such large letter, then all large letters to the left of $x$ must decrease and are greater than $x$, while all large letters to the right of $x$ must decrease and are less than $x$.  Thus, the full subsequence comprising the large letters is decreasing.  Furthermore, the set $S$ in this case must comprise an interval of the form $[\ell,\ell+s-1]$ for some $4 \leq \ell \leq n-1$ and $1 \leq s \leq n-\ell$, for otherwise there would be a 2341.  Note that there is no restriction on the number of large letters lying between the first and the last small letters.  Since there are $\ell-2$ small letters and $n-\ell-s+1$ large letters, at least one of which occurs between small letters, there are $\binom{n-s-1}{\ell-2}-(n-\ell-s+2)$ possible arrangements of these letters, by subtraction.  It is seen that all permutations arising in this manner are $T$-avoiding.  Considering all $\ell$ and $s$ gives
\begin{align*}
&\sum_{\ell=4}^{n-1}\sum_{s=1}^{n-\ell}\left[\binom{n-s-1}{\ell-2}-(n-\ell-s+2)\right]=\sum_{\ell=4}^{n-1}\left[\binom{n-1}{\ell-1}-\binom{n-\ell+2}{2}\right]\\
&=\sum_{\ell=3}^{n}\left[\binom{n-1}{\ell-1}-\binom{n-\ell+2}{2}\right]=2^{n-1}-n-\binom{n}{3}
\end{align*}
permutations in this case.

Now suppose that no large letter occurs between the first and the last of the small letters.  Then $\pi$ in this case consists of three runs of descending letters: (i) the initial decreasing run of length $s+1$ (including $1$), (ii) a decreasing run which includes all of the small letters, and (iii) a possibly empty decreasing run of large letters.  From this, one need only select the elements comprising (i) and any large letters in (ii), which uniquely determines $\pi$.  If $m=a_1$, then there are $\binom{m-3}{s-1}$ choices for the additional letters in (i) and $2^{n-m}$ choices for the large letters in (ii).  Note that since 2 is a small letter, we have $m \geq s+2$.  Considering all possible $s$ and $m$ then yields
$$\sum_{s=1}^{n-3}\sum_{m=s+2}^{n-1}2^{n-m}\binom{m-3}{s-1}=\sum_{m=3}^{n-1}2^{n-m}\sum_{s=1}^{m-2}\binom{m-3}{s-1}=\sum_{m=3}^{n-1}2^{n-3}=(n-3)2^{n-3}$$
additional permutations.  Combining this case with the others gives \eqref{gnle1}.
\end{proof}

For any of the sequence above, put zero if $n$ is such that the corresponding set of permutations is empty.  For example, put $e_n=0$ if $n \leq 4$. Combining cases I through IV above, and including permutations that start with the letter $n$, implies that the number of $T$-avoiding permutations of length $n$ is given by
\begin{equation}\label{combre}
a_n=b_n+d_n+e_n+g_n+C_{n-1}, \qquad n \geq 2,
\end{equation}
with $a_0=a_1=1$.

Now we are ready to find a formula for the generating function $A(x)=\sum_{n\geq0}a_nx^n$. Define
$B(x)=\sum_{n\geq3}b_nx^n$, $D(x)=\sum_{n\geq4}d_nx^n$, $E(x)=\sum_{n\geq5}e_nx^n$  and $G(x)=\sum_{n\geq2}g_nx^n$. First we rewrite the recurrences from the preceding lemmas in terms of generating functions by multiplying both sides by $x^n$ and summing over all possible $n$. By Lemmas \ref{bnl}--\ref{gnl}, we obtain
\begin{align}
B(x)&=\frac{1}{1-x}\biggl(u(x)+\frac{x^4}{(1-2x)^2}+\frac{x^7}{(1-x)^6}-\frac{x^4}{(1-x)^3}\notag\\
&\quad+xq\left(\frac{x}{1-x},1-x\right)-xq(x,1)\biggr),\label{eqt185tb}\\
D(x)&=xB(x)+\frac{x^5(3x^4-5x^3+6x^2-4x+1)}{(1-x)^7(1-2x)},\label{eqt185td}\\
E(x)&=2x^5-\frac{x^6(2x-3)}{(1-x)^4(1-2x)^2}-\frac{x^7(x^2+x-1)}{(1-x)^5(1-2x)^2}+x^2(C(x)-5x^3-2x^2-x-1)\label{eqt185te}\\
&\quad-\frac{24x^6}{1-2x}+\frac{x^6}{1-x}+\frac{8x^6}{1-x}+\frac{x^2}{1-x}(C(x)-5x^3-2x^2-x-1),\notag\\
G(x)&=\frac{x}{1-x}(C(x)-1)+\frac{x^4(2x^4-8x^3+10x^2-7x+2)}{(1-x)^5(1-2x)^2}.\label{eqt185tg}
\end{align}
By solving \eqref{eqt185tb}--\eqref{eqt185tg}, along with use of \eqref{eqtqq} and Lemma \ref{ungf}, we obtain
\begin{align*}
B(x)&=x(1-x)C^2(x)\\
&\quad+\frac{x(3x^9-15x^8+59x^7-134x^6+192x^5-181x^4+112x^3-44x^2+10x-1)}{(1-x)^7(1-2x)^2},\\
D(x)&=x^2(1-x)C^2(x)-\frac{x^2(3x^6-12x^5+27x^4-31x^3+20x^2-7x+1)}{(1-x)^4(1-2x)^2},\\
E(x)&=\frac{x^3(2-x)}{1-x}C^2(x)-\frac{x^3(3x^6-14x^5+38x^4-49x^3+34x^2-13x+2)}{(1-x)^5(1-2x)^2},\\
G(x)&=\frac{x^2}{1-x}C^2(x)+\frac{x^4(2x^4-8x^3+10x^2-7x+2)}{(1-x)^5(1-2x)^2}.
\end{align*}
By \eqref{combre}, we have
$$A(x)=1+x+B(x)+D(x)+E(x)+G(x)+x(C(x)-1),$$
which, upon substituting the expressions just found for the generating functions $B(x)$, $D(x)$, $E(x)$ and $G(x)$, implies
the following result.
\begin{theorem}\label{th185a}
Let $T=\{1234,2341,4123\}$. Then
$$F_T(x)=\frac{1-11x+51x^2-130x^3+199x^4-183x^5+91x^6-15x^7-6x^8+4x^9}{(1-x)^7(1-2x)^2}
+\frac{1+x}{1-x}(C(x)-1)\,.$$
\end{theorem}

\subsection{Case 209: $\{3142,1432,1243\}$} This is one of the shortest cases.
\begin{theorem}\label{th209a}
Let $T=\{3142,1432,1243\}$. Then the generating function
$F=F_T(x)$ satisfies
$$F=1-xF+x\left(2+\frac{x^2}{(1-x)^2}\right)F^2-x^2F^3\,.$$
\end{theorem}
\begin{proof}

To find $G_m(x)$ for $m\geq2$, let
$\pi=i_1\pi^{(1)}i_2\pi^{(2)}\cdots i_m\pi^{(m)}\in S_n(T)$ with
$m\ge 2$ left-right maxima. Since $\pi$ avoids $1243$, we see
that $\pi^{(2)}\cdots\pi^{(m)}<i_2$. We consider two cases:
\begin{itemize}
\item There is no letter in $\pi^{(2)}\cdots\pi^{(m)}$ between
$i_1$ and $i_2$, that is, $\pi^{(2)}\cdots\pi^{(m)}<i_1$. Since
$\pi$ avoids $3142$, we see that
$\pi^{(1)}>\pi^{(2)}>\cdots>\pi^{(m)}$. Therefore, we have a
contribution of $x^mF^m$.

\item There is some $p\in [\kern .1em 2,m\kern .1em]$ minimal such that $\pi^{(p)}$ has a letter
between $i_1$ and $i_2$. Then
$\pi^{(2)}\pi^{(3)}\cdots\pi^{(p-1)}<i_1$ by the minimality of $p$ and also  $\pi^{(p+1)}\cdots\pi^{(m)}<i_1$ since $\pi$ avoids $1243$ and $1432$. Furthermore, since $\pi$ avoids $3142$,
$\pi^{(2)},\pi^{(3)},\dots,\pi^{(p-1)}$ are all empty.
Hence, $i_1+1,i_1+2,\dots,i_2-1$ all lie in $\pi^{(p)}$  and furthermore occur in that order (or $i_1i_2$ is the $14$ of a $1432$).

So, with $d=i_2-i_1$, we have $d\ge 2$ and can write $\pi^{(p)}$ as
$$\al_1(i_1+1)\al_2(i_1+2)\cdots\al_{d-1}(i_2-1)\al_d\,.$$
We also have $\pi^{(1)}>\al_1>\al_2>\cdots>\al_{d}>\pi^{(p+1)}>\cdots>\pi^{(m)}$
(a violator of any of these inequalities would be the $12$ of a $3142$), and
$\pi^{(1)}$ is decreasing ($u<v$ in $\pi^{(1)}$ would make $uvi_2(i_1+1)$ a $1243$).
Hence, the contribution in this case is
$\sum_{p=2}^m\frac{x^m}{1-x}F^{m-p}\sum_{d\ge2}x^{d-1}F^d$.
\end{itemize}
Adding the two contributions, we have for $m\ge 2$,
$$G_m(x)= x^m F^m + \frac{x^{m+1}\sum_{p=2}^mF^{m+2-p}}{(1-x)(1-xF)}\,.$$
Summing over $m$ yields the stated expression for $F$.
\end{proof}

\subsection{Case 216: $\{2143,3142,3412\}$}
Let $H_m=H_m(x)$ be the generating function for permutations $\pi=\pi_1\pi_2\cdots\pi_n\in S_n(T)$ such that $\pi_1<\pi_2<\cdots<\pi_m=n$.
Clearly, $H_1=xF_T(x)$.

\begin{lemma}\label{lem216a}
For $m\geq2$,
$$(1-2x)H_m=x(1-x)H_{m-1}+x\sum_{j\geq m}\big((1-x)H_j-xH_{j-1}\big)\,.$$
\end{lemma}
\begin{proof}
Consider a $T$-avoider $\pi=\pi_1\pi_2\cdots\pi_n\in S_n(T)$ counted by $H_m$, where $m \geq 2$. If $\pi_1=1$, the contribution is $xH_{m-1}$.
Otherwise, $\pi$ has the form $\pi=\pi_1\pi_2\cdots\pi_m\al1\beta$, where $\beta>\pi_{m-1}$ (or $\pi_{m-1}\pi_m 1$ is the $341$ of a $3412$) and $\beta$ is increasing (or $\pi_{m-1}1$ is the $21$ of a $2143$). If $\beta=\emptyset$, the contribution is $xH_m$, and so $H_m=xH_{m-1}+xH_m+J_m$ where $J_m$ is the contribution for the case $\beta\neq\emptyset$.

Now let us write an equation for $J_m$. If $n-1$ is in $\beta$, then $n-1$ is the last letter in $\pi$ and, deleting $n-1$, we have a contribution of $x(H_m-xH_{m-1})$.
Otherwise, $n-1$ is in $\al$, and $\pi$ has the form shown in Figure \ref{fig2i6J} below,
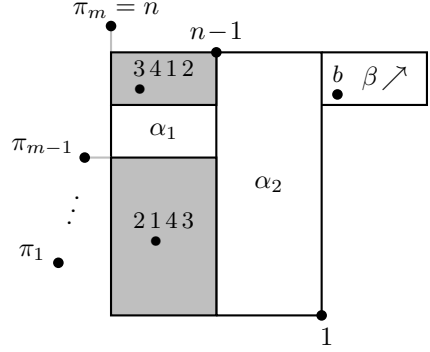
\begin{figure}[H]
\begin{center}
\begin{pspicture}(-5,0)(10,4)
\psset{xunit=.7cm}\psset{yunit=.7cm}
\psline(1,3)(1,4)\psline(3,0)(5,0)(5,5)(3,5)(3,0)\psline(5,4)(7,4)(7,5)(5,5)
\psline[linecolor=lightgray](1,5)(1,5.5)\psline[linecolor=lightgray](.5,3)(1,3)
\pspolygon[fillstyle=solid,fillcolor=lightgray](1,0)(3,0)(3,3)(1,3)(1,0)
\pspolygon[fillstyle=solid,fillcolor=lightgray](3,4)(3,5)(1,5)(1,4)(3,4)
\rput(2,3.5){\textrm{$\al_1$}}
\rput(4,2.5){\textrm{$\al_2$}}\rput(6.2,4.5){\textrm{$\be\!\nearrow$}}
\rput(5.1,-.4){\textrm{$1$}}\rput(-.5,1.2){\textrm{$\pi_1$}}
\rput(-.3,3.2){\textrm{$\pi_{m-1}$}}\rput(1.1,5.8){\textrm{$\pi_m=n$}}
\rput(5.3,4.6){\textrm{\small $b$}}\rput(3,5.4){\textrm{$n\!-\!1$}}
\rput[b]{35}(.4,1.8){$\iddots$}\qdisk(0,1){2pt}\qdisk(.5,3){2pt}
\qdisk(1,5.5){2pt}\qdisk(3,5){2pt}\qdisk(5,0){2pt}\qdisk(5.3,4.2){2pt}
\rput(2,1.8){\textrm{\small $2\,1\,4\,3$}}\rput(1.85,1.4){\textrm{\small $\bullet$}}
\rput(2,4.7){\textrm{\small $3\,4\,1\,2$}}\rput(1.55,4.3){\textrm{\small $\bullet$}}
\end{pspicture}
\caption{A $T$-avoider counted by $J_m$ with $n-1$ before 1}\label{fig2i6J}
\end{center}
\end{figure}
where $\be$ is increasing and nonempty (bullet marked $b$), $\al_1$ is increasing (or $(n-1)b$ is the $43$ of a $2143$), and shaded regions are empty for the indicated reason. Deleting $\pi_m=n$ leaves avoiders counted by $J_{m+t}$ where $t$ is the number of letters in $\al_1$.

Adding the contributions in the two cases for $n-1$ gives
$$J_m=x(H_m-xH_{m-1})+x\sum_{t\ge 0}J_{m+t}\,.$$
Since $J_m=(1-x)H_m-xH_{m-1}$, this implies
$$(1-x)H_m-xH_{m-1}=x(H_m-xH_{m-1})+x\sum_{j\geq m}\big((1-x)H_j-xH_{j-1}\big),$$
and the result follows.
\end{proof}

\begin{lemma}\label{lem216b}
We have
$$\sum_{m\ge 1}H_m(x) = xC(x)F_T(x)\,.$$
\end{lemma}
\begin{proof}
Define $H(x,v)=\sum_{m\geq1}H_m(x)v^{m-1}$.
By Lemma \ref{lem216a}, we have
\begin{align*}
&(1-2x)H(x,v)-x(1-2x)F_T(x)\\
&=xv(1-x)H(x,v)+\frac{xv(1-x)}{1-v}\big(H(x,1)-H(x,v)\big)-\frac{x^2v}{1-v}\big(H(x,1)-vH(x,v)\big),
\end{align*}
which implies
\begin{align*}
\left(1-2x-xv(1-x)+\frac{xv(1-x-xv)}{1-v}\right)H(x,v)=x(1-2x)F_T(x)+\frac{xv(1-2x)}{1-v}H(x,1).
\end{align*}
By taking $v=C(x)$, we find that
$H(x,1)=xC(x)F_T(x)$.
\end{proof}

\begin{theorem}\label{th216a}
Let $T=\{2143,3142,3412\}$. Then
$$F_T(x)=\frac{(1-3x)C(x)}{1-2x-x(1-x)C(x)}\,.$$
\end{theorem}
\begin{proof}
To write an equation for $G_m(x)$ where $m\geq2$, suppose $\pi=i_1\pi^{(1)}\cdots i_m\pi^{(m)}\in S_n(T)$ has $m\ge 2$ left-right maxima. If $\pi^{(m)}=\emptyset$, the contribution is $xG_{m-1}(x)$. Otherwise, there exists $j$ maximal, $1\leq j\leq m$, such that $\pi^{(m)}$ has a letter between $i_{j-1}$ and $i_j$ (where $i_0=0$). Since $\pi$ avoids $2143$, $\pi^{(1)}\cdots\pi^{(j-1)}=\emptyset$ and $\pi^{(j)}\cdots\pi^{(m-1)}>i_{j-1}$. Since $\pi$ avoids $3142$ and $j$ is maximal, we have $\pi^{(j)}\cdots\pi^{(m-1)}>\pi^{(m)}$.  There are two cases:
\begin{itemize}
\item $j=m$. If $\pi^{(m)}<i_{m-1}$, then $\pi^{(m)}$ is decreasing because we avoid $3412$. So the contribution is $H_m-\frac{x^m}{(1-x)^{m-1}}$, where $\frac{x^m}{(1-x)^{m-1}}$ counts all the permutations $\pi$ with $i_{m-1}=n-1$.

\item $1\leq j\leq m-1$. Again, $\pi^{(m)}$ is decreasing, and we have a contribution of $\frac{x^{j+1}}{(1-x)^{j}}G_{m-j}(x)$.
\end{itemize}
Adding all the contributions gives
$$G_m(x)=xG_{m-1}(x)+H_m-\frac{x^m}{(1-x)^{m-1}}+\sum_{j=1}^{m-1}\frac{x^{j+1}}{(1-x)^{j}}G_{m-j}(x)\,.$$
Sum over $m\geq2$ using Lemma \ref{lem216b} to obtain
$$F_T(x)-xF_T(x)-1=x\big(F_T(x)-1\big)+xC(x)F_T(x)-xF_T(x)-\frac{x^2}{1-2x}+\frac{x^2}{1-2x}\big(F_T(x)-1\big).$$
Solving for $F_T(x)$ completes the proof.
\end{proof}

\subsection{Case 225: $\{1243,2413,3142\}$}
\begin{theorem}\label{th225a}
Let $T=\{1243,2413,3142\}$. Then
$$F_T(x)=\frac{1-x(1-x)C(x)-\sqrt{1-5x+10x^2-5x^3-x(1-x)(1+x)C(x)}}{2x(1-x)}.$$
\end{theorem}
\begin{proof}
To find $G_m(x)$ for $m\geq2$, let $\pi=i_1\pi^{(1)}i_2\pi^{(2)}\cdots i_m\pi^{(m)}\in S_n(T)$ with $m\ge 2$ left-right maxima. Then
$\pi^{(j)}<i_2$ for $j=1,2,\ldots,m$ (a violator $v$ makes $i_1 i_2 i_j v$ a 1243) and $\pi$ has the form
shown in Figure \ref{fig225} below,
\begin{figure}[H]
\begin{center}
\begin{pspicture}(-.3,0)(8,5.5)
\psset{xunit=.8cm}\psset{yunit=.5cm}
\psline[linecolor=lightgray](1,5)(8,5)\psline[linecolor=lightgray](2,8)(3,8)
\psline[linecolor=lightgray](4,7)(5,7)\psline[linecolor=lightgray](1,4)(2,4)
\psline[linecolor=lightgray](3,3)(4,3)\psline[linecolor=lightgray](5,2)(6,2)
\psline[linecolor=lightgray](1,5)(1,8)\psline[linecolor=lightgray](2,4)(2,8)
\psline[linecolor=lightgray](3,4)(3,9.4)\psline[linecolor=lightgray](4,3)(4,7)
\psline[linecolor=lightgray](5,3)(5,9.8)\psline[linecolor=lightgray](6,2)(6,6)
\psline[linecolor=lightgray](9,1)(9,5)\psline[linecolor=lightgray](8,6)(8,10.3)
\psline(0,4)(1,4)(1,5)(0,5)(0,4)\psline(2,3)(3,3)(3,4)(2,4)(2,3)
\psline(4,2)(5,2)(5,3)(4,3)(4,2)\psline(6,1)(7,1)(7,2)(6,2)(6,1)
\psline(9,0)(10,0)(10,1)(9,1)(9,0)\psline(1,8)(2,8)(2,9)(1,9)(1,8)
\psline(3,7)(4,7)(4,8)(3,8)(3,7)\psline(5,6)(6,6)(6,7)(5,7)(5,6)
\psline(8,5)(9,5)(9,6)(8,6)(8,5)
\rput(0.5,4.5){\textrm{$\be_1$}}\rput(2.5,3.5){\textrm{$\be_2$}}
\rput(4.5,2.5){\textrm{$\be_3$}}\rput(6.5,1.5){\textrm{$\be_4$}}
\rput(9.5,0.5){\textrm{$\be_m$}}\rput(1.5,8.5){\textrm{$\al_2$}}
\rput(3.5,7.5){\textrm{$\al_3$}}\rput(5.5,6.5){\textrm{$\al_4$}}
\rput(8.5,5.5){\textrm{$\al_m$}}
\rput(-.3,5.2){\textrm{$i_1$}}\rput(.7,9.2){\textrm{$i_2$}}
\rput(2.7,9.6){\textrm{$i_3$}}\rput(4.7,10){\textrm{$i_4$}}
\rput(7.7,10.5){\textrm{$i_m$}}
\rput[b]{10}(6.3,10.1){$\dots$}\rput[b]{25}(7,5.9){$\ddots$}\rput[b]{25}(8,.9){$\ddots$}
\qdisk(0,5){2pt}\qdisk(1,9){2pt}\qdisk(3,9.4){2pt}\qdisk(5,9.8){2pt}\qdisk(8,10.3){2pt}
\end{pspicture}
\caption{A $T$-avoider with $m\ge 2$ left-right maxima}\label{fig225}
\end{center}
\end{figure}
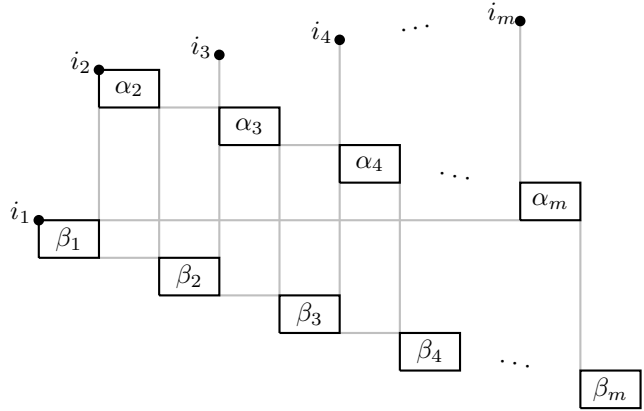
\vspace*{-5mm}
where, for each $j\in[\kern .1em 2,m\kern .1em]$, one has $\al_j>i_1>\be_j$ as shown (or $i_{1}i_j$ is the 24 of a 2413), with $\al_j$ avoiding 132 (or $i_1$ is the 1 of a 1243). Also, for each
$j\in[\kern .1em 2,m-1\kern .1em],\ \al_j>\al_{j+1}$ (or $i_j i_{j+1}$ is the 34 of a 3142) and the same holds for the $\beta_j$.

There are two cases:
\begin{itemize}
\item $\al_1\cdots\al_m=\emptyset$. In this case, $\pi$ avoids $T$ if and only if $\be_i$ avoids $T$ for all $i=1,2,\ldots,m$, giving a contribution of $x^mF_T(x)^m$.

\item $\al_1\cdots\al_m\neq\emptyset$. Let $j$ be the maximal index such that $\al_j\neq\emptyset$ and say $a$ is a letter in $\al_j$. Then $\be_2=\cdots=\beta_{j-1}=\emptyset$ (or $b$ in $\be_i$ with $2\le i \le j-1$ makes $i_1 i_2ba$ a 2413) and $\be_1$ is decreasing (or $i_2 a$ is the 43 of a 1243). So $\pi$ avoids $T$ if and only if $\be_i$ avoids $T$ for $i=j,j+1,\ldots,m$ and $\al_i$ avoids $132$ for  $i=2,3,\ldots,j$, giving a contribution of
$$\sum_{j=2}^m\frac{x^m}{1-x}C(x)^{j-2}\big(C(x)-1\big)F_T(x)^{m-j+1}.$$
\end{itemize}
Hence,
\begin{align*}
G_m(x)&=x^mF_T(x)^m+\sum_{j=2}^m\frac{x^m}{1-x}C(x)^{j-2}\big(C(x)-1\big)F_T(x)^{m-j+1}.
\end{align*}
Summing  over $m\geq2$, we obtain
$$F_T(x)=1+xF_T(x)+\frac{\Big((1-x)\big(1-xC(x)\big)F_T(x)+C(x)-1\Big)x^2F_T(x)}{(1-x)\big(1-xC(x)\big)\big(1-xF_T(x)\big)}\,,$$
and solving for $F_T(x)$ completes the proof.
\end{proof}

\subsection{Case 228: $\{2341,2413,3412\}$} In this case, we find a cubic equation for $F_T$.
\begin{theorem}\label{th228a}
Let $T=\{2341,2413,3412\}$. Then the generating function $F=F_T(x)$ satisfies
$$F=(1-x)^2+xF+x(2-3x+2x^2)F^2-x^2(1-x)F^3.$$
\end{theorem}
\begin{proof}
First, we write an equation for $G_m(x)$ with $m\geq3$. Suppose $\pi=i_1\pi^{(1)}i_2\pi^{(2)}\cdots i_m\pi^{(m)}\in S_n(T)$ with $m\ge 3$ left-right maxima. Since $\pi$ avoids $2341$, we see that $\pi^{(j)}>i_{j-2}$ for all $j=2,3,\ldots,m$ (with $i_0=0$). If $\pi^{(2)}>i_1$, then $\pi$ avoids $T$ if and only if $\pi^{(1)}$ avoids $T$ and $i_2\pi^{(2)}\cdots i_m\pi^{(m)}$ avoids $T$, which gives a contribution of $x\F G_{m-1}(x)$. Otherwise, $\pi^{(2)}$ has a letter smaller than $i_1$, which implies $\pi^{(3)}>i_2$ ($\pi$ avoids $2413$), and $\pi$ avoids $T$ if and only if $i_1\pi^{(1)}i_2\pi^{(2)}$ and $i_3\pi^{(3)}\cdots i_m\pi^{(m)}$ both avoid $T$, which gives a contribution of $\big(G_2(x)-x^2\F^2\big)G_{m-2}(x)$. Hence, for $m\ge 3$,
\begin{equation}\label{eq228m3}
G_m(x)=x\F G_{m-1}(x)+\big(G_2(x)-x^2\F^2\big)G_{m-2}(x)\,.
\end{equation}

Next, to find $G_2(x)$, suppose $\pi=i\pi'n\pi''\in S_n(T)$ has $2$ left-right maxima. Let $k$ be
the number of letters smaller than $i$ in $\pi''$.
If $k=0$, the contribution is $x^2\F^2$.
Otherwise, $k\ge 1$ and these $k$ letters, say $i_1,i_2,\dots,i_k$, are decreasing (to avoid  $3412$)
and occur at the end of $\pi''$ (or $in$ is the 24 of a 2413). Also, the letters $n-1,n-2,\dots,i+1$
in $\pi''$ occur in that order (to avoid 2341). So $\pi$ has the form
\[
i\pi'n(n-1)\cdots(i+1)i_1 i_2 \cdots i_k\,.
\]
Set $i_0=i$. Then the letters in $[\kern .1em i_k,i_0\kern .1em]\backslash\{i_1,i_2,\dots,i_k\}$ are decreasing
(or $ni_k$ is the 41 of a 2341) and they split $i\pi'=i_0\pi'$ into segments as follows:
\begin{align*}
i_0\pi'=\ &i_0\al_{0,0}(i_0-1)\al_{0,1}(i_0-2)\al_{0,2}\cdots(i_1+1)\al_{0,i_0-i_1-1}\\
&\hspace*{9mm}(i_1-1)\al_{1,1}(i_1-2)\al_{1,2}\cdots(i_2+1)\al_{1,i_1-i_2-1}\\
&\hspace*{15mm}\vdots\\
&\hspace*{9mm}(i_{k-1}-1)\al_{k-1,1}(i_{k-1}-2)\al_{k-1,2}\cdots(i_k+1)\al_{k-1,i_{k-1}-i_k-1}\,,
\end{align*}
where $\al_{0,0}<\al_{0,1} < \cdots < \al_{1,1}< \al_{1,2} <\cdots$, i.e., each $\al$ is less than its successor (to avoid 2413).
Now $\pi$ avoids $T$ if and only if each $\al_{i,j}$ avoids $T$.
Hence, for each $k\geq1$, we have a contribution of $$\frac{x^{k+1}}{1-x}\times \frac{x\F}{(1-x\F)^k}\,.$$
Adding the contributions gives
$$G_2(x)=x^2\F^2+\sum_{k\geq1}\frac{x^{k+2}\F}{(1-x)(1-x\F)^k}=x^2\F^2+\frac{x^3\F}{(1-x)(1-x-x\F)}\,.$$

Summing (\ref{eq228m3}) over $m\geq3$, we obtain
$$\F=1+x\F+G_2(x)+x\F(\F-1-x\F)+(\F-1)\big(G_2(x)-x^2\F^2\big)\,.$$
Now substitute for $G_2(x)$ and solve for $\F$ to complete the proof.
\end{proof}

\subsection{Case 230: $\{2341,1243,1234\}$} Note that all three patterns in $T$ contain $123$.
We count by initial letters and define $a(n;i_1,i_2,\dots, i_m)$ to be the number of $T$-avoiders in $S_n$ whose first $m$ letters are $i_1,i_2,\dots, i_m$, with $a(n):=|S_n(T)|$.
Clearly, $a(n;n)=a(n;n-1)=a(n-1)$. So we need to consider $T$-avoiders with first letter $\le n-2$.
Accordingly, for $m\ge 1$, denote the \gf for $T$-avoiders $\pi=\pi_1\pi_2 \cdots \pi_n$ such that
$n-2 \ge \pi_1>\pi_2> \cdots >\pi_m$ and $\pi_{m+1}$ is arbitrary (resp. $\pi_{m+1}=n-1$) by $A_m$
(resp. $B_m$). Also for $m\ge 1$, denote the \gf for $T$-avoiders $\pi=\pi_1\pi_2 \cdots \pi_n$
such that $n-1 \ge \pi_1>\pi_2> \cdots >\pi_m$ by $D_m$ (which will be needed to define a recursion).
\begin{lemma}\label{lem230D} We have
\[
D_m=\begin{cases}
A_m+x D_{m-1} & \textrm{ if $m\ge 2$,} \\
A_1+x(F_T-1)  & \textrm{ if $m=1.$}
\end{cases}
\]
\end{lemma}
\begin{proof}
Split $T$-avoiders counted by $D_m$ into those with $\pi_1\le n-2$, counted by $A_m$, and those for which $\pi_1=n-1$. Deleting $n-1$ from the latter avoiders gives the result.
\end{proof}

\begin{corollary}For $m\ge 1$,
\begin{equation} \label{eq230a2}
A_m=A_{m+1}+B_m+xA_m+x^2A_{m-1}+\cdots+x^{m}A_1+x^{m+1}A_0\,,
\end{equation}
where $A_0:=F_T-1$.
\end{corollary}
\begin{proof}
For $\pi$ counted by $A_m$, consider $\pi_{m+1}$. Note that $\pi_{m+1}$ cannot lie in the interval $[\pi_m +1,n-2]$ for then $\pi_m\pi_{m+1}$ would be the 12 of either a 1234 or a 1243 involving $\{n-1,n\}$. So, if $\pi_{m+1}<\pi_m$, the contribution is $A_{m+1}$, if $\pi_{m+1}=n-1$, the contribution is $B_m$, and if
$\pi_{m+1}=n$, then by deleting $n$, the contribution is $xD_m$. The result follows by applying Lemma \ref{lem230D} repeatedly to $D_m$.
\end{proof}

Since $a(n;n)=a(n;n-1)=a(n-1)$ for $n \geq 2$, we have $a(n)=\sum_{i=1}^n a(n;i)=a(n;n)+a(n;n-1)+\sum_{i=1}^{n-2} a(n;i)=2a(n-1)+[x^n]A_1$, which implies
\begin{align}\label{eq130a1}
F_T-1-x=2x(F_T-1)+A_1.
\end{align}

To obtain a recurrence for $B_m$, define $\mathcal{B}_{n,m;j}=\{\pi_1\pi_2\cdots\pi_n\in S_n(T):\,
n-2\ge \pi_1>\pi_2>\cdots>\pi_m,\ \pi_{m+1}=n-1,\ \pi_{m+2}=n-j\}$.

\begin{lemma}\label{lem230D1} For all $3\leq j\leq m+1$,
$$\sum_{n\geq0}|\mathcal{B}_{n,m;j}|x^n=x^{m+3}C^{m+3-j}(x).$$
\end{lemma}
\begin{proof}
Let $\pi=\pi_1\pi_2\cdots\pi_n\in S_n(T)$ such that  $n-2\ge \pi_1>\pi_2>\cdots>\pi_m$, $\pi_{m+1}=n-1$ and $\pi_{m+2}=n-j$. Since $\pi_m<\pi_{m+2}$ and $n$ lies to the right of $\pi_{m+2}$, the letters $n-2,n-3,\ldots, n-j+1$ all lie to the left of $\pi_{m+2}$ (otherwise $\pi$ contains $1234$ or $1243$). Thus, $\pi=(n-2)(n-3)\cdots(n-j+1)\pi_{j-1}\cdots\pi_{n-1}\pi_n\in S_n(T)$ with $\pi_{j-1}>\pi_j>\cdots>\pi_m$, $\pi_{m+1}=n-1$, and $\pi_{m+2}=n-j$. Since $\pi$ avoids $2341$ and $\pi_1=n-2<\pi_{m+1}=n-1$, we have that $\pi_n=n$. Therefore, $\pi$ avoids $T$ if and only if $\pi'=\pi_{j-1}\cdots\pi_m\pi_{m+3}\cdots\pi_{n-1}$ avoids $123$ and $\pi_{j-1}>\cdots>\pi_m$. Hence, $\mathcal{B}_{n,m;j}$ equals the number of permutations in $S_{n-1-j}(123)$ with initial descent sequence (IDS) of length at least $m+2-j$.  Using Lemma \ref{refingf}, one can show that the generating for permutations in $S_n(123)$ having IDS of length at least $m'$ is given by $x^{m'}C^{m'+1}(x)$. Thus, for $j=3,4,\ldots,m+1$,
$$\sum_{n\geq0}|\mathcal{B}_{n,m;j}|x^n=x^{m+3}C^{m+3-j}(x),$$
as required.
\end{proof}

\begin{lemma}\label{lem230D2} For all $m\geq1$,
$$B_m(x)=B_{m+1}(x)+x^{m+3}\sum_{i=2}^mC^i(x)+xB_m(x)+x^{m+2}\frac{1-x}{1-2x}.$$
\end{lemma}
\begin{proof}
By Lemma \ref{lem230D1},
\begin{align*}
B_m(x)&=\sum_{n\geq0}\left(\sum_{n-1>i_1>i_2>\cdots>i_{m+1}\geq1}a(n;i_1i_2\cdots i_m(n-1)i_{m+1})\right)x^n\\
&\quad+\sum_{n\geq0}\left(\sum_{n-1>i_1>i_2>\cdots>i_{m}\geq1}\sum_{j=3}^{m+1}a(n;i_1i_2\cdots i_m(n-1)(n-j))\right)x^n\\
&=B_{m+1}(x)+\sum_{j=3}^{m+1}\sum_{n\geq0}|\mathcal{B}_{n,m;j}|x^n\\
&\quad+\sum_{n\geq0}\left(\sum_{n-1>i_1>i_2>\cdots>i_{m}\geq1}a(n;i_1i_2\cdots i_m(n-1)(n-2))\right)x^n
\end{align*}
\begin{align*}
&\quad+\sum_{n\geq0}\left(\sum_{n-1>i_1>i_2>\cdots>i_{m}\geq1}a(n;i_1i_2\cdots i_m(n-1)n)\right)x^n\qquad\\
&=B_{m+1}(x)+x^{m+3}\sum_{i=2}^mC^i(x)+xB_m(x)+x^{m+2}\frac{1-x}{1-2x},
\end{align*}
which completes the proof.
\end{proof}

Define $B(x,u)=\sum_{m\geq1}B_m(x)u^m$ and $A(x,u)=\sum_{m\geq1}A_m(x)u^m$. Rewriting the recurrence relations from Lemmas \ref{lem230D1} and \ref{lem230D2} in terms of $A(x,y)$ and $B(x,u)$ using \eqref{eq130a1}, we obtain
\begin{align*}
(1-x)B(x,u)&=\frac{1}{u}\left(B(x,u)-B_1(x)u\right)+\frac{x^5u^2C^2(x)}{(1-xu)(1-xuC(x))}
+\frac{x^3(1-x)u}{(1-2x)(1-xu)},\\
A(x,u)&=\frac{1}{u(1-xu)}A(x,u)+\frac{x^2u}{1-xu}(F_T(x)-1)-((1-2x)F_T(x)-1+x)+B(x,u).
\end{align*}
Taking $u=1/(1-x)$, we obtain
\begin{align*}
B_1(x)&=\frac{x^5C^4(x)}{1-2x}+\frac{x^3(1-x)}{(1-2x)^2},
\end{align*}
which leads to
\begin{align*}
(1-x-1/u)B(x,u)&=\frac{x^5u^2C^2(x)}{(1-xu)(1-xuC(x))}-\frac{x^5C^4(x)}{1-2x}\\
&\quad+\frac{x^3(1-x)u}{(1-2x)(1-xu)}-\frac{x^3(1-x)}{(1-2x)^2}.
\end{align*}
Thus, letting $u=C(x)$, we have
$$B(x,C(x))=\frac{x^3(4x^2-3x+1-2x^2C(x))}{(2x^2C(x)-xC(x)-3x+1)(1-2x)^2}.$$
Hence, substituting $u=C(x)$ into the equation for $A(x,u)$, we obtain
$$x^2C^2(x)(F_T(x)-1)-((1-2x)F_T(x)-1+x)+B(x,C(x))=0,$$
which implies the following result.
\begin{theorem}\label{th230a}
Let $T=\{2341,1243,1234\}$. Then
$$F_T(x)=\frac{x(4x^4+3x^3-11x^2+6x-1)C(x)+2x^4-12x^3+16x^2-7x+1}{(1-2x)^2\big(1-4x+2x^2-x(1-3x)C(x)\big)}.$$
\end{theorem}
\vspace*{-2mm}
\qed


For the final three cases, denoting the set of triples involved by $\mathcal{T}$, we recall the generating trees method.
The notion of generating tree to enumerate pattern avoiders was introduced by West \cite{W}.
To enumerate $S_n(T)$  for each $T\in \mathcal{T}$,
we consider the generating forest whose vertices are identified
with $S:=\bigcup_{n\ge 2}S_n(T)$ where
12 and 21 are the roots and each non-root $\pi \in S$ is a child of the permutation obtained from $\pi$ by deleting its largest element. We will show that it is possible to label the vertices
so that if $v_1$ and $v_2$ are any two vertices with the same label and $\ell$ is any label,
then $v_1$ and $v_2$ have the same number of children with label $\ell$.
Indeed, we will specify (i) the labels of the roots, and (ii) a set of succession
rules explaining how to derive from the label of a parent the labels of all of its children.
This will determine a labelled generating forest depending on $T$.

A permutation $\pi=\pi_1\pi_2\cdots\pi_n\in S_n$ determines $n+1$ positions, called \emph{sites}, between its entries. The sites are denoted $1,2,\dots,n+1$ left to right.
In particular, site $i$ is the space between $\pi_{i-1}$ and $\pi_{i}$ if $2\le i \le n$.
Site $i$ in $\pi$ is said to be {\em active} (with respect to $T$) if, by inserting $n+1$ into $\pi$ in
site $i$, we get a permutation in $S_{n+1}(T)$, otherwise \emph{inactive}.
For the next two triples $T$ under consideration and
$\pi \in S_n(T)$, sites 1 and $n+1$ are always active, and if $\pi_n=n$, then site $n$ is active.

Given $T \in \mathcal{T}$ and $\pi\in S_n(T)$,
define $A(\pi)$ to be the set of all active sites for $\pi$, and
define $L(\pi)$ to be the set of active sites lying to the left of $n$. For example, for $T=\{3412,3421,2341\}$, $L(23154)=\{1,2,4\}$ since there are 4 possible sites to insert 6 to the left of $n=5$
and, of these insertions, only $236154 $ is not in $S_6(T)$.

\begin{lemma}\label{lemmaActiven}
If $T$ denotes one of the following two cases and $\pi \in S_n(T)$, then $A(\pi)=L(\pi)\cup\{n+1\}$ unless $\pi_n<n$ and site $n$ is active, in which case $A(\pi)=L(\pi)\cup\{n,n+1\}$.
\end{lemma}
\begin{proof}
No site to the right of $n$ is active except (possibly) site $n$ and (definitely) site $n+1$ for if $n+1$ is inserted after $n$ in a site $\le n-1$, then $n\,(n+1)\,\pi_{n-1}\,\pi_n$ is a 3412 or a 3421, both forbidden.
\end{proof}

\subsection{Case 240: $\{3412,3421,2341\}$}
Here, if site $n$ is inactive, then 1 and $n+1$ are the only active sites iff $\pi_1=n$. In particular, there are at least 3 active sites unless site $n$ is inactive and $\pi_1=n$.

To construct the generating forest, we first specify the labels.
\begin{itemize}
\item Suppose $\pi=\pi_1\pi_2\cdots\pi_n\in S_n(T)$ has  $k$ active sites and $n\geq2$. If $\pi_n=n$, then label $\pi$ by $k$. Otherwise, if the site $n$ is active, then label $\pi$ by $\overline{k}$, and if the site $n$ is inactive, then label it by $\overline{\overline{k}}$.
\end{itemize}

For instance, all 3 sites for $\pi=12$ are active and $\pi_n=n$, so the label for 12 is 3. Also, 12
has three children $312$, $132$
and $123$ with active sites $\{1,3,4\},\,\{1,2,3,4\}$ and $\{1,2,3,4\}$, respectively, hence labels $\bar{3}$, $\bar{4}$ and $4$. All sites for $21$ are active, so its label is $\bar{3}$, and it has three children $321$, $231$ and $213$ with active sites $\{1,3,4\},\,\{1,2,4\}$ and $\{1,2,3,4\}$, respectively, hence labels $\bar{3}$, $\bar{\bar{3}}$ and $4$.

To establish the succession rules, we need to know how the active sites of a child are related to those of its parent. This relationship is given by the following proposition. The proof makes frequent use of Lemma \ref{lemmaActiven} and is left to the reader.
\begin{proposition}
Suppose $\pi \in S_n(T)$ and let $\pi^j$ denote the result of inserting $n+1$ into an active site $j$.

If $j\le n-1$, then for $1\le i \le j$, site $i$ is active in $\pi^j$ iff site $i$ is active in $\pi$;
site $n+1$ is active in $\pi^j$ iff site $n$ is active in $\pi$; site $n+2$ is active in $\pi^j$; all other sites in $\pi^j$ are inactive.

Next, suppose $n$ is an active site in $\pi$.
Then for $1\le i \le n-1$, site $i$ is active in $\pi^n$ iff site $i$ is active in $\pi$; sites $n$ and $n+2$ are active in $\pi^n$; all other sites in $\pi^n$ are inactive. Also, site $n+1$ is active in $\pi$ (as always) and for $1\le i \le n-1$, site $i$ is active in $\pi^{n+1}$ iff site $i$ is active in $\pi$; sites $n,n+1,n+2$ are active in $\pi^{n+1}$; all other sites in $\pi^{n+1}$ are inactive.

Lastly, suppose $n$ is not an active site in $\pi$. Then site $n+1$ is active in $\pi$ and for $1\le i \le n-1$, site $i$ is active in $\pi^{n+1}$ iff site $i$ is active in $\pi$; sites $n+1,n+2$ are active in $\pi^{n+1}$; all other sites in $\pi^{n+1}$ are inactive. \qed
\end{proposition}

\begin{corollary}\label{cor240a1}
The labelled generating forest $\mathcal{F}$ is given by
$$\begin{array}{lll}
\mbox{\bf Roots: }&3,\bar{3}\\
\mbox{\bf Rules: }&k\rightsquigarrow \bar{3},\bar{4},\ldots,\overline{k+1},k+1 & \textrm{\quad for $k\ge 3$,}\\
&\bar{k}\rightsquigarrow\bar{3},\bar{4},\ldots,\bar{k},\bar{\bar{k}},k+1 &  \textrm{\quad for $k\ge 3$,}\\
&\bar{\bar{k}}\rightsquigarrow\bar{\bar{2}},\bar{\bar{3}},\bar{\bar{4}},\ldots,\bar{\bar{k}},k+1  & \textrm{\quad for $k\ge 2$.}
\end{array}$$
\end{corollary}
\begin{proof}
We treat the double-bar case (the others are similar).
So suppose $\pi \in S_n(T)$ has label $\bar{\bar{k}},\ k\ge 2$.
This means that site $n$ is not active, $L(\pi)$ has $k-1$ entries, all $\le n-1$, and
$A(\pi) \backslash L(\pi)=\{n+1\}$. Say $L(\pi)=\{L_1=1<L_2<\dots<L_{k-1}\le n-1\}$. Then for $1\le i \le k-1$, $\pi^{L_i}$ does not end with its largest letter and its active sites are $\{L_1,\dots,L_i,n+2\}$. In particular,
its penultimate site, namely $n+1$, is not active. Hence, its label is $\overline{\overline{i+1}}$. Also,
$\pi^{n+1}$ \emph{does} end with its largest letter and its active sites are $\{L_1,\dots,L_{k-1},n+1,n+2\}$ and so its label is $k+1$. All told, the labels of the children of $\pi$ are $\bar{\bar{2}},\bar{\bar{3}},\bar{\bar{4}},\ldots,\bar{\bar{k}},k+1$.
\end{proof}

\begin{theorem}\label{th240a}
Let $T=\{3412,3421,2341\}$. Then
\[
F_T(x)=1 + \frac{ 1 + r t}{ 1 + 1/r + 1/t - 1/x}\, ,
\]
where
\[
r=r(x)= \frac{\left(\sqrt{5}-1\right) \left(1-\sqrt{1-\left(3+\sqrt{5}\right) x+\frac{1}{2} \left(3-\sqrt{5}\right)
   x^2}\right)-\left(1+\sqrt{5}\right) x}{4 x}
\] and $t$ is the same as $r$ but with $\sqrt{5}$ replaced by $-\sqrt{5}$.
\end{theorem}
\begin{proof}
Let $a_k(x)$, $b_k(x)$ and $c_k(x)$ be the generating functions for the number of permutations in the $n$th level of the labelled generating forest $\mathcal{F}$ with label $k$, $\bar{k}$ and $\bar{\bar{k}}$, respectively.
By Corollary \ref{cor240a1}, we have
\begin{align*}
a_k(x)&=x^2\delta_{k=3}+xa_{k-1}(v)+xb_{k-1}(v)+xc_{k-1}(v),\\
b_k(x)&=x^2\delta_{k=3}+v^2b_2(x)+x(b_k(x)+b_{k+1}(x)+\cdots)+x(a_{k-1}(x)+a_k(x)+\cdots),\\
c_k(x)&=xb_k(x)+x(c_k(x)+c_{k+1}(x)+\cdots),\\
b_2(x)&=xb_2(x)+x(c_3(x)+c_4(x)+\cdots),
\end{align*}
for all $k\geq3$.

Now let $A(x,v)=\sum_{k\geq3}a_k(x)v^k$, $B(x,v)=\sum_{k\geq2}b_k(x)v^k$ and $C(x,v)=\sum_{k\geq3}c_k(x)v^k$. Hence, the above recurrences can be written as
\begin{align}
A(x,v)&=x^2v^3+xvA(x,v)+xvB(x,v)+xvC(x,v),\label{eq240ax1}\\
B(x,v)&=x^2v^3+v^2b_2(x)+\frac{x}{1-v}(v^3B(x,1)-vB(x,v))+\frac{x}{1-v}(v^3A(x,1)-v^2A(x,v)),\label{eq240ax2}\\
C(x,v)&=x(B(x,v)-v^2b_2(x))+\frac{x}{1-v}(v^3C(x,1)-vC(x,v)),\label{eq240ax3}
\end{align}
where $b_2(x)=\frac{x}{1-x}C(x,1)$.

By finding $C(x,v)$ from \eqref{eq240ax1}, and then substituting it in \eqref{eq240ax2}-\eqref{eq240ax3}, we obtain
\begin{align*}
&\left(1+\frac{xv}{1-v}\right)B(x,v)+\frac{xv^2}{1-v}A(x,v)\\
&\qquad\qquad\qquad=\frac{(1-v+xv)v^2}{1-v}A(x,1)-\frac{xv^2(1-2v+xv)}{(1-x)(1-v)}B(x,1)
-\frac{x^2v^2(1-v+xv)}{1-x},\\
&-\left(1+x+\frac{xv}{1-v}\right)B(x,v)+\frac{(1-xv)(1-v+xv)}{xv(1-v)}A(x,v)\\
&\qquad\qquad\qquad=\frac{(v-x)v^2}{1-v}A(x,1)+\frac{xv^2(x-v)}{(1-x)(1-v)}B(x,1)
+\frac{xv^2(1-x+x^2)}{1-x}.
\end{align*}
Multiplying the first equation by $1+x+\frac{xv}{1-v}$ and the second by $1+\frac{xv}{1-v}$, then summing the results, we have
\begin{align}\label{eq240K}
\begin{split}
K(x,v)A(x,v)=&\:-(1-v+xv)xv^3A(x,1)+\frac{x^2v^3(1-v)^2}{1-x}B(x,1)\\
&-\frac{x^2v^3(1-v)(1-v+xv)(xv-2x+1)}{1-x} ,
\end{split}
\end{align}
where $K(x,v)=x^2v^4+x(1-3x)v^3+(x^2-1)v^2+(2-x)v-1$.

There are two power series $v'$ and $v''$ such that $K(x,v')=K(x,v'')=0$ given by $v'=r+1$ and $v''=t+1$, where $r$ and $t$ are as stated above. For example, the expansion of $v'$ begins
$$v'=1+x+\frac{1}{2}(3+\sqrt{5})x^2+(4+2\sqrt{5})x^3+(15+7\sqrt{5})x^4+\frac{1}{2}(119+53\sqrt{5})x^5+\cdots\, .$$
Substituting $v'$ and $v''$ for $v$ in (\ref{eq240K}) gives a pair of equations for $A(x,1)$ and $B(x,1)$ with solution
\begin{align*}
A(x,1)&=\frac{(v''-1)(v'-1)(v'v''-v'-v''+2)x^2}{(x-1)v'v''+v'+v''-x-1},\\
B(x,1)&=\frac{-(xv''-v''+1)(xv'-v'+1)(xv'+xv''-3x+1)}{(x-1)v'v''+v'+v''-x-1},
\end{align*}
which, by \eqref{eq240ax1}, implies
$$C(x,1)=\frac{(1-x)((v''-1)(v'-1)+(v'^2v''^2-v'^2v''-v''^2v'+1)x-v'v''(v'+v''-3)x^2)}{(x-1)v'v''+v'+v''-x-1}.$$
Since $F_T(x) = 1+x+A(x,1)+B(x,1)+C(x,1)$, the result follows.
\end{proof}

\subsection{Case 109: $\{3412,3421,2143\}$}
In this case,
if site $n$ is inactive, then 1 and $n+1$ are the only active sites iff $\pi_1\ge 3$.
If site $i$ is active for $i\le n-1$, then all sites $\le i$ are active. Hence, for $n\ge 2$,
the active sites $\le n-1$ form a nonempty initial segment of the positive integers.

For $n\ge 2$, say $\pi \in S_n$ is \emph{mostly increasing} if it has exactly one descent and
the descent bottom is the last entry, that is, if $\pi$
has the form $\pi=12 \cdots(j-1)(j+1) \cdots nj$ for some $j$ with $1\le j \le n-1$.

We now assign labels. So suppose $n\ge 2$ and $\pi \in S_n(T)$ has $k$ active sites.

If $\pi$ is increasing ($\pi= 12\cdots n$), then all sites are active, so $k=n+1$ and label $\pi$ by $k$.
If $\pi$ is mostly increasing ($\pi= 12\cdots(j-1)(j+1)\cdots nj$), then once again all sites are active,
so $k=n+1$ and label $\pi$ by $\bar{k}$.

Now suppose $\pi$ is neither increasing nor mostly increasing.
If site $n$ is active, label $\pi$ by $\overline{\overline{k}}$.
If site $n$ is inactive, label $\pi$ by $\overline{\overline{\overline{k}}}$.

For instance, all 3 sites are active for both $12$ and $21$ and $12$ is increasing while $21$ is mostly increasing, so their labels are 3 and $\bar{3}$ respectively;
$12$ has three children $312$, $132$ and $123$ with active sites $\{1,3,4\},\,\{1,2,3,4\}$ and $\{1,2,3,4\}$, respectively,
hence labels $\bar{\bar{3}}$, $\bar{4}$ and 4 because the second is mostly increasing and the third is increasing;
$21$ has three children $321$, $231$ and $213$ with active sites $\{1,3,4\},\,\{1,2,3,4\}$ and $\{1,2,4\}$, respectively, hence labels  $\bar{\bar{3}}$, $\bar{4}$ and $\overline{\overline{\overline{3}}}$.

To establish the succession rules, we have the following proposition. The proof is left to the reader.
\begin{proposition}
Suppose $n\ge 2$ and $\pi \in S_n(T)$ has $k$ active sites.

If $\pi$ is increasing, then all sites $\{1,2,\dots,n+1\}$ are active $($hence, $k=n+1)$ and
\[
A(\pi^i)=
\begin{cases}
\{1,2,\dots,i,n+1,n+2\} & \textrm{ if $1\le i \le n$,} \\
\{1,2,\dots,n,n+1,n+2\} & \textrm{ if $i=n+1$.}
\end{cases}
\]
If $\pi$ is mostly increasing, then all sites $\{1,2,\dots,n+1\}$ are active $($hence, $k=n+1)$ and
\[
A(\pi^i)=
\begin{cases}
\{1,2,\dots,i,n+1,n+2\} & \textrm{ if $1\le i \le n$,} \\
\{1,2,\dots,n,n+2\} & \textrm{ if $i=n+1$.}
\end{cases}
\]
Now suppose $\pi$ is neither increasing nor mostly increasing.

If site $n$ is active so that the active sites for $\pi$ are $\{1,2,\dots,k-2,n,n+1\}$, then
\[
A(\pi^i)=
\begin{cases}
\{1,2,\dots,i,n+1,n+2\} & \textrm{ if $1\le i \le k-2$,} \\
\{1,2,\dots,k-2,n+1,n+2\} & \textrm{ if $i=n$,} \\
\{1,2,\dots,k-2,n+2\} & \textrm{ if $i=n+1$.}
\end{cases}
\]
If site $n$ is inactive so that the active sites for $\pi$ are $\{1,2,\dots,k-1,n+1\}$, then
\[
A(\pi^i)=
\begin{cases}
\{1,2,\dots,i,n+2\} & \textrm{ if $1\le i \le k-1$,} \\
\{1,2,\dots,k-1,n+2\} & \textrm{ if $i=n+1$.}
\end{cases}
\]
\qed
\end{proposition}
\vspace*{-4mm}
An immediate consequence is
\begin{corollary}\label{cor109a1}
The labelled generating forest $\mathcal{F}$ is given by
$$\begin{array}{lll}
\mbox{\bf Roots: }&3,\overline{3}\\
\mbox{\bf Rules: }&k\rightsquigarrow \overline{\overline{3}},\overline{\overline{4}},\ldots,\overline{\overline{k}},\overline{k+1},k+1, & \textrm{\quad for $k\ge 3$,}\\
&\overline{k}\rightsquigarrow\overline{\overline{3}},\overline{\overline{4}},\ldots,\overline{\overline{k}},\overline{k+1},\overline{\overline{\overline{k}}} & \textrm{\quad for $k\ge 3$,}\\
&\overline{\overline{k}}\rightsquigarrow\overline{\overline{3}},\overline{\overline{4}},\ldots,\overline{\overline{k-1}},\overline{\overline{k}},\overline{\overline{k}},\overline{\overline{\overline{k-1}}} & \textrm{\quad for $k\ge 3$,}\\
&\overline{\overline{\overline{k}}}\rightsquigarrow\overline{\overline{\overline{2}}},\overline{\overline{\overline{3}}},\ldots,\overline{\overline{\overline{k-1}}},\overline{\overline{\overline{k}}},\overline{\overline{\overline{k}}} & \textrm{\quad for $k\ge 2$.}
\end{array}$$
\end{corollary}

\begin{theorem}\label{th109a}
Let $T=\{3412,3421,2143\}$. Then
$$F_T(x)=\frac{1-8x+23x^2-27x^3+12x^4-5x^5}{(1-3x+x^2)^3}.$$
\end{theorem}
\begin{proof}
Let $a_k(x)$, $b_k(x)$, $c_k(x)$ and $d_k(x)$ be the generating functions for the number of permutations in the $n$th level of the labelled generating forest $\mathcal{F}$ with label $k$, $\overline{k}$, $\overline{\overline{k}}$ and $\overline{\overline{\overline{k}}}$, respectively.
By Corollary \ref{cor109a1}, we have
\begin{align*}
a_k(x)&=xa_{k-1}(x),\quad k\geq4,\\
b_k(x)&=x(a_{k-1}(x)+b_{k-1}(x)),\quad k\geq4,\\
c_k(x)&=xc_k(x)+x\sum_{j\geq k}(c_j(x)+b_j(x)+a_j(x)),\quad k\geq3,\\
d_k(x)&=xd_k(x)+xb_k(x)+xc_{k+1}(x)+x\sum_{j\geq k}d_j(x),\quad k\geq2,
\end{align*}
with $a_3(x)=b_3(x)=x^2$. Clearly, $a_k(x)=x^{k-1}$ and $b_k(x)=(k-2)x^{k-2}$ for all $k\geq3$.

Now let $A(x,v)=\sum_{k\geq3}a_k(x)v^k$, $B(x,v)=\sum_{k\geq3}b_k(x)v^k$, $C(x,v)=\sum_{k\geq3}c_k(x)v^k$ and  $D(x,v)=\sum_{k\geq2}d_k(x)v^k$. Thus, $A(x,v)=\frac{x^2v^3}{1-xv}$ and $B(x,v)=\frac{x^3v^3}{(1-xv)^2}$.
Hence, the above recurrences can be written as
\begin{align}
\left(1-x+\frac{xv}{1-v}\right)C(x,v)&=\frac{xv^3}{1-v}C(x,1)+\frac{x^3v^3(2-x-xv)}{(1-x)^2(1-xv)^2},\label{eq109ax1}\\
\left(1-x+\frac{xv}{1-v}\right)D(x,v)&=xB(x,v)+\frac{x}{v}C(x,v)+\frac{xv^2}{1-v}D(x,1).\label{eq109ax2}
\end{align}
To solve the first functional equation, we apply the kernel method and take $v=\frac{1-x}{1-2x}$. This gives
$$C(x,1)=\frac{x^3(2-6x+3x^2)}{(1-x)^2(1-3x+x^2)^2}.$$
Multiplying \eqref{eq109ax2} by $1-x+\frac{xv}{1-v}$, and using \eqref{eq109ax1}, yields
\begin{align*}
&\left(1-x+\frac{xv}{1-v}\right)^2D(x,v)=x\left(1-x+\frac{xv}{1-v}\right)B(x,v)\\
&\qquad+\frac{x}{v}\left(\frac{xv^3}{1-v}C(x,1)+\frac{x^3v^3(2-x-xv)}{(1-x)^2(1-xv)^2}\right)+\frac{xv^2}{1-v}\left(1-x+\frac{xv}{1-v}\right)D(x,1).
\end{align*}
Differentiating this functional equation with respect to $v$ and then substituting $v=\frac{1-x}{1-2x}$, we have $$D(x,1)=\frac{x^3(1-2x-x^3)}{(1-3x+x^2)^3}.$$
The result now follows from the fact $F_T(x)=1+x+A(x,1)+B(x,1)+C(x,1)+D(x,1)$.
\end{proof}

\subsection{Case 188: $\{1432,2143,3214\}$}
To find $F_T(x)$, we modify the generating trees of the last two cases as follows.

For any set of patterns $R$, to enumerate $S_n(R)$ one may consider the generating forest whose vertices are identified with $S:=\bigcup_{n\ge 2}S_n(R)$ where 12 and 21 are the roots and each non-root $\pi \in S$ is a child of the permutation obtained from $\pi$ by deleting its largest element. We will show that it is possible to label the vertices by irreducible permutations (as defined below) so that if $c_1$ and $c_2$ are any two vertices, then $c_1$ and $c_2$ have the same number of children. Indeed, we will specify (i) the irreducible permutations of the roots, and (ii) a set of succession rules explaining how to derive from the irreducible permutation of a parent the labels (presented by irreducible permutations) of all of its children. This will determine a generating forest depending on $R$.

For $\sigma=\sigma_1\cdots\sigma_k\in S_k(R)$, we say that $\sigma$ is reducible if there is a subsequence $\sigma'$ of length $k-1$ such that by inserting $k$ in some position in the set $A(\sigma')$ of active sites we obtain $\sigma$. Otherwise, $\sigma$ is irreducible. Thus, to specify the generating
forest $\mathcal{F}$, we need to know the set of irreducible permutations and the succession rules
(what irreducible permutations are obtained by inserting a letter in a given irreducible permutation).

For instance, if $R=\{132\}$, then $12$ is reducible because if $\pi=\pi'1\pi''2\pi'''$ avoids $132$ then
$\pi''=\emptyset$, which leads to the conclusion that $\pi$ avoids $132$ if and only if $\pi'1\pi''$ avoids $132$, thus $2$ is a letter verifying the reducibility of $12$. In general, it is not hard to see that there is exactly one irreducible permutation in $S_k(132)$, namely, $k\cdots21$. Moreover, by inserting $k+1$ in $k\cdots 21$, we obtain the following irreducible permutations $(k+1)k\cdots21, k\cdots21, \ldots,1$, thus we have the following succession rule
$$k\cdots21\rightsquigarrow (k+1)k\cdots21,k\cdots21,\ldots,1,$$
for all $k\geq1$ (as expected, see \cite{W}).

Now, we are ready to consider our case, $T=\{1432,2143,3214\}$. The irreducible permutations and the succession rules are given in the next proposition; the proof is left to the reader.

\begin{proposition}\label{pro188a}
The irreducible permutations of $S_k(T)$ are given by $\alpha_{i,k}=i(i+1)\cdots k12\cdots(i-1)$, $\alpha'_{i,k}=i(i+1)\cdots(k-2)1(k-1)2k34\cdots(i-1)$, $\alpha''_{i,k}=i(i+1)\cdots(k-1)1k23\cdots(i-1)$, $\beta_k=1k23\cdots(k-1)$, $\beta'_k=1(k-1)2k34\cdots(k-2)$, $\gamma_k=23\cdots(k-1)1k$, $\gamma'_k=2k34\cdots(k-2)1(k-1)$, and $\delta_k=67\cdots k 14253$.

The generating forest $\mathcal{F}$ is given by
$$\begin{array}{lll}
\mbox{\bf Roots: }&\alpha_{1,2},\alpha_{2,2}\\
\mbox{\bf Rules: }
&\alpha_{1,k}\rightsquigarrow\alpha_{1,k+1}\beta_3\cdots\beta_{k+1}\alpha_{k+1,k+1},&\textrm{for $k\geq2$},\\
&\alpha_{i,k}\rightsquigarrow\gamma_{k+3-i}\alpha''_{3,k+4-i}\cdots\alpha''_{i,k+1}
\alpha_{i,k+1}\beta_3\cdots\beta_{k+2-i}\alpha_{k+2-i,k+2-i},&\textrm{for $2\leq i\leq k$},\\
&\alpha'_{i,k}\rightsquigarrow\alpha'_{3,k+3-i}\cdots\alpha'_{i,k}(361425)\gamma'_5\cdots\gamma'_{k+2-i}\alpha_{k-i,k-i},&\textrm{for $3\leq i\leq k-2$},\\
&\alpha''_{i,k}\rightsquigarrow\alpha'_{3,k+4-i}\cdots\alpha'_{i,k+1}\alpha''_{i,k}\beta_3\gamma'_5\cdots\gamma'_{k+3-i}\alpha_{k+1-i,k+1-i},&\textrm{for $3\leq i\leq k-1$},\\
&\beta_k\rightsquigarrow\gamma_{3}\beta'_5\cdots\beta'_{k+1}\beta_k\alpha_{2,2},&\textrm{for $k\geq3$},\\
&\beta'_k\rightsquigarrow(142536)\beta'_5\cdots\beta'_k\delta_6,&\textrm{for $k\geq5$},\\
&\gamma_k\rightsquigarrow\gamma_k\gamma_3\gamma'_5\cdots\gamma'_{k+1}\alpha_{k-1,k-1},&\textrm{for $k\geq3$},\\
&\gamma'_k\rightsquigarrow(253614)\beta'_5\cdots\beta'_{k-1}\gamma'_k\alpha_{2,2},&\textrm{for $k\geq5$},\\
&\delta_k\rightsquigarrow\delta_{k+1}\beta_3\cdots\beta_{k-4}\alpha_{k-4,k-4},&\textrm{for $k\geq6$},\\
&(361425)\rightsquigarrow(361425)\alpha_{2,2},\\
&(142536)\rightsquigarrow(142536)\delta_6,\\
&(253614)\rightsquigarrow(253614)\delta_6.
\end{array}$$
\end{proposition}

Let $A_{i,k}=A_{i,k}(x)$ be the generating function for the number of nodes $\alpha_{i,k}$ in $\mathcal{F}$ at level $n$ (the roots are at level $2$). Similarly, we define $A'_{i,k}$, $A''_{i,k}$, $B_k$, $B'_k$, $G_k$, $G'_k$, $D_k$, $L$, $L'$ and $L''$, for the number of nodes $\alpha'_{i,k}$, $\alpha''_{i,k}$, $\beta_k$, $\beta'_k$, $\gamma_k$, $\gamma'_k$, $\delta_k$, $361425$, $142536$ and $253614$ in $\mathcal{F}$ at level $n$, respectively.

Define $A_k(v)=\sum_{i=1}^kA_{i,k}v^{k-i}$, $A'_k(v)=\sum_{i=3}^{k-2}A'_{i,k}v^{k-2-i}$ and
$A''_k(v)=\sum_{i=3}^{k-1}A_{i,k}v^{k-1-i}$. Also, define
$A(v,w)=\sum_{k\geq2}A_k(v)w^{k-2}$, $A'(v,w)=\sum_{k\geq5}A'_k(v)w^{k-5}$, $A''(v,w)=\sum_{k\geq4}A''_k(v)w^{k-4}$, $B(w)=\sum_{k\geq3}B_kw^{k-3}$, $B'(w)=\sum_{k\geq5}B'_kw^{k-5}$, $G(w)=\sum_{k\geq3}G_kw^{k-3}$, $G'(w)=\sum_{k\geq5}G'_kw^{k-5}$, $D(w)=\sum_{k\geq6}D_kw^{k-6}$, $L(w)=\sum_{k\geq6}L_kw^{k-6}$, $L'(w)=\sum_{k\geq6}L'_kw^{k-6}$, $L''(w)=\sum_{k\geq6}L''_kw^{k-6}$.

Proposition \ref{pro188a} leads to the following system of equations:
\begin{align*}
A(v,w)=&\:vx^2+vwxA(v,w)+A(0,w),\\
A(0,w)=&\:\frac{x^2}{1-wx}+xL+x(B(1)+G'(1))+xG(w)+xD(w)\\
&+x(A(w,1)-wx^2/(1-wx))+x(A'(w,1)+A''(w,1)),\\
A'(v,w)=&\:\frac{x}{1-w}(A'(vw,1)-wA'(v,w))+\frac{x}{1-w}(A''(vw,1)-wA''(v,w)),\\
(1-x)A''(v,w)=&\:\frac{x}{(1-w)(1-vw)}(A(0,1)-A(0,w))-\frac{vx}{(1-v)(1-vw)}(A(0,w)-A(0,vw)),\\
B(w)=&\:\frac{x}{1-w}(A(1,1)-A(w,1))+xA''(1,1)+xB(w)+\frac{x}{1-w}(D(1)-D(w)),\\
B'(w)=&\:\frac{x}{1-w}(B(1)-B(w))+\frac{x}{1-w}(B'(1)-wB'(w))+\frac{x}{1-w}(G'(1)-G'(w)),\\
G(w)=&\:xG(w)+xG(1)+xB(1)+x(A(w,1)-wx^2/(1-wx)),\\
(1-x)G'(w)=&\:\frac{x}{1-w}(A'(1,1)-A'(1/w,w))+\frac{x}{1-w}(A''(1,1)-A''(1/w,w))\\
&+\frac{x}{1-w}(G(1)-G(w)),\\
(1-x)D(w)=&\:\frac{x}{1-wx}(B'(1)+xG'(1)),\\
(1-x)L=&\:xA'(1,1),\\
(1-x)L'=&\:xB'(1),\\
(1-x)L''=&\:xG'(1).
\end{align*}
Let $K=5x^8-44x^7+128x^6-208x^5+209x^4-132x^3+51x^2-11x+1$. Then,
by computer algebra, one can show the solution of the above system is given by
\begin{align*}
A(v,w)&=\frac{x^2(x-1)^5(wx^3+4x^3-9x^2+5x-1)+vx^2K(1-wx)}{(1-wx)(1-vwx)K},\\
A'(v,w)&=\frac{x^5(1-2x)(1-x)^4}{(1-wx)(1-vwx)K},\\
A''(v,w)&=\frac{x^4(1-2x)^2(1-x)^4}{(1-wx)(1-vwx)K},\\
B(w)&=\frac{x^3(1-x)(1-7x+19x^2-(w+25)x^3+2(2w+7)x^4-(4w+1)x^5)}{(1-wx)K},\\
B'(w)&=\frac{x^5(1-2x)(1-x)^4}{(1-wx)K},\\
G(w)&=\frac{x^3(1-x)^2(1-4x-2(w-4)x^2+3(2w-3)x^3+5(1-w)x^4)}{(1-wx)K},\\
G'(w)&=\frac{x^5(1-3x+3x^2)(1-x)^2}{(1-wx)K},\phantom{h\hspace*{50mm}h}\\
D(w)&=\frac{x^6(1-3x+2x^2+x^3)}{(1-wx)K},
\end{align*}
\begin{align*}
L&=\frac{x^6(1-2x)(1-x)}{K},\\
L'&=\frac{x^6(1-2x)(1-x)^2}{K},\\
L''&=\frac{x^6(1-3x+3x^2)}{K}.
\end{align*}
Hence, by Proposition \ref{pro188a}, we have
$$F_T(x)=1+x+L+L'+L''+D(1)+G'(1)+G(1)+B'(1)+B(1)+A''(1,1)+A'(1,1)+A(1,1),$$
which leads to the following result.
\begin{theorem}\label{th188a}
Let $T=\{1432,2143,3214\}$. Then
$$F_T(x)=\frac{(1-x)^4(1-6x+12x^2-9x^3+x^4)}{1-11x+51x^2-132x^3+209x^4-208x^5+128x^6-44x^7+5x^8}.$$
\end{theorem}

\textbf{Acknowledgement:}  We wish to thank Vince Vatter for pointing out to us the reference \cite{AA20160600}.

\section{Appendix}\label{App}
{\footnotesize\begin{longtable}[c]{|l||l|}
\caption{Small Wilf classes of three 4-letter patterns counted by INSENC.\label{longinsenc}}\\ \hline
\multicolumn{2}{| c |}{Begin of Table}\\ \hline
No. &$F_T(x)$ \\ \hline
\endfirsthead  \hline
\multicolumn{2}{|c|}{Continuation of Table \ref{longinsenc}}\\ \hline
No. & $F_T(x)$ \\ \hline
\endhead \hline
\endfoot \hline
\multicolumn{2}{| c |}{End of Table}\\ \hline\hline
\endlastfoot
1&$F_{\{4321,3412,1234\}}(x)=73x^9+ 199x^8+ 240x^7+ 162x^6+ 69x^5+ 21x^4+ 6x^3+ 2x^2+ x + 1$\\\hline
2&$F_{\{4321,3142,1234\}}(x)=85x^9+ 221x^8+ 252x^7+ 164x^6+ 69x^5+ 21x^4+ 6x^3+ 2x^2+ x + 1$\\\hline
3&$F_{\{2143,4312,1234\}}(x)=\frac{18x^7+31x^6+22x^5+8x^4+2x^3+2x^2-2x+1}{(1-x)^{3}}$\\\hline
4&$F_{\{4231,2143,1234\}}(x)=\frac{2x^{10}-6x^9+6x^8+4x^7+4x^6+8x^5+6x^4-4x^3+7x^2-4x+1}{(1-x)^{5}}$\\\hline
5&$F_{\{2143,3412,1234\}}(x)=\frac{2x^5+10x^4-11x^3+11x^2-5x+1}{(1-x)^6}$\\\hline
7&$F_{\{3421,4312,1234\}}(x)=\frac{-9x^7+24x^6+23x^5+8x^4+2x^3+2x^2-2x+1}{(1-x)^3}$\\\hline 8&$F_{\{2431,4213,1234\}}(x)=\frac{26+21x+15x^2}{2(1-x-x^2-x^3)}+\frac{4x^{10}(1+3x)-62x^9-28x^8+66x^7+27x^6-15x^5-53x^4+41x^3+51x^2-75x+24}{(x-1)^3(1-x-x^2)^2}$\\\hline
9&$F_{\{2134,4312,1243\}}(x)=\frac{-3x^7-5x^6+3x^5+10x^4-11x^3+11x^2-5x+1}{(1-x)^6}$\\\hline 10&$F_{\{4213,1432,1234\}}(x)=\frac{2x^{11}+4x^{10}+10x^9+12x^8+6x^7-19x^6-19x^5-7x^4-x^3+2x-1}{(x-1)(x^5+3x^4+2x^3+x^2+x-1)(x^3+x^2+x-1)}$\\\hline
11&$F_{\{4231,1432,1234\}}(x)=\frac{5x^9-2x^8-x^7+9x^6+9x^5+6x^4-4x^3+7x^2-4x+1}{(1-x)^5}$\\\hline
12&$F_{\{2341,4312,1324\}}(x)=\frac {x^{10}-4x^9+3x^8+ 5x^7-7x^5+ 21x^4- 22x^3+ 16x^2- 6x+ 1}{(1-x)^{7}}$\\\hline
13&$F_{\{3214,1432,1234\}}(x)=-\frac{x^5+x^3+x^2+x-1}{x^{12}+16x^{11}+10x^{10}+17x^9+25x^8+25x^7-7x^6-14x^5-5x^4-2x^3-x^2-2x+1}$\\\hline
14&$F_{\{4231,2134,1243\}}(x)=\frac{4x^9-11x^8+10x^7+2x^6-7x^5+21x^4-22x^3+16x^2-6x+1}{(1-x)^7}$\\\hline 16&$F_{\{2314,1432,4123\}}(x)=\frac{2x^9-x^8+x^7+3x^6+6x^5-6x^4+11x^3-13x^2+6x-1}{(x^2-3x+1)(x^3+x^2+x-1)(1-x)^3}$\\\hline
17&$F_{\{2341,2143,4123\}}(x)=\frac{x^7- 13x^5+ 25x^4 - 29x^3+ 20x^2- 7x + 1}{(x^2- 3x + 1)(1-x)^5}$\\\hline
18&$F_{\{2341,1432,4123\}}(x)=\frac{x^{10}-7x^9+19x^8-25x^7+12x^6+10x^5-20x^4+25x^3-19x^2+7x-1}{(x^2+1)(x^2-3x+1)(x^3-x^{2}-2x+1)(x-1)^3}$\\\hline
19&$F_{\{2431,4312,1234\}}(x)=\frac{6x^9- 7x^8- 7x^7 + 4x^6+ 10x^5+ 6x^4- 4x^3+ 7x^2- 4x+ 1}{(1-x)^5}$\\\hline
20&$F_{\{4312,1432,1234\}}(x)=\frac{(x + 1)(2x^9- 18x^8+ 33x^7- 20x^6+ 12x^5- 22x^4+ 16x^3- 12x^2+ 5x -1)}{(x-1)^5}$\\\hline
21&$F_{\{4312,3142,1234\}}(x)=\frac{2x^9- 3x^8- 2x^{6} - 6x^5+ 21x^4- 22x^3+ 16x^2- 6x + 1}{(1-x)^7}$\\\hline
22&$F_{\{2134,4312,1432\}}(x)=\frac {x^6+ 6x^5- 21x^4+22x^3- 16x^2+ 6x - 1}{(x - 1)^7}$\\\hline
23&$F_{\{2431,4132,1234\}}(x)=\frac{1-2x}{x^2-3x+1}+\frac{(24x^9-116x^8+213x^7-158x^6+9x^5+37x^4-9x^3+x^2-3x+1)x^3}{(x-1)^5(2x-1)^3}$\\\hline
24&$F_{\{4231,3412,1234\}}(x)=\frac{2x^6- 6x^5+ 21x^4- 22x^3+ 16x^2- 6x + 1}{(1-x)^7}$\\\hline
25&$F_{\{3412,4132,1234\}}(x)=\frac{3x^6- 6x^5+ 21x^4- 22x^3+ 16x^2- 6x + 1}{(1-x)^7}$\\\hline
26&$F_{\{2134,4312,1342\}}(x)=-\frac {x^8- 9x^6+ 27x^5- 43x^4+ 38x^3- 22x^2+ 7x - 1}{(1-x)^8}$\\\hline
27&$F_{\{2314,4312,1432\}}(x)=-\frac{3x^9-x^8-18x^7+17x^6+15x^5-44x^4+47x^3-27x^2+8x-1}{(2x-1)(x^2+x-1)(x-1)^6}$\\\hline
28&$F_{\{4231,3142,1234\}}(x)=\frac{2x^8- 10x^7+ 40x^6- 70x^5+ 81x^4- 60x^3+ 29x^2- 8x + 1}{(1-x)^9}$\\\hline
31&$F_{\{2314,4312,1342\}}(x)=\frac{5x^{10}-22x^9+12x^8+89x^7-249x^6+354x^5-316x^4+179x^3-62x^2+12x-1}{(x^2-3x+1)(2x-1)^3(x-1)^4}$\\\hline
32&$F_{\{2134,1432,4123\}}(x)=\frac {x^{10} - 4x^9+ 4x^8-x^6- 5x^5+ 6x^4- 11x^3+ 13x^2- 6x + 1}{(x^2- 3x + 1)(x^3+ x^2+x- 1)(x - 1)^3}$\\\hline
33&$F_{\{2134,3412,4132\}}(x)=\frac {2x^7- 16x^5+ 36x^4 - 42x^3+ 26x^2- 8x + 1}{(2x - 1)^{3}(x - 1)^3}$\\\hline
34&$F_{\{2143,4132,1234\}}(x)=-\frac {x^9- 2x^8- x^7+ 4x^6- x^5- 2x^4+ 3x^3- 8x^2+ 5x - 1}{(x^2 - 3x + 1)(x -1)(2x-1)}$\\\hline
35&$F_{\{1324,2143,3412\}}(x)=\frac{1-9x+33x^2-62x^3+64x^4-38x^5+10x^6}{(1-3x+x^2)(1-2x)^2(1-x)^3}$\\\hline
36&$F_{\{3412,3124,1432\}}(x)=-\frac {x^8+ 2x^7- 26x^6+ 62x^5- 83x^4+ 69x^3- 34x^2+ 9x - 1}{(x-1)^{5}(x^2- 3x + 1)(2x-1)}$\\\hline
37&$F_{\{3142,1432,1234\}}(x)=\frac {(x^3- 2x^2+ 3x - 1)^2}{x^8- x^7+ 4x^6- 7x^5+ 19x^4- 24x^3 + 18x^2- 7x+1}$\\\hline
38&$F_{\{4321,1423,1234\}}(x)=147x^9+ 359x^8+ 367x^7+ 198x^6+72x^5+ 21x^4+ 6x^3+ 2x^2+ x + 1$\\\hline
39&$F_{\{4321,4123,1234\}}(x)=185x^9+ 400x^8+ 396x^7+ 205x^6+72x^5+ 21x^4+ 6x^3+ 2x^2+ x + 1$\\\hline
40&$F_{\{2341,4312,1234\}}(x)=\frac{x^9- 5x^8+ 6x^7+ x^6+ 5x^5- 21x^4+ 22x^3- 16x^2+ 6x - 1}{(x - 1)^7}$\\\hline
41&$F_{\{4312,1342,1234\}}(x)=-\frac {2x^7- 8x^6+ 26x^5- 43x^4+ 38x^3- 22x^2+ 7x - 1}{(x - 1)^8}$\\\hline
42&$F_{\{2341,4132,1234\}}(x)=\frac{4x^8- 5x^7- 7x^{6} - 7x^5+ 22x^4- 28x^3+ 20x^2- 7x + 1}{(2x-1)^2(x-1)^4}$\\\hline
43&$F_{\{2314,4213,1432\}}(x)=-\frac {9x^6- 35x^5+ 54x^4- 49x^3+ 27x^2- 8x + 1}{(3x^3- 5x^2+ 4x - 1)(2x -1)(x-1)^3}$\\\hline
44&$F_{\{4213,1342,1234\}}(x)=\frac {x^{10} - 6x^9+ 9x^8+9x^7- 54x^6+ 94x^5- 104x^4+ 76x^3- 35x^2+ 9x - 1}{(x^3- 2x^2+3x- 1)(2x - 1)(x -1)^5}$\\\hline
45&$F_{\{4213,2134,1432\}}(x)=\frac{x^{10} - 2x^9- x^8- 13x^7+ 54x^6- 99x^5+ 108x^4- 77x^3+ 35x^2- 9x + 1}{(x-1)^{2}(3x^3- 5x^2+ 4x - 1)^2}$\\\hline
46&$F_{\{2341,4132,1324\}}(x)=\frac{2x^7+ 5x^6- 3x^{5}+ 3x^4+ 6x^3- 12x^2+ 6x - 1}{(x - 1)(2x - 1)(x^2- 3x + 1)(x^2+x-1)}$\\\hline
47&$F_{\{2413,4132,1234\}}(x)=-\frac {3x^6- 21x^5+ 40x^4- 43x^3+ 26x^2- 8x + 1}{(2x - 1)(x - 1)^{4}(x^2-3x+1)}$\\\hline
48&$F_{\{4312,3124,1342\}}(x)=-\frac {x^9- 15x^8+ 73x^7 - 175x^6+ 247x^5- 228x^4+ 138x^3- 52x^2+ 11x -1}{(x^2-3x+1)^{2}(x - 1)^6} $\\\hline
51&$F_{\{4213,3124,1432\}}(x)=\frac{x^6- 7x^4+ 12x^3-13x^2+ 6x - 1}{(x^2- 3x + 1)(3x^3- 5x^2+4x - 1)}$\\\hline
52&$F_{\{1432,4123,1234\}}(x)=-\frac{x^8- 4x^7+ 3x^6+ 4x^5- 11x^4+ 20x^3- 18x^2+ 7x - 1}{(x-1)^2(x^2-3x+1)^2}$\\\hline
53&$F_{\{2134,4132,1243\}}(x)=\frac{x^{10}-4x^9-6x^8+68x^7-186x^6+291x^5-283x^4+170x^3-61x^2+12x-1}{(2x-1)^2(x^2-3x+1)^2(x-1)^3}$\\\hline
54&$F_{\{3124,1432,1234\}}(x)=\frac{(1-x)^3(2x^3-2x^2+3x-1)}{2x^9-7x^8+7x^7-10x^6+16x^5-27x^4+29x^3-19x^2+7x-1}$\\\hline
57&$F_{\{2143,1432,1234\}}(x)=\frac {x^7+ x^6- x^5+ 3x^3+ 2x^2+ 2x - 1}{x^7+ x^6- x^5- x^4+ 2x^3+x^2+3x-1}$\\\hline
58&$F_{\{4321,1243,1234\}}(x)=144x^9+ 396x^8+ 382x^7+ 202x^6+73x^5+ 21x^4+ 6x^3+ 2x^2+ x + 1$\\\hline
59&$F_{\{4321,1324,1234\}}(x)=334x^9+ 669x^8+ 484x^7+ 215x^6+73x^5+ 21x^4+ 6x^3+ 2x^2+ x + 1$\\\hline
60&$F_{\{4312,4132,1234\}}(x)=\frac{x^7+16x^6+12x^5+6x^4-4x^3+7x^2-4x+1}{(1-x)^5}$\\\hline
61&$F_{\{4312,1243,1234\}}(x)=\frac{x^{10} -4x^9+ 3x^8+2x^7+ x^6+ 4x^5- 21x^4+ 22x^3- 16x^{2}+ 6x - 1}{(x-1)^7}$\\\hline
62&$F_{\{4231,4312,1234\}}(x)=\frac{3x^8- 8x^7+ 4x^6 - 4x^5+ 21x^4- 22x^3+ 16x^2- 6x + 1}{(1-x)^7}$\\\hline
63&$F_{\{4312,1324,1234\}}(x)=\frac{x^{10}-5x^9+ 6x^8+2x^7- 5x^6+ 4x^5- 21x^4+ 22x^3- 16x^2 + 6x -1}{(x-1)^7}$\\\hline
64&$F_{\{4312,3412,1234\}}(x)=\frac {3x^7+ 5x^6- 4x^5 + 21x^4- 22x^3+ 16x^2- 6x + 1}{(1-x)^7}$\\\hline
65&$F_{\{4213,4132,1234\}}(x)=-\frac{3x^8+ 5x^7+ 13x^6+ 7x^5+ 2x^4+ x^3+ 5x^2- 4x+1}{(x^2+x-1)(x^3+x^2+x-1)(x-1)^3}$\\\hline
66&$F_{\{4231,4132,1234\}}(x)=\frac{2x^7+ 8x^6- 4x^5 + 21x^4- 22x^3+ 16x^2- 6x + 1}{(1-x)^7}$\\\hline
67&$F_{\{4312,1324,1243\}}(x)=\frac{2x^{10}-7x^8+65x^7-187x^6+274x^5-248x^4+145x^3-53x^2+11x-1}{(x-1)^6(2x-1)^3}$\\\hline
68&$F_{\{4312,1342,1243\}}(x)=\frac{3x^7-4x^6-14x^5+36x^4- 42x^3+ 26x^2- 8x + 1}{(x - 1)^3(2x-1)^3}$\\\hline
70&$F_{\{4312,3124,1243\}}(x)=-\frac{11x^7- 62x^6+ 128x^5- 146x^4+ 102x^3- 43x^2+ 10x -1}{(2x-1)^3(x-1)^5}$\\\hline
71&$F_{\{4231,1243,1234\}}(x)=-\frac{4x^8- 2x^7- 17x^6+ 25x^5- 43x^4+ 38x^3- 22x^2+ 7x-1}{(x-1)^8}$\\\hline
73&$F_{\{4231,1324,1234\}}(x)=-\frac{x^{10}-15x^8+55x^7-111x^6+ 149x^5- 141x^4+ 89x^3-37x^2+9x-1}{(x-1)^{10}}$\\\hline
79&$F_{\{2134,4132,1234\}}(x)=\frac{2x^{11}+x^{10}-10x^9-9x^8+12x^7+17x^6-30x^5+2x^4+28x^3-24x^2+8x-1}{(x^2+2x-1)(2x-1)(x-1)^3(x^2+x-1)^2}$\\\hline
81&$F_{\{2431,4312,1324\}}(x)=\frac{145x^3+11x-1-248x^4-193x^6+274x^5-53x^2-x^9-13x^8+80x^7)}{(2x-1)^3(x-1)^6}$\\\hline
82&$F_{\{4312,3142,1243\}}(x)=\frac{x^7+ 2x^6- 27x^5+59x^4- 61x^3+ 33x^2- 9x + 1}{(x - 1)^{2}(2x-1)^4}$\\\hline
83&$F_{\{4312,3412,1243\}}(x)=\frac{x^7+ 2x^6+ 4x^5-23x^4+ 36x^3- 25x^2+ 8x - 1}{(x - 1)(2x - 1)^4}$\\\hline
85&$F_{\{2314,4132,1432\}}(x)=-\frac{x^5+ 5x^4- 11x^3+13x^2- 6x + 1}{(2x - 1)(x^2- 3x + 1)(x - 1)^2}$\\\hline
87&$F_{\{4312,3124,1432\}}(x)=-\frac {2x^9- 46x^7+ 143x^6- 226x^5+ 221x^4- 137x^3+ 52x^2- 11x+1}{(x-1)^5(x^2- 3x + 1)(2x- 1)^2}$\\\hline
89&$F_{\{3142,4132,1234\}}(x)=-\frac {4x^5- 16x^4+ 24x^3- 19x^2+ 7x - 1}{(2x - 1)(x^2- 3x + 1)(x-1)^3}$\\\hline
91&$F_{\{4213,1342,1243\}}(x)=-\frac {4x^5- 14x^4+ 17x^3- 14x^2+ 6x - 1}{(3x - 1)(x^2- x + 1)(x - 1)^3}$\\\hline
92&$F_{\{2314,3124,1432\}}(x)=\frac {(x^3- 2x^2+ 3x - 1)(x^2+ x - 1)(1-x)^3}{x^9- 2x^8+ 6x^7- 4x^6- 7x^5+ 32x^4- 40x^3+25x^2- 8x + 1}$\\\hline
95&$F_{\{2314,4132,1342\}}(x)=-\frac{4x^6- 25x^5+ 51x^4-56x^3+ 32x^2- 9x + 1}{(2x - 1)(x - 1)^2(x^2- 3x +1)^2}$\\\hline
96&$F_{\{2134,4132,1342\}}(x)=-\frac{4x^8+ 6x^7- 45x^6+ 100x^5- 126x^4+ 95x^3- 42x^2+ 10x -1}{(2x - 1)^2(x^2- 3x + 1)(x -1)^4}$\\\hline
97&$F_{\{2341,4312,4123\}}(x)=-\frac{(x - 1)^4(x^3- 2x^2 + 3x - 1)}{x^8- 4x^7+ 18x^6- 35x^5+ 51x^4- 47x^3+ 26x^2- 8x +1}$\\\hline
98&$F_{\{2134,3124,1432\}}(x)=\frac{(x - 1)^{3}(x^3+ 2x - 1)}{4x^6- 7x^5+ 9x^4- 15x^3+ 13x^2- 6x+ 1}$\\\hline
100&$F_{\{4312,1342,4123\}}(x)=-\frac{4x^6- 16x^5+ 30x^4- 31x^3+ 20x^2- 7x + 1}{(x - 1)^3(2x^4-7x^3+ 8x^2- 5x +1)}$\\\hline
101&$F_{\{3124,4132,1342\}}(x)=-\frac{4x^6- 16x^5+ 30x^4- 31x^3+ 20x^2- 7x + 1}{(x - 1)^3(2x^4- 7x^3+ 8x^2- 5x +1)}$\\\hline
102&$F_{\{2413,3142,1234\}}(x)=-\frac {(x - 1)^{3}(x^3- 2x^2+ 3x - 1)}{x^7- 4x^6+ 12x^5- 23x^4+ 28x^3- 19x^2+ 7x -1}$\\\hline
104&$F_{\{2134,4132,1423\}}(x)=-\frac{3x^7- 4x^6+ 17x^5- 46x^4+ 55x^3- 32x^2+ 9x - 1}{(x - 1)(x^2- 3x + 1)(2x - 1)^3}$\\\hline
105&$F_{\{4213,2134,1342\}}(x)=-\frac{7x^6- 25x^5+ 51x^4- 56x^3+ 32x^2- 9x + 1}{(2x - 1)(x - 1)^2(x^2- 3x + 1)^2}$\\\hline
107&$F_{\{4213,3412,1342\}}(x)=-\frac{(x^2- x + 1)(2x - 1)^3}{(4x^3- 7x^2+ 5x - 1)(x - 1)^3}$\\\hline
110&$F_{\{2134,3142,1432\}}(x)=-\frac{(x - 1)(3x^3- 5x^2 + 4x - 1)^2}{x^9+ 2x^8- 27x^7+ 86x^6-144x^5+ 150x^4- 100x^3+ 42x^2- 10x + 1}$\\\hline
111&$F_{\{2143,3142,1234\}}(x)=\frac{(x^3- 2x^2+ 3x - 1)^2}{(2x^3- 3x^2+ 4x - 1)(x - 1)^3}$\\\hline
113&$F_{\{2134,1432,1234\}}(x)=\frac{2x^5- x^4- 3x^3-2x^2- 2x + 1}{2x^5- 2x^3- x^2- 3x + 1}$\\\hline
114&$F_{\{4312,1423,1234\}}(x)=-\frac{x^{10} - 3x^9+ 2x^8 + 4x^7- 9x^6+ 24x^5- 43x^4+ 38x^3- 22x^2+ 7x - 1}{(x - 1)^8}$\\\hline
115&$F_{\{4231,1423,1234\}}(x)=-\frac{2x^{10} - 17x^9+ 66x^8- 158x^7+ 256x^6- 289x^5+ 230x^4- 126x^3+46x^2- 10x + 1}{(x - 1)^{11}}$\\\hline
116&$F_{\{4312,4123,1243\}}(x)=\frac {x^9- 23x^8+ 133x^7-315x^6+ 419x^5- 350x^4+ 188x^3- 63x^2+12x - 1}{(2x - 1)^4(x - 1)^5}$\\\hline
117&$F_{\{3124,4132,1234\}}(x)=\frac{x^9- 6x^8+ 22x^7-53x^6+ 92x^5- 104x^4+ 76x^3- 35x^2+ 9x - 1}{(x^3- 2x^2+ 3x - 1)(2x - 1)(x - 1)^5}$\\\hline
119&$F_{\{4312,1432,1324\}}(x)=\frac{-350x^4-63x^2+419x^5-26x^8+138x^7-317x^6+188x^3+x^9-1+12x}{(2x-1)^4(x-1)^5}$\\\hline
120&$F_{\{4132,1423,1234\}}(x)=-\frac{x^8+ 4x^7- 41x^6+99x^5- 126x^4+ 95x^3- 42x^2+ 10x - 1}{(x^2- 3x + 1)(2x - 1)^2(x - 1)^4}$\\\hline
122&$F_{\{4213,1432,1324\}}(x)=-\frac{5{x}^{5}-19{x}^{4}+25{x}^{3}-19{x}^{2}+7x-1}{(x-1)({x}^{2}-3x+1)(3{x}^{3}-5{x}^{2}+4x-1)}$\\\hline
123&$F_{\{4132,1342,1234\}}(x)=-\frac{3x^8- 46x^7+ 141x^6- 225x^5+ 221x^4- 137x^3+ 52x^2- 11x + 1}{(x^2- 3x + 1)(2x - 1)^{2}(x - 1)^5}$\\\hline
124&$F_{\{2341,4132,4123\}}(x)=-\frac{(2x - 1)(x - 1)^4}{2x^6- 8x^5+ 19x^4- 27x^3+ 19x^2- 7x + 1}$\\\hline
128&$F_{\{2341,3142,4123\}}(x)=\frac{\left({x}^{2}-3x+1\right)\left(x-1\right)^{5}}{4{x}^{7}-23{x}^{6}+55{x}^{5}-78{x}^{4}+66{x}^{3}-33{x}^{2}+9x-1}$\\\hline
135&$F_{\{1432,4123,1243\}}(x)=-\frac{5x^5- 14x^4+ 22x^3-18x^2+ 7x - 1}{(x - 1)^4(2x^2- 4x + 1)}$\\\hline
136&$F_{\{4213,1342,4123\}}(x)=\frac{x^5- 3x^3+ 4x^2-4x + 1}{x^5+ x^4- 6x^3+ 7x^2- 5x + 1}$\\\hline
137&$F_{\{3124,1432,1342\}}(x)=-\frac{(x^2- 3x + 1)(x^2+2x - 1)}{(x - 1)(x^4- 2x^3- 5x^2+ 5x - 1)}$\\\hline
138&$F_{\{2134,3142,1243\}}(x)=-\frac {(x^2- 3x + 1)(x^2 + 2x - 1)}{(x - 1)(x^4- 2x^3- 5x^2+ 5x - 1)}$\\\hline
139&$F_{\{2143,3124,1342\}}(x)=-\frac {(x^2- 3x + 1)(x^2+ 2x - 1)}{(x - 1)(x^4- 2x^3- 5x^2+ 5x - 1)}$\\\hline
140&$F_{\{3124,1432,1243\}}(x)=\frac{(3x - 1)(x - 1)^3}{9x^4- 19x^3+ 17x^2- 7x + 1}$\\\hline
141&$F_{\{2143,1423,1234\}}(x)=\frac{2x^4- 4x^3+ 7x^2- 5x + 1}{4x^4- 9x^3+ 11x^2- 6x + 1}$\\\hline
142&$F_{\{1432,1342,4123\}}(x)=\frac{x^3+ 3x - 1}{x^3- 2x^2+ 4x - 1}$\\\hline
143&$F_{\{4312,4123,1234\}}(x)=-\frac{x^8- 3x^7- 12x^6+ 23x^5- 43x^4+ 38x^3- 22x^2+ 7x - 1}{(x-1)^8}$\\\hline
144&$F_{\{4231,4123,1234\}}(x)=-\frac{3x^8- 15x^7+ 40x^6- 66x^5+ 81x^4- 60x^3+ 29x^2- 8x + 1}{(x - 1)^9}$\\\hline
145&$F_{\{4312,1423,1243\}}(x)=\frac{2x^7- 2x^6- 25x^5+ 59x^4- 61x^3+ 33x^2- 9x + 1}{(x - 1)^2(2x - 1)^4}$\\\hline
146&$F_{\{4132,1243,1234\}}(x)=\frac{x^9- 4x^8+ 20x^6- 58x^5+ 83x^4- 69x^3+ 34x^2- 9x + 1}{(2x - 1)(x^2- 3x + 1)(x - 1)^5}$\\\hline
147&$F_{\{4132,1324,1234\}}(x)=-\frac{13x^{10} - 45x^9+ 83x^8- 38x^7- 141x^6+ 308x^5- 306x^4+ 178x^3- 62x^2+ 12x - 1}{(x^2- 3x + 1)(x^2+ x - 1)(2x - 1)^2(x - 1)^5}$\\\hline
148&$F_{\{2134,4132,1324\}}(x)=-\frac{5x^8- 51x^7+ 172x^6- 288x^5+ 283x^4- 170x^3+ 61x^2- 12x+1}{(2x - 1)^2(x^2- 3x + 1)^2(x - 1)^3}$\\\hline
152&$F_{\{4231,2341,4123\}}(x)=\frac{(x - 1)^{6}(x^2- 3x + 1)}{5x^8- 31x^7+ 83x^6- 134x^5+ 144x^4-99x^3+ 42x^2- 10x + 1}$\\\hline
154&$F_{\{4312,1342,1423\}}(x)=-\frac{3x^5- 14x^4+ 21x^3- 18x^2+ 7x - 1}{(x - 1)(2x^3- 4x^2+ 4x - 1)(x^2- 3x + 1)}$\\\hline
155&$F_{\{3124,4132,1243\}}(x)=-\frac{3x^5- 14x^4+ 21x^3- 18x^2+ 7x - 1}{(x - 1)(2x^3- 4x^2+ 4x - 1)(x^2- 3x + 1)} $\\\hline
160&$F_{\{4312,1432,1342\}}(x)=\frac{2x^5- 4x^4- 10x^3+ 16x^2- 7x + 1}{(x - 1)(3x - 1)(2x - 1)(x^2 + 2x - 1)}$\\\hline
161&$F_{\{4312,4132,1342\}}(x)=-\frac{7x^5- 22x^4+ 33x^3- 24x^2+ 8x - 1}{(x^3- 3x^2+ 4x - 1)(x - 1)(2x - 1)^2}$\\\hline
167&$F_{\{3142,3124,1432\}}(x)=-\frac {(2x - 1)(x - 1)(x^2- 3x + 1)}{x^5- 7x^4+ 18x^3- 17x^2+ 7x - 1}$\\\hline
168&$F_{\{3124,1432,1423\}}(x)=-\frac{(x^2- 3x + 1)(2x- 1)^{2}}{(x - 1)(x^4- 13x^3+ 16x^2- 7x + 1)}$\\\hline
169&$F_{\{3142,1423,1234\}}(x)=-\frac{(x^2- 3x + 1)(2x- 1)^2}{(x - 1)(x^4- 13x^3+ 16x^2- 7x + 1)}$\\\hline
179&$F_{\{2134,1432,1423\}}(x)=-\frac{2x^5- 8x^4+ 12x^3- 12x^2+ 6x - 1}{(x^4- 5x^3+ 10x^2- 6x+ 1)(x^2- x + 1)}$\\\hline
181&$F_{\{2143,1324,1234\}}(x)=-\frac {2x^3+ 3x - 1}{x^4- 2x^3+ 2x^2- 4x + 1}$\\\hline
183&$F_{\{4132,4123,1234\}}(x)=\frac {x^8- 8x^7+ 31x^{6}- 75x^5+ 98x^4- 75x^3+ 35x^2- 9x + 1}{(2x - 1)^2(x - 1)^6}$\\\hline
186&$F_{\{4132,4123,1243\}}(x)=-\frac {-27{x}^{5}+55{x}^{4}-57{x}^{3}+32{x}^{2}-9x+1+4{x}^{6}}{ \left( 3x-1 \right) \left( {x}^{2}-3x+1 \right)  \left( x-1 \right) ^{4}}$\\\hline
189&$F_{\{2143,2134,1432\}}(x)=-\frac{x^4- 7x^3+ 8x^2- 5x + 1}{x^5- 5x^4+ 13x^3- 12x^2+ 6x - 1}$\\\hline
200&$F_{\{2143,3124,1243\}}(x)=\frac {{x}^{3}-6{x}^{2}+5x-1}{ \left( x-1 \right)  \left( 5{x}^{2}-5x+1 \right) }$\\\hline
202&$F_{\{1432,1423,1234\}}(x)=\frac {x^4- 4x^3+ 10x^{2}- 6x + 1}{3x^4- 11x^3+ 15x^2- 7x + 1}$\\\hline
205&$F_{\{1432,1324,1234\}}(x)=-\frac{x^7- 2x^6+ 4x^5-17x^4+ 24x^3- 18x^2+ 7x - 1}{2x^6- 14x^5+ 34x^4- 38x^3+ 24x^2- 8x + 1}$\\\hline
206&$F_{\{1432,1243,1234\}}(x)=-\frac{x^6+ 5x^4- 12x^3+ 12x^2- 6x + 1}{2x^5- 13x^4+ 21x^3- 17x^2+ 7x - 1}$\\\hline
\end{longtable}}

\begin{thebibliography}{11pt}

\bibitem{AA20160600}
M. H. Albert, C. Homberger, J. Pantone, N. Shar and V. Vatter, Generating permutations with restricted containers, http://arxiv.org/abs/1510.00269.

\bibitem{AA20160607}
D. Callan and T. Mansour, On permutations avoiding 1324, 2143, and another 4-letter pattern, {\em Pure Math. Appl. (PU.M.A.)} {\bf 26:1} (2017), 1--10.

\bibitem{AA20160601}
D. Callan and T. Mansour, On permutations avoiding 1243, 2134, and another 4-letter pattern, {\em Pure Math. Appl. (PU.M.A.)} {\bf 26:1} (2017), 11--21.

\bibitem{AA20170501}
D. Callan and T. Mansour, Enumeration of small Wilf classes avoiding 1324 and two other $4$-letter patterns, {\em Pure Math. Appl. (PU.M.A.)}, to appear.

\bibitem{AA1342}
D. Callan and T. Mansour, Enumeration of small Wilf classes avoiding 1342 and two other $4$-letter patterns, submitted.

\bibitem{CMS3patI}
D. Callan, T. Mansour and M. Shattuck, Wilf classification of triples of 4-letter patterns I, {\em Discrete Math. Theor. Comput. Sci.} {\bf19:1} (2017), \#5.

\bibitem{CMS3patII}
D. Callan, T. Mansour and M. Shattuck, Wilf classification of triples of 4-letter patterns II, {\em Discrete Math. Theor. Comput. Sci.} {\bf19:1} (2017), \#6.

\bibitem{HYL} D. Callan, T. Mansour and M. Shattuck, Wilf classification of triples of 4-letter patterns,  http://arxiv.org/abs/1605.04969.

\bibitem{FM}
G. Firro and T. Mansour, Three-letter-pattern avoiding permutations and functional equations, {\em Electron. J. Combin.} {\bf13} (2006), \#R51.

\bibitem{HM}
Q. Hou and T. Mansour, Kernel method and linear recurrence system, {\em J. Computat. Appl. Math.} {\bf261:1} (2008), 227--242.

\bibitem{K} D. E. Knuth, \emph{The Art of Computer Programming}, 3rd edition, Addison Wesley, Reading, MA, 1997.

\bibitem{KS}
D. Kremer and W. C. Shiu, Finite transition matrices for permutations avoiding
pairs of length four patterns, {\em Discrete Math.} {\bf268:1-3} (2003), 171--183.

\bibitem{L}
I. Le, Wilf classes of pairs of permutations of length 4, {\em Electron. J. Combin.} {\bf12} (2005), \#R25.


\bibitem{SiS}
R. Simion and F. W. Schmidt, Restricted permutations, \emph{European J. Combin.} {\bf 6} (1985), 383--406.

\bibitem{Sl}
N. J. A. Sloane, The On-Line Encyclopedia of Integer Sequences, http://oeis.org.

\bibitem{St0}
Z. E. Stankova, Forbidden subsequences, \emph{Discrete Math.} {\bf132} (1994), 291--316.

\bibitem{St}
Z. Stankova, Classification of forbidden subsequences of length four, \emph{European J. Combin.} {\bf17} (1996), 501--517.

\bibitem{V}
V. Vatter, Finding regular insertion encodings for permutation
classes, \emph{J. Symbolic Comput.} {\bf47:3} (2012), 259--265.

\bibitem{W}J. West, Generating trees and the Catalan and Schr\"{o}der numbers, \emph{Discrete Math.} {\bf146} (1995), 247--262.

\bibitem{wikipermpatt} Wikipedia, Permutation pattern,
https://en.wikipedia.org/wiki/Permutation\_pattern

\bibitem{wikipermpatt2} Wikipedia, Enumerations of specific permutation classes,\newline
https://en.wikipedia.org/wiki/Enumerations\_of\_specific\_permutation\_classes

\end{thebibliography}
\end{document}